\RequirePackage{latexsym, amsmath, amstext, amssymb, amsfonts, amscd, bm, array, amsbsy, mathrsfs}
\documentclass[cmp, numbook, final, a4paper]{svjour}
\pdfoutput=1
\usepackage{colortbl}
\usepackage[usenames,dvipsnames]{xcolor}
\usepackage{array}
\usepackage{t1enc}
\usepackage[mathscr]{eucal}
\usepackage{indentfirst}
\usepackage{pb-diagram}
\usepackage{graphicx}
\usepackage{fancyhdr}
\usepackage{fancybox}
\usepackage{enumerate}
\usepackage{rotating}
\usepackage{color}
\usepackage[all]{xy}
\usepackage[colorlinks=true, linkcolor=black]{hyperref}
\usepackage{tikz}
\usepackage{xparse}
\usetikzlibrary{matrix}
\usepackage{float}
\colorlet{linkequation}{blue}

\newcommand*{\SavedEqref}{}
\let\SavedEqref\eqref
\renewcommand*{\eqref}[1]{%
  \begingroup
    \hypersetup{
      linkcolor=blue,
      linkbordercolor=blue,
    }%
    \SavedEqref{#1}%
  \endgroup
}

\DeclareSymbolFont{extraup}{U}{zavm}{m}{n}
\DeclareMathSymbol{\varheart}{\mathalpha}{extraup}{86}
\DeclareMathSymbol{\vardiamond}{\mathalpha}{extraup}{87}

\usepackage{fancyhdr}

\usepackage{url}
\usepackage[sort&compress,comma]{natbib}
\bibpunct{[}{]}{,}{n}{}{,}
\newtheorem{thm}{Theorem}[section]

\newtheorem{prop}[thm]{Proposition}

\newtheorem{cor}[thm]{Corollary}

\newtheorem{pdef}[thm]{Proposition-Definition}

\newcommand{\D}{{\check{D}^\ad}}

\newcommand{\im}{\mathrm{im}}

\newcommand{\End}{\mathrm{End}}
\newcommand{\Aut}{\mathrm{Aut}}

\newcommand{\Mod}{\mathrm{Mod}}

\newcommand{\Rep}{{\rm Rep}}

\newcommand{\Mat}{\mathrm{Mat}}
\newcommand{\ev}{\mathrm{ev}}
\newcommand{\eqdef}{\stackrel{{\rm def.}}{=}}

\DeclareFontFamily{U}{rsf}{}
\DeclareFontShape{U}{rsf}{m}{n}{<5> <6> rsfs5 <7> <8> <9> rsfs7 <10-> rsfs10}{}
\DeclareMathAlphabet\Scr{U}{rsf}{m}{n}

\def\N{\mathbb{N}}
\def\Z{\mathbb{Z}}
\def\C{\mathbb{C}}
\def\R{\mathbb{R}}

\def\S{\mathbb{S}}

\def\H{\mathbb{H}}

\def\S{\mathbb{S}}

\def\rk{{\rm rk}}

\def\GL{\mathrm{GL}}

\def\Ad{\mathrm{Ad}}

\def\pr{\mathrm{pr}}

\def\fp{\frak{p}}
\def\Set{\mathrm{Set}}

\newcommand{\be}{\begin{equation*}}
\newcommand{\ee}{\end{equation*}}
\newcommand{\ben}{\begin{equation}}
\newcommand{\een}{\end{equation}}
\newcommand{\beqa}{\begin{eqnarray*}}
\newcommand{\eeqa}{\end{eqnarray*}}
\newcommand{\beqan}{\begin{eqnarray}}
\newcommand{\eeqan}{\end{eqnarray}}

\newcommand{\nn}{\nonumber}

\newcommand{\twopartdef}[4]
{
	\left\{
		\begin{array}{ll}
			#1 & \mbox{if } #2 \\
			#3 & \mbox{if } #4
		\end{array}
	\right .
}

\newcommand{\twopartdefmod}[4]
{
	\left\{
		\begin{array}{ll}
			#1 & \mbox{~} #2 \\
			#3 & \mbox{~} #4
		\end{array}
	\right .
}

\newcommand{\threepartdef}[6]
{
	\left\{
		\begin{array}{ll}
			#1 & \mbox{~} #2 \\
			#3 & \mbox{~} #4 \\
                        #5 & \mbox{~} #6
		\end{array}
	\right .
}

\newcommand{\Res}{{\mathrm Res}}
\newcommand{\id}{\mathrm{id}}

\newcommand{\tr}{\mathrm{tr}}

\newcommand{\sign}{\mathrm{sign}}

\def\cC{{\mathcal C}}

\def\cB{\Scr B}

\def\Cl{\mathrm{Cl}}

\def\odd{\mathrm{odd}}
\def\Spin{\mathrm{Spin}}

\def\Pin{\mathrm{Pin}}
\def\Spin{\mathrm{Spin}}

\def\SO{\mathrm{SO}}
\def\O{\mathrm{O}}
\def\U{\mathrm{U}}

\def\cN{\mathcal{N}}
\def\cG{\mathcal{G}}

\def\cF{\mathcal{F}}
\def\cC{\mathcal{C}}
\def\SU{\mathrm{SU}}
\def\Sp{\mathrm{Sp}}
\def\G_2{\mathrm{G_2}}
\def\mG{\mathbb{G}}

\def\cL{\mathcal{L}}
\def\cS{\mathcal{S}}

\def\P{\mathbb{P}}

\newcommand{\Hom}{{\rm Hom}}
\newcommand{\Isom}{{\rm Isom}}
\def\Quad{\mathrm{Quad}}
\def\Res{\mathrm{Res}}
\def\Aut{\mathrm{Aut}}
\def\Ob{\mathrm{Ob}}
\def\ClB{\mathrm{ClB}}
\def\Alg{\mathrm{Alg}}
\def\TwMod{\mathrm{TwMod}}
\def\Gp{\mathrm{Gp}}
\def\tw{\mathrm{tw}}
\def\ClRep{\mathrm{ClRep}}

\def\tAd{\widetilde{\Ad}}

\def\op{\mathrm{op}}

\def\Im{\mathrm{Im}}

\def\tdelta{\tilde{\delta}}
\def\tN{\tilde{N}}
\def\ttau{\tilde{\tau}}

\def\G{\mathrm{G}}
\def\L{\mathrm{L}}
\def\D{\mathbb{D}}
\def\T{\mathbb{T}}
\def\R{\mathbb{R}}
\def\hom{\mathrm{hom}}

\def\w{\mathrm{w}}

\def\trho{\tilde{\rho}}
\def\tlambda{\tilde{\lambda}}
\def \BO{\mathrm{BO}}

\def\sq{\mathrm{sq}}
\def\bc{\mathbf{c}}
\def\bi{\mathbf{i}}
\def\bq{\mathbf{q}}
\def\bs{\mathbf{s}}
\def\fG{\mathfrak{G}}
\def\rM{\mathrm{M}}

\begin{document}

\title{Real pinor bundles and real Lipschitz structures}
  
\author{C. I. Lazaroiu\inst{1} \and C. S. Shahbazi\inst{2}}
\institute{Center for Geometry and Physics, Institute for Basic Science, Pohang 790-784, Republic of Korea, \email{calin@ibs.re.kr} \and 
Institut de Physique Th\'eorique, CEA-Saclay, France, \email{carlos.shabazi-alonso@cea.fr}.}

\date{\today}

\maketitle

\abstract{Let $(M,g)$ be a pseudo-Riemannian manifold of arbitrary dimension and signature. We prove that there exist mutually quasi-inverse equivalences between the groupoid of weakly faithful real pinor bundles on $(M,g)$ and the groupoid of weakly faithful real Lipschitz structures on $(M,g)$, from which follows that every bundle of weakly faithful real Clifford modules is associated to a real Lipschitz structure. The latter gives a generalization of spin structures based on certain groups which we call real Lipschitz groups. In the irreducible case, we classify real Lipschitz groups in all dimensions and signatures. Using this classification and the previous correspondence we obtain the topological obstruction to existence of a bundle of irreducible real Clifford modules over a pseudo-Riemannian manifold $(M,g)$ of arbitrary dimension and signature. As a direct application of the previous results, we show that the supersymmetry generator of eleven-dimensional supergravity in ``mostly plus'' signature can be interpreted as a global section of a bundle of irreducible Clifford modules if and {\em only if} the underlying eleven-manifold is orientable and spin.}

\setcounter{tocdepth}{1} 
\tableofcontents

\section*{Introduction}

The classical approach to spin geometry \cite{ABS, Spingeometry,
  Friedrich, BGV} assumes existence of a spin structure $Q$ on a
pseudo-Riemannian manifold $(M,g)$. Given a linear representation of
the corresponding spin group, $Q$ induces a spinor bundle $S$ through
the associated vector bundle construction. One then shows that $S$
carries a globally-defined ``internal'' Clifford multiplication
$TM\otimes S\rightarrow S$ , which turns $S$ into a bundle of modules
over the Clifford bundle $\Cl(M,g)$. This allows one to lift a metric
connection on $(M,g)$ to a connection on $S$ and to define a
corresponding Dirac operator. This construction generalizes to
$\Spin^{c}$ and $\Spin^q$ structures \cite{Nagase}, which can also be
used to construct ``spinor bundles'' with globally-defined internal
Clifford multiplication. Thus, existence of such a spinorial structure
implies existence of a bundle of Clifford modules over $(M,g)$, but
the converse is not always true.\footnote{A well-known variant
  \cite{ABS, KirbyTaylor} starts with a $\Pin$ structure, leading to a
  bundle $S$ which need not admit an ``internal'' Clifford
  multiplication, but rather a map that takes $TM\otimes S$ into a
  bundle $S'$ which (depending on dimension and signature) may be
  non-isomorphic with $S$. In that case, one can sometimes define a
  ``modified'' version of the Dirac operator (see, for example,
  \cite{TrautmanPin1,TrautmanPin2,TrautmanPin3}). We stress that in
  this paper we are interested only in vector bundles $S$ which admit
  an {\em internal} Clifford multiplication $TM\otimes S\rightarrow
  S$. A brief discussion of external and internal Clifford
  multiplications can be found in Appendix \ref{app:cliffmul}.}

For various applications, it is important to develop spin geometry
starting from the assumption that $(M,g)$ admits a bundle of Clifford
modules $S$, without first choosing by hand a particular spinor structure 
to which $S$ is associated. The condition that $S$ be a
bundle of Clifford modules is equivalent to the requirement that $S$
be endowed with a globally-defined ``internal'' Clifford
multiplication $TM\otimes S\rightarrow S$. In general, this condition
is weaker than existence of a spin structure, as illustrated by the
theory of $\Spin^c$ and $\Spin^q$ structures. However, a
systematic study of the necessary and sufficient conditions under
which a bundle $S$ of Clifford modules exists on $(M,g)$ does not
appear to have been carried out for the case of real Clifford
representations. The purpose of the present paper is to perform such a
study.

Our approach relies on an equivalence of categories (which we
establish in Section 7) between the groupoid of bundles of real
Clifford modules obeying a certain ``weak faithfulness'' condition and
the groupoid of so-called {\em real Lipschitz structures}, a notion
which generalizes that of complex Lipschitz structure which was
introduced in \cite{FriedrichTrautman}. This result allows us to
extract the {\em necessary and sufficient} conditions under which
$(M,g)$ admits a bundle of real Clifford modules with given fiberwise
representation type.

A real Lipschitz structure is a generalization of a spin structure
where the spin group is replaced by a so-called {\em real Lipschitz
  group}. The latter is the group of all implementers of
pseudo-orthogonal transformations through operators acting in the
representation space of a weakly faithful real Clifford representation
$\gamma$ and arises naturally as the automorphism group of $\gamma$ in
a certain category of real Clifford representations and unbased
morphisms. The character of the Lipschitz group depends on $\gamma$
and on the signature $(p,q)$ of $g$. When $\gamma$ is irreducible,
then it is automatically weakly faithful and the character of its
Lipschitz group depends on the $\mod 8$ reduction of $p-q$. We
identify all such {\em elementary Lipschitz groups} as well as certain
homotopy-equivalent reduced forms thereof, the latter being summarized
in Table \ref{table:introduction}. The $\Spin^{o}(V,h)$ structures
arising when $p-q\equiv_8 3, 7$ (and whose real Lipschitz groups are
discussed in Subsection \ref{sec:spino}) appear to be new and are
studied in detail in the companion paper \cite{spino}.

\begin{table}[H]
	\centering
	\begin{tabular}{|c|c|c|c|c|c|c|c|}
		\hline
		$\begin{array}{c} p-q\\ {\rm mod}~8 \end{array}$ & Reduced Lipschitz group  \\
		\hline\hline
		$0,2$ & $\Pin(V,h)$ \\
		\hline
		$3,7$ & $\Spin^o(V,h)$   \\
		\hline
		$4,6$ & $\Pin^q(V,h)$ \\
		\hline
		$1$ & $\Spin(V,h)$   \\
		\hline
		$5$ & $\Spin^q(V,h)$    \\
		\hline 
	\end{tabular}
	\vskip 0.2in
	\caption{The reduced elementary Lipschitz group of a pseudo-Riemannian manifold $(M,g)$ of dimension $d=p+q$ and signature $(p,q)$.}
	\label{table:introduction}
\end{table}

We also study certain representations of Lipschitz groups, which turn
out to play an important role for further developments of the
theory. Using these results, we extract the topological obstructions
to existence of bundles of irreducible real Clifford modules on
$(M,g)$ in any dimension and signature. This allow us to identify the
minimal requirements for developing a version of real spin geometry
based on such bundles.

Our study is motivated, in particular, by physical theories such as
supergravity and string theory, where it is important to understand
the weakest assumptions under which certain models can be defined
globally. As we shall show in later papers, this leads to new
questions and problems which do not appear to have been systematically
considered before and which are of mathematical and physical interest.

The results of this paper lead to various questions which may be of
interest for further study. For example, one can ask what
modifications may arise in the index theorem for Dirac operators
defined on pinor bundles which are associated to real Lipschitz
structures in various dimensions and signatures. Our results afford a
systematic study of Killing spinors and generalized Killing spinors on
the most general pseudo-Riemannian manifolds admitting bundles of
irreducible real Clifford modules. For example, they could be used to
extend Wang's results \cite{Wangparallel} from spin manifolds to
manifolds of arbitrary signature which admit bundles of irreducible
real Clifford modules. Killing spinors \cite{Bar} and generalized
Killing spinors \cite{2013arXiv1303.6179M} were studied in the
literature on manifolds admitting $\Spin$, $\Spin^{c}$ and $\Spin^{q}$
structures \cite{MoroianuSpinc,Herrera}; more general spinorial
structures (whose associated vector bundles need not be bundles of
irreducible Clifford modules) are studied in references
\cite{Herrera1, Herrera2}. However, Killing spinors on $\Spin^{o}$
  manifolds do not seem to have been studied systematically. In addition, and especially for applications to supergravity and string theory, it would be very interesting to perform an analysis similar to the one presented in this work for the case of bundles of irreducible as well as faithful Clifford modules over the \emph{even} bundle of Clifford algebras of the underlying manifold. Work in this direction is already in progress. 

The paper is organized as follows. Section \ref{sec:realclifford}
summarizes some facts on real Clifford algebras and associated groups,
with the purpose of fixing terminology and notation; it also proves
some results which will be needed later on, some of which are not
readily available in the literature. Section \ref{sec:enlargedspinor}
summarizes certain enlargements of the Spin group which will turn out
to provide models for the reduced Lipschitz group of irreducible real
Clifford representations in various dimensions and signatures. The
same section considers certain representations of these
groups. Section \ref{sec:Cliffordrep} considers a certain category of
real Clifford representations and unbased morphisms and discusses
certain subspaces associated to such representations as well as the
notion of weak faithfulness. Section \ref{sec:elementaryrep} defines
real Lipschitz groups and discusses some properties of their
elementary representations in a general setting. Section
\ref{sec:irreps} considers the case of irreducible real Clifford
representations, which always turn out to be weakly faithful. In that
section, we classify the Lipschitz groups of such representations and establish
isomorphisms between their reduced versions and various enlarged
spinor groups introduced in Section \ref{sec:enlargedspinor}. We also
describe the elementary representations of such Lipschitz groups and
connect them to those of the enlarged spinor groups. Section
\ref{sec:realpinors} discusses bundles of weakly faithful real
Clifford modules as well as real Lipschitz structures, establishing a general
equivalence between the corresponding groupoids. Section
\ref{sec:structures} discusses certain enlarged spinorial structures
which are relevant later on. Section \ref{sec:elementary2} considers
the case of bundles of irreducible real Clifford modules and the
corresponding Lipschitz structures, which we call {\em
  elementary}. Using the results of the previous sections, we
determines the topological obstructions to existence of such bundles
in every dimension and signature. Section \ref{sec:Mtheory} discusses
a direct application of our results to the global formulation of M-theory
on an eleven-dimensional Lorentzian manifold. Section
\ref{sec:relation} outlines the relation of our work with certain
results in the literature. The appendices contain technical material.

\subsection{Notations, conventions and terminology} 

In this paper, a quadratic vector space means a pair $(V,h)$, where
$V$ is a finite-dimensional $\R$-vector space and $h:V\times
V\rightarrow \R$ is a non-degenerate symmetric bilinear form. 
Throughout the paper, we assume $V\neq 0$. The
Clifford algebra $\Cl(V,h)$ is considered only over $\R$.  
We use the \emph{plus} convention for Clifford algebras, so $\Cl(V,h)$
is the unital associative algebra generated by $V$ over $\R$ with the relations:
\be
v^2 = h(v,v)~~\forall v \in V~~.
\ee
A {\em Clifford representation} is a finite-dimensional unital
representation $\gamma:\Cl(V,h)\rightarrow \End_\R(S)$ through
endomorphisms of a {\em real} finite-dimensional vector space $S$ ---
we never use the complexification of $\Cl(V,h)$ or the
complexification of $S$. An irreducible Clifford representation is
always assumed to be realized in a space $S$ of positive dimension
(i.e. $S\neq 0$). 

Given two groups $A, B$ such that $A\times B$ contains a given central
$\Z_2$ subgroup $C$, we use the notation $A\cdot B$ for the quotient
$A\times B/C$.  Let $\mG_m\simeq \Z_m$ denote the group of complex
roots of unity of order $m\in \N_{>0}$ and $D_4\simeq \Z_2\times \Z_2$
denote the dihedral group of order $4$.

For $\S\in \{\R,\C,\H\}$ any of the three finite-dimensional
associative division algebras over $\R$, let $|~|\colon\S\rightarrow
\R_+$ denote the canonical norm and $\U(\S)$ denote the group of unit
norm elements:
\be
\U(\S)=\threepartdef{\mG_2}{\S=\R}{\U(1)}{\S=\C}{\Sp(1)}{\S=\H}~~.
\ee
For the algebras $\R, \C$ and $\H$, we define $\rM:\S\rightarrow \R_+$ through
$\rM(s)\eqdef |s|^2$. For the algebra $\D$ of hyperbolic (a.k.a. split
complex) numbers, let $\rM:\D\rightarrow \R$ denote the hyperbolic
modulus and $\U(\D)$ denote the group of unit hyperbolic numbers (see
Appendix \ref{app:hyp}).

Let $\Alg$ denote the category of finite-dimensional associative and
unital $\R$-algebras, $\Gp$ denote the category of groups and $\Set$
denote the category of sets. For any ring $R$, let $R^\times$ denote
its group of invertible elements. For any category $\cC$, let
$\cC^\times$ denote its unit groupoid (the groupoid obtained from
$\cC$ by keeping as morphisms only the isomorphisms of $\cC$). The
symbol $\simeq_{\cC}$ indicates existence of an isomorphism between
two objects of $\cC$. For notational uniformity, we define:
\be
{\hat \O}(V,h)\eqdef\twopartdef{\O(V,h)}{d=\ev}{\SO(V,h)}{d=\odd}~~
\ee
for any finite-dimensional quadratic vector space $(V,h)$ over $\R$. 

All manifolds $M$ considered in the paper are connected, Hausdorff and
paracompact (hence also second countable). All fiber bundles
considered are smooth. We assume $\dim M>0$ throughout.

\section{Real Clifford algebras and extended Clifford groups}
\label{sec:realclifford}

This section summarizes various facts about Clifford algebras and
associated spinor groups, with the purpose of fixing notations and
terminology; some statements which are well-known or easily
established are given without proof. We draw the reader's attention to
our discussion of extended Clifford groups, extended Clifford norms
and extended Pin groups as well as to the delicate distinction between
twisted and untwisted vector representations. The extended (sometimes
called ``untwisted'') Clifford group $\G^e(V,h)$ consists of those
elements of the multiplicative subgroup of $\Cl(V,h)$ whose {\em
  untwisted} adjoint action fixes $V$. It differs from the ordinary
(or ``twisted'') Clifford group $\G(V,h)$ by elements belonging to the
center of $\Cl(V,h)$. Unlike the ordinary Clifford group, the extended
Clifford group need not be $\Z_2$-graded, since it may contain
inhomogeneous elements from the center of the Clifford algebra. The
extended Clifford group and its untwisted vector representation
(adjoint representation restricted to $V$) are classical, but have
seen decreased use after the work of \cite{ABS}, which promoted the
use of the ordinary Clifford group and of its twisted vector
representation --- none of which are, however, natural for our
purpose. In this paper, the extended Clifford group and its vector
representation will play a central role since, as we shall see in
later sections, these objects relate most directly to the Lipschitz
group of irreducible Clifford representations and to the natural
representation of that Lipschitz group on $V$. The distinction between
the twisted and untwisted vector representations of pin groups will
also be essential when discussing topological obstructions to
existence of elementary real Lipschitz structures. As is well-known,
that distinction also plays an important role in the theory of Pin
structures and associated Dirac operators
\cite{TrautmanPin1,TrautmanPin2,TrautmanPin3, KirbyTaylor}.

\subsection{The category of real quadratic spaces}

A real quadratic space is a pair $(V,h)$, where $V\neq 0$ is a
finite-dimensional $\R$-vector space and $h:V\times V\rightarrow \R$
is a non-degenerate symmetric bilinear pairing on $V$. A morphism of
quadratic spaces from $(V,h)$ to $(V',h')$ (also known as an isometry)
is an $\R$-linear map $\varphi:V\rightarrow V'$ such that
$h'(\varphi(v_1),\varphi(v_2))=h(v_1,v_2)$ for all $v_1,v_2\in
V$. Quadratic spaces over $\R$ and their morphisms form a category
$\Quad$ whose unit groupoid $\Quad^{\times}$ we call the {\em groupoid
  of real quadratic spaces}; its objects coincide with those of
$\Quad$ while its morphisms are the invertible isometries. When $h$
is positive-definite, an isometry
$\varphi:(V,h)\rightarrow (V',h')$ is necessarily injective.

\subsection{Non-degenerate vectors and reflections}

Let $(V,h)$ be a real quadratic space. 

\begin{definition}
The {\em signature} $\epsilon(v)\in \{-1,0,1\}$ of a vector $v\in V$
is the signature of the real number $h(v,v)$. The vector $v$ is called
{\em non-degenerate} if $h(v,v)\neq 0$. It is called a {\em unit
  vector} if $|h(v,v)|=1$.
\end{definition}

\

\noindent The {\em reflection} determined by a non-degenerate vector
$v\in V$ is the linear map $R_v\in \O(V,h)$ given by:
\ben
\label{reflection}
R_v(x)\eqdef x-2 \frac{h(x,v)}{h(v,v)} v = x-2 \epsilon(v) \frac{h(x,v)}{|h(v,v)|} v ~~,
\een
which describes the $h$-orthogonal reflection of $V$ with respect to
the hyperplane $v^\perp=\{y\in V|h(y,v)=0\}\subset V$ orthogonal to $v$.
We have $R_v=R_{\lambda v}$ for any $\lambda \in \R^\times$.

\subsection{The category of real Clifford algebras}

The Clifford algebra construction gives a functor
$\Cl:\Quad\rightarrow \Alg$, where $\Alg$ denotes the category of
unital associative $\R$-algebras and unital algebra morphisms. For
each object $(V,h)$ of $\Quad$, $\Cl(V,h)$ is the Clifford algebra of
the quadratic space $(V,h)$ while for each isometry
$\varphi:(V,h)\rightarrow (V',h')$, $\Cl(\varphi):\Cl(V,h)\rightarrow
\Cl(V',h')$ denotes the unique unital morphism of algebras which
satisfies the condition $\Cl(\varphi)|_V=\varphi$. The image of the
functor $\Cl$ is a {\em non-full} sub-category of $\Alg$ which we
denote by $Cl$ and whose unit groupoid we denote by $Cl^\times$. 
Namely:

\begin{definition}
A {\em morphism of Clifford algebras} is a morphism
$\alpha:\Cl(V,h)\rightarrow \Cl(V',h')$ in the category $Cl$, i.e. a
morphism of unital algebras which satisfies $\alpha(V)\subset V'$ and
hence is necessarily of the form $\alpha=\Cl(\varphi)$ for a
(uniquely-determined) isometry $\varphi:(V,h)\rightarrow (V',h')$,
given by $\varphi\eqdef \alpha|_V$.
\end{definition}

By definition, isomorphisms of Clifford algebras are the morphisms of
$Cl^\times$. These are those isomorphisms of unital algebras
$\alpha:\Cl(V,h)\rightarrow \Cl(V',h')$ which satisfy $\alpha(V)=V'$
and hence are of the form $\alpha=\Cl(\varphi)$ for a
uniquely-determined invertible isometry $\varphi:(V,h)\rightarrow
(V',h')$ (given by $\varphi=\alpha|_V$). The corestriction of the
functor $\Cl$ to its image gives an isomorphism of categories
$\Quad\simeq Cl$, which in turn restricts to an isomorphism
$\Quad^\times\simeq Cl^\times$. Note that two Clifford algebras may be
isomorphic as unital associative algebras {\em without} being
isomorphic as {\em Clifford} algebras (i.e. without being isomorphic
in the category $Cl\simeq \Quad$).

The category $\Quad$ admits a skeleton whose objects are the {\em
  standard quadratic spaces} $\R^{p,q}\eqdef (\R^{p+q},h_{p,q})$,
where $h_{p,q}:\R^{p+q}\times \R^{p+q}\rightarrow \R$ is the {\em
  standard symmetric bilinear form of signature $(p,q)$}:
\be
h_{p,q}(x,y)\eqdef \sum_{i=1}^p x_i y_i-\sum_{j=p+1}^{p+q}x_j y_j~~\forall x,y\in \R^{p+q}~~.
\ee
The objects of this skeleton form a countable set indexed by the pairs
$(p,q)\in \N \times \N$. Accordingly, the category $Cl$ admits a
skeleton whose objects are the {\em standard real Clifford algebras}
$\Cl_{p,q}\eqdef \Cl(\R^{p,q})$.

\subsection{Parity, reversion and twisted reversion}
The Clifford algebra $\Cl(V,h)$ admits three canonical involutive
(anti-)automorphisms:
\begin{enumerate}[1.]
\itemsep 0.0em
\item The {\em parity involution} of $\Cl(V,h)$ is the unique unital
  $\R$-algebra automorphism $\pi\in \Aut_\Alg(\Cl(V,h))$ such that
  $\pi(v)=-v$ for all $v\in V$.
\item The {\em reversion} is the unique unital anti-automorphism
  $\tau$ of $\Cl(V,h)$ such that $\tau(v)=v$ for all $v\in V$.
\item The {\em twisted reversion} is the unique unital
  anti-automorphism $\ttau$ of $\Cl(V,h)$ such that $\ttau(v)=-v$. 
\end{enumerate}
We have: 
\be
\pi^2=\tau^2=\ttau^2=\id_{\Cl(V,h)}~~,~~\ttau=\tau\circ \pi=\pi\circ \tau
\ee
and the group $\{\id_{\Cl(V,h)}, \pi,\tau,\ttau\}$ is isomorphic with
$D_4$.

\begin{remark}
The underlying vector space of $\Cl(V,h)$ is canonically isomorphic
with $\wedge V$ through the Chevalley-Riesz-Crumeyrolle isomorphism
\cite{Riesz} (which depends on $h$). This gives a decomposition
$\Cl(V,h)=\oplus_{k=0}^d\Cl_k(V,h)$, where $\Cl_k(V,h)$ is the
subspace corresponding to $\wedge^k V$ (this decomposition does {\em
  not} give a $\Z$-grading of the associative algebra $\Cl(V,h)$). We
have:
\be
\pi|_{\Cl_k(V,h)}=(-1)^k\id_{\Cl_k(V,h)}~,~\tau|_{\Cl_k(V,h)}=(-1)^{\frac{k(k-1)}{2}}\id_{\Cl_k(V,h)}~~.
\ee
\end{remark} 

\subsection{The canonical $\Z_2$-grading}

The algebra $\Cl(V,h)$ admits the
{\em canonical $\Z_2$-grading}:
\be
\Cl(V,h)=\Cl_+(V,h)\oplus \Cl_-(V,h)~~,~~\Cl_\pm(V,h)\eqdef \ker(\pi\mp \id_{\Cl(V,h)})~~.
\ee
Namely, $\Cl_+(V,h)$ is the subalgebra generated by all Clifford
monomials $v_1\ldots v_k$ ($v_j\in V$) with even $k$ and $\Cl_-(V,h)$
is the subspace generated by all Clifford monomials with odd $k$.

\subsection{The group of Clifford units and its adjoint and twisted adjoint actions}

The group of Clifford units is the group $\Cl(V,h)^\times$ formed by
all invertible elements of the algebra $\Cl(V,h)$. Let
$\Cl^\times_\pm(V,h)\eqdef \Cl(V,h)^\times \cap \Cl_\pm(V,h)$.  Then:
\be
\Cl^\times_{\hom}(V,h)\eqdef \Cl^\times_+(V,h)\sqcup \Cl_-^\times(V,h)
\ee
is a subgroup of $\Cl(V,h)^\times$ called the {\em group of
  homogeneous units}. This group is $\Z_2$-graded by the disjoint
union decomposition given above. 

The {\em adjoint action} of $\Cl(V,h)^\times$ is the group morphism $\Ad^\Cl:\Cl(V,h)^\times \rightarrow
\Aut_\Alg(\Cl(V,h))$ given by:
\be
\Ad^\Cl(a)(x)\eqdef axa^{-1}~~\forall a\in \Cl(V,h)^\times~~\forall x\in \Cl(V,h)~~,
\ee
which gives a representation of $\Cl(V,h)^\times$ through unital $\R$-algebra automorphisms of $\Cl(V,h)$. 

The {\em twisted adjoint action} of $\Cl(V,h)^\times$ is the group 
morphism $\tAd^\Cl:\Cl(V,h)^\times \rightarrow \Aut_\R(\Cl(V,h))$ given by: 
\be
\tAd^\Cl(a)(x)\eqdef \pi(a)xa^{-1}~~\forall a\in \Cl(V,h)^\times~~\forall x\in \Cl(V,h)~~.
\ee
This gives a representation of the group of units through automorphisms of 
the underlying vector space of $\Cl(V,h)$, which (unlike the adjoint action) 
need not be $\R$-algebra automorphisms. We have:
\be
\tAd^\Cl|_{\Cl_\pm^\times(V,h)}=\pm \Ad^\Cl|_{\Cl_\pm^\times(V,h)}~~.
\ee

\noindent Notice that $\pi(\Cl_\pm^\times(V,h))= \Cl_\pm^\times(V,h)$ as well as the relations: 
\be
\Ad^\Cl\circ \pi|_{\Cl_\hom^\times(V,h)}=\Ad^\Cl|_{\Cl_\hom^\times(V,h)}~~,~~\tAd^\Cl\circ \pi|_{\Cl_\hom^\times(V,h)}=\tAd^\Cl|_{\Cl_\hom^\times(V,h)}~~.
\ee

\subsection{Clifford volume elements}
\label{sec:CliffordVolume} 
Any orientation of $V$ determines a {\em Clifford volume element}
$\nu=e_1\ldots e_d\in \Cl(V,h)^\times$, where $(e_1,\ldots e_d)$ is any
oriented orthonormal basis of $(V,h)$. This element is independent of
the choice of oriented orthonormal basis; it depends only on $h$ and
on the chosen orientation of $V$.  Moreover, the Clifford volume
element determined by the opposite orientation of $V$ equals $-\nu$.
The Clifford volume element has the following properties, which will
be used intensively later on:
\beqan
\label{nurel}
&&\nu^2=\sigma_{p,q}\eqdef (-1)^{q+\left[\frac{d}{2}\right]}=\twopartdef{(-1)^{\frac{p-q}{2}}}{d=\mathrm{even}}{(-1)^{\frac{p-q-1}{2}}}{d=\mathrm{odd}}=\twopartdef{+1}{p-q\equiv_4 0,1}{-1}{p-q\equiv_4 2, 3}\\
&&\tau(\nu)=(-1)^{\frac{d(d-1)}{2}}\nu=(-1)^{\left[\frac{d}{2}\right]}\nu=\twopartdef{+\nu}{d\equiv_4 0,1}{-\nu}{d\equiv_4 2,3}=\twopartdef{(-1)^{\frac{p+q}{2}}\nu}{d=\mathrm{even}}{(-1)^{\frac{p+q-1}{2}}\nu}{d=\mathrm{odd}}\nn\\
&&\tau(\nu)=_{d=\mathrm{odd}}\twopartdef{-(-1)^q\nu}{p-q\equiv_8 3,7}{+(-1)^q\nu}{p-q\equiv_8 1,5}~~,~~\nu^2=_{d=\mathrm{odd}}\twopartdef{-1}{p-q\equiv_8 3,7}{+1}{p-q\equiv_8 1,5}\nn~~.
\eeqan
Notice that $\tau(\nu)=(-1)^q\nu^{-1}$. 

\subsection{The volume grading}
\label{subsec:ClVolGrading}

We have:
\ben
\label{AdRel}
\Ad^\Cl(\nu)=\pi^{d-1}=\twopartdef{\id_{\Cl(V,h)}}{d=\mathrm{odd}}{\pi}{d=\mathrm{even}}~~.
\een
In particular, $\pi$ is an inner automorphism of $\Cl(V,h)$ when $d$ is even. 
The involutive $\R$-algebra automorphism $\Ad^\Cl(\nu)\in \Aut_\Alg(\Cl(V,h))$ 
induces a $\Z_2$-grading called the {\em volume grading} of $\Cl(V,h)$:
\begin{eqnarray}
\Cl^0(V,h) &\eqdef & \{x\in \Cl(V,h)|\Ad^\Cl(\nu)(x)=+x\}\, , \nonumber\\ 
\Cl^1(V,h) &\eqdef & \{x\in \Cl(V,h)|\Ad^\Cl(\nu)(x)=-x\}\, .
\end{eqnarray}
Relation \eqref{AdRel} implies: 
\begin{enumerate}[1.]
\itemsep 0.0em
\item For odd $d$, we have $\Cl^0(V,h)=\Cl(V,h)$ and $\Cl^1(V,h)=0$ and hence the volume grading is concentrated in degree zero. 
\item For even $d$, we have $\Cl^0(V,h)=\Cl_+(V,h)$ and $\Cl^1(V,h)=\Cl_-(V,h)$ and hence the volume grading coincides with the canonical $\Z_2$-grading. 
\end{enumerate}

\subsection{The Clifford center and pseudocenter}

\begin{definition}
An element $x\in \Cl(V,h)$ is called:
\begin{enumerate}[1.]
\itemsep 0.0em
\item {\em Central}, if it commutes with all elements of $\Cl(V,h)$:
\be
xy=y x~~\forall y\in \Cl(V,h)
\ee
\item {\em Twisted central} if it satisfies the condition:
\be
xy=\pi(y)x~~\forall y\in \Cl(V,h)
\ee
\item {\em Pseudocentral}, if it is central or twisted central. 
\end{enumerate}
\end{definition}

\noindent Let $Z(V,h)$, $A(V,h)$ and $T(V,h)$ denote the subspaces of
$\Cl(V,h)$ consisting of all central, twisted central and
pseudocentral elements, respectively. Then $Z(V,h)$ is
the center of the Clifford algebra, while $T(V,h)$ is called its {\em pseudocenter}. 
Clearly both $Z(V,h)$ and $T(V,h)$ are unital subalgebras
of $\Cl(V,h)$, while $A(V,h)$ is only a subspace. Since
$\Cl(V,h)$ is generated by $V$ (over $\R$), we have:
\beqa
&& Z(V,h)= \{x\in \Cl(V,h)|xv=vx~~\forall v\in V\}~~\\
&& A(V,h)=\{x\in \Cl(V,h)|xv=-vx~~\forall v\in V\}~~.
\eeqa
Moreover, it is easy to see that $Z(V,h)\cap A(V,h)=0$ and that $T(V,h)$ has the decomposition: 
\be
T(V,h)=Z(V,h)\oplus A(V,h)~~,
\ee
which gives a $\Z_2$-grading of the algebra $T(V,h)$ 
with components $T^0(V,h)\eqdef Z(V,h)$ and $T^1(V,h)\eqdef A(V,h)$.

\

\begin{prop}
\label{prop:center}
We have: 
\be
T(V,h)=\R\oplus \R\nu\simeq_\Alg \R[\nu]/(\nu^2=\sigma_{p,q})\simeq_\Alg  \twopartdef{\D}{p-q\equiv_4 0,1}{\C}{p-q\equiv_4 2,3}~~,
\ee
where $\D$ is the $\R$-algebra of hyperbolic (a.k.a. split complex, or
double) numbers (see Appendix \ref{app:hyp}) and $\nu$ corresponds to
the hyperbolic unit $j\in \D$ or to the imaginary unit $i\in \C$.
Moreover: 
\begin{enumerate}[1.]
\itemsep 0.0em
\item When $d$ is even, we have $\nu \in  \Cl_+(V,h)$, $Z(V,h)=\R$ and $A(V,h)=\R\nu$.  
\item When $d$ is odd, we have  $\nu\in \Cl_-(V,h)$, $A(V,h)=0$ and: 
\be
Z(V,h)=T(V,h) \simeq_\Alg\twopartdef{\D}{p-q\equiv_8 1,5}{\C}{p-q\equiv_8 3,7}~~.
\ee 
\end{enumerate}
\end{prop}

\begin{proof}
The statements regarding $Z(V,h)$ are well-known (for example, see
\cite{Karoubi}). For the statements regarding $A(V,h)$, distinguish
the cases:
\begin{enumerate}[1.]
\itemsep 0.0em
\item $d$ is even. Then clearly $\nu$ belongs to $A(V,h)$. For any
  $x\in A(V,h)$, we thus have $x\nu\in Z(V,h)=\R$, which implies $x\in
  \R\nu$ since $\nu^2=\sigma_{p,q}$. Thus $A(V,h)=\R\nu$ in this
  case.
\item $d$ is odd. Then $\nu$ is central in $\Cl(V,h)$ and belongs to
  $\Cl_-(V,h)$. These two facts imply that any $x\in A(V,h)$ satisfies
  both $x\nu=\nu x$ and $x\nu=-\nu x$, which implies $x\nu=0$ and
  hence $x=0$ since $\nu$ is invertible in $\Cl(V,h)$. Thus
  $A(V,h)=0$.
\end{enumerate}
The statements regarding $T(V,h)$ now follow immediately. 
\qed
\end{proof}

\noindent When $d$ is odd, the above gives isomorphisms of groups:
\ben
\label{Ztimes}
Z(V,h)^\times\simeq_{d=\odd} \twopartdef{\D^\times\simeq \R_{>0}\times \U(\D)}{p-q\equiv_8 1,5}{\C^\times\simeq \R_{>0}\times \U(1)}{p-q\equiv_8 3, 7}~~.
\een

\begin{remark}
Notice that $T(V,h)$ does {\em not} coincide with the super-center
$Z_{\mathrm{super}}(V,h)$ of $\Cl(V,h)$, when the latter is viewed as
a $\Z_2$-graded associative algebra. By definition of the super-center,
we have:
\be
Z_{\mathrm{super}}(V,h)=Z(V,h)\cap\Cl_+(V,h)\oplus A(V,h)\cap \Cl_-(V,h)=\R~~.
\ee
In particular $Z_{\mathrm{super}}(V,h)$ is a sub-superalgebra of
$\Cl(V,h)$, while $T(V,h)$ is not. 
\end{remark}

\subsection{Normal, simple and quaternionic cases}

It is useful to distinguish various cases according to whether
$\Cl(V,h)$ is a simple $\R$-algebra and according to the isomorphism
type of the centralizer (a.k.a. Schur algebra) $\S$ of the irreps of
$\Cl(V,h)$. This gives the following classification controlled by the
value of $p-q \mod 8$ (see Table \ref{table:CliffClassif}):
\begin{enumerate}[A.]
\itemsep 0.0em
\item Simple cases:
\begin{enumerate}[1]
\itemsep 0.0em
\item The normal simple case, when $p-q\equiv_8 0,2$
\item The complex case, when $p-q\equiv_8 3,7$
\item The quaternionic simple case, when $p-q\equiv_8 4,6$
\end{enumerate}
\item Non-simple cases:
\begin{enumerate}[1]
\itemsep 0.0em
\item The normal non-simple case, when $p-q\equiv_8 1$
\item The quaternionic non-simple case, when $p-q\equiv_8 5$
\end{enumerate}
\end{enumerate}

\begin{table}[H]
\centering
\begin{tabular}{|c|c|c|c|c|c|c|c|c|}
\hline
$\begin{array}{c} p-q\\ {\rm mod}~8 \end{array}$ & $\S$ & type &
simplicity & $Z(V,h)$ & $A(V,h)$ & $T(V,h)$ & $\nu^2$  \\
\hline\hline
$0,2$ &$\R$ &normal & simple & $\R$ & $\R\nu$ & $\D,\C$ & $1, -1$ \\
\hline
$3,7$ & $\C$& complex  & simple & $\C$ & $0$ & $\C$ & $-1$ \\
\hline
$4,6$ & $\H$& quaternionic & simple & $\R$ &$\R\nu$& $\D,\C $ & $1,-1$ \\
\hline
$1$ & $\R$ & normal & non-simple & $\D$ & $0$ & $\D$ &$+1$  \\
\hline
$5$ &$\H$& quaternionic & non-simple & $\D$ & $0$& $\D$  & $+1$  \\
\hline 
\end{tabular}
\vskip 0.2in
\caption{Classification of Clifford algebras. Here, $\nu$ is the
  Clifford volume element with respect to an orientation of $V$ while
  $Z(V,h)$ and $T(V,h)$ are the center and pseudocenter of $\Cl(V,h)$.
  Moreover, $A(V,h)$ denotes the subspace of twisted-central
  elements. Finally, $\S\subset \End_\R(S)$ the Schur algebra
  (centralizer) of any irreducible representation of $\Cl(V,h)$ (see
  Section \ref{sec:Cliffordrep}).}
\label{table:CliffClassif}
\end{table}

\subsection{Clifford norm and twisted Clifford norm}

The {\em Clifford norm} is the map $N:\Cl(V,h)\rightarrow \Cl(V,h)$ defined through:
\be
N(x)\eqdef \tau(x)x~~.
\ee
The {\em twisted Clifford norm} is the map $\tN:\Cl(V,h)\rightarrow \Cl(V,h)$ defined through:
\be
\tN(x)\eqdef \ttau(x)x~~.
\ee
These maps are $\R$-quadratic, in particular we
have $N(\lambda x)=\lambda^2 N(x)$ for any $x\in \Cl(V,h)$ and a similar relation for $\tN$.  
We also have:
\beqa
&& N(1)=\tN(1)=1~~\\
&& N(v)=-\tN(v)=h(v,v)~~\mathrm{for}~~v\in V\nn~~.
\eeqa
Notice the relation: 
\be
\tN|_{\Cl_\pm(V,h)}=\pm N|_{\Cl_\pm(V,h)}~~.
\ee

\subsection{The ordinary Clifford group}

The {\em ordinary Clifford group}\footnote{Sometimes called the {\em
    twisted Clifford group} (``groupe de Clifford tordu'' in reference
  \cite{Karoubi}).}  is the following subgroup of the group of
homogeneous units:
\be
\G(V,h)\eqdef \{a\in \Cl^\times_\hom(V,h)|\Ad^\Cl(a)(V)=V\}\subset \Cl^\times_\hom(V,h)\, .
\ee
It admits the $\Z_2$-grading inherited from $\Cl_\hom^\times(V,h)$, which has components: 
\be
\G_\pm(V,h)\eqdef \{a\in \Cl^\times_\pm(V,h)|\Ad^\Cl(a)(V)=V\} \subset \Cl^\times_{\pm}(V,h)\, .
\ee
This $\Z_2$-grading is induced by the {\em signature morphism} $\tdelta:\G(V,h)\rightarrow \mG_2$:
\ben
\label{tdelta}
\tdelta(a)=\twopartdef{+1}{a\in \G_+(V,h)}{-1}{a\in \G_-(V,h)}~~.
\een
The subgroup $\G_+(V,h)$ is called the {\em special Clifford
  group}. Since $\G(V,h)$ is generated by non-degenerate vectors, the
restriction of the Clifford norm to the ordinary Clifford group takes
values in $\R^\times$ and hence gives a group morphism $N_G\eqdef
N|_{\G(V,h)}:\G(V,h)\rightarrow \R^\times$.  Composing this with the
absolute value epimorphism $\R^\times \stackrel{|~|}\longrightarrow
\R_{>0}$ gives a surjective group morphism $|N_G|\eqdef |~|\circ
N_G:\G(V,h)\rightarrow \R_{>0}$ called {\em absolute Clifford
  norm}. We also have $\tN(\G(V,h))\subset \R^\times$ and the twisted Clifford norm 
gives a group morphism $\tN_G\eqdef \tN|_{\G(V,h)}:\G(V,h)\rightarrow
\R^\times$. We have $\tN|_{\G_\pm(V,h)}=\pm N|_{\G_\pm(V,h)}$ and hence
$|\tN_G|=|N_G|$, where $|\tN_G|\eqdef |~|\circ \tN_G$.

\subsection{Vector representations of the ordinary Clifford group} 

\begin{definition}
The {\em untwisted vector representation} of $\G(V,h)$ is the group morphism
$\Ad_0^\Cl:\G(V,h)\rightarrow \O(V,h)$ given by:
\be
\Ad_0^\Cl(a)\eqdef \Ad^\Cl(a)|_V~~\forall a\in \G(V,h)~~.
\ee
The {\em twisted vector representation} of $\G(V,h)$ is the group morphism
$\tAd_0^\Cl:\G(V,h)\rightarrow \O(V,h)$ given by:
\be
\tAd_0^\Cl(a)\eqdef \tAd^\Cl(a)|_V~~\forall a\in \G(V,h)~~.
\ee
\end{definition}

\noindent We have $\Ad_0^\Cl(v)=-R_v$ and $\tAd_0(v)=+R_v$ for any
non-degenerate vector $v\in V$. This implies the following well-known
results (see, for example, \cite{Karoubi}):

\

\begin{prop}
The ordinary Clifford group $\G(V,h)$ is generated by non-degenerate
vectors $v\in V$ and we have:
\ben
\label{tdeltaDet}
\det\circ \tAd_0^\Cl=\tdelta~~.
\een
\end{prop}
\begin{prop}
The twisted vector representation of the ordinary Clifford group
satisfies $\tAd_0^\Cl(\G(V,h))=\O(V,h)$ and gives a short exact
sequence:
\ben
\label{CliffTwVect}
1\longrightarrow \R^\times \hookrightarrow \G(V,h)\stackrel{\tAd_0^\Cl}{\longrightarrow} \O(V,h)\longrightarrow 1~~,
\een
which restricts to a short exact sequence: 
\ben
\label{CliffRes}
1\longrightarrow \R^\times \hookrightarrow \G_+(V,h)\stackrel{\tAd_0^\Cl|_{\G_+(V,h)}=\Ad_0^\Cl|_{\G_+(V,h)}}{\longrightarrow} \SO(V,h)\longrightarrow 1~~.
\een
\end{prop}
Let: 
\ben
\label{hatO}
{\hat \O}(V,h)\eqdef\twopartdef{\O(V,h)}{d=\mathrm{even}}{\SO(V,h)}{d=\odd}~~.
\een
\begin{prop}
The untwisted vector representation of the ordinary Clifford group
satisfies $\Ad_0^\Cl(\G(V,h))={\hat \O}(V,h)$ and gives a short exact
sequence:
\ben
\label{CliffVect}
1\longrightarrow \R^\times \hookrightarrow \G(V,h)\stackrel{\Ad_0^\Cl}{\longrightarrow} {\hat \O}(V,h)\longrightarrow 1~~,
\een
which restricts to \eqref{CliffRes}.
\end{prop}

\

\noindent Consider the morphism of groups $f:\O(V,h)\rightarrow \O(V,h)$
given by $f(R)\eqdef (\det R) R$. 

\begin{prop}
\label{prop:fmorphism}
We have $\Ad_0^\Cl=f\circ \tAd_0^\Cl$ on $\G(V,h)$. Moreover:
\begin{enumerate}[1.]
\itemsep 0.0em
\item When $d$ is even, $f$ is an automorphism of $\O(V,h)$.
\item When $d$ is odd, $f$ induces an isomorphism
  $\O(V,h)/\{-\id_V,\id_V\}\simeq \SO(V,h)$. In this case, the map
  $\varphi\eqdef f\times \det:\O(V,h)\stackrel{\sim}{\rightarrow}
  \SO(V,h)\times \mG_2$ given by $\varphi(R)=(f(R),\det R)$ is an
  isomorphism of groups and we have:
\be
\Ad_0^\Cl\times \tdelta=\varphi\circ \tAd_0^\Cl~~.
\ee
\end{enumerate}
\end{prop}

\proof The relation $\det f(R)=(\det R)^{d+1}$ gives: 
\be
f(\O(V,h))=\twopartdef{\O(V,h)}{d=\mathrm{even}}{\SO(V,h)}{d=\mathrm{odd}}~~,~~\ker f=\twopartdef{\{\id_V\}}{d=\mathrm{even}}{\{-\id_V,\id_V\}}{d=\mathrm{odd}}~~.
\ee
The remaining statements are obvious. \qed

\

\begin{remark}
I general, the group $\O(p,q)$ has non-trivial outer automorphisms.
Determining the full outer automorphism group of $\O(p,q)$ for general $p,q$ turns out to be a surprisingly subtle problem.
\end{remark}
\begin{remark}
It is well-known that the ordinary Clifford group can also
be described as:
\be
\G(V,h)=\{a\in \Cl(V,h)^\times|\tAd^\Cl(a)(V)=V\}~~.
\ee
\end{remark}

\subsection{The pin and spin groups and their vector representations}

\begin{definition}
The {\em pin group} $\Pin(V,h)$ is the subgroup of $\G(V,h)$ generated
by the vectors $v\in V$ satisfying $|h(v,v)| = 1$. The {\em spin group} is the subgroup
$\Spin(V,h)\eqdef \Pin(V,h)\cap \Cl_+(V,h)$.
\end{definition}

\

\noindent The decomposition $\Pin(V,h)=\Pin_+(V,h)\sqcup \Pin_-(V,h)$,
where $\Pin_+(V,h)\eqdef \Spin(V,h)$ and $\Pin_-(V,h)\eqdef
\Pin(V,h)\cap \Cl_-(V,h)$ gives a $\Z_2$-grading of the pin group. 
We have $\ker |N_G|=\Pin(V,h)$ and an exact sequence:
\ben
\label{pinseq}
1\longrightarrow \Pin(V,h) \hookrightarrow \G(V,h)\stackrel{|N_G|}{\longrightarrow} \R_{>0}\longrightarrow 1~~,
\een
which restricts to an exact sequence:
\ben
\label{spinseq}
1\longrightarrow \Spin(V,h)\hookrightarrow \G_+(V,h)\stackrel{|N_{G_+}|}{\longrightarrow} \R_{>0}\longrightarrow 1~~,
\een
where $N_{G_+}\eqdef N|_{G_+(V,h)}$. In particular, $\G(V,h)\simeq
\R_{>0}\times \Pin(V,h)$ is homotopy-equivalent with $\Pin(V,h)$ while
$\G_+(V,h)\simeq \R_{>0}\times \Spin(V,h)$ is homotopy-equivalent with
$\Spin(V,h)$.

\begin{definition}
The untwisted vector representation of $\Pin(V,h)$ is the restriction
of $\Ad_0^\Cl$ to $\Pin(V,h)$.  The twisted vector representation of
$\Pin(V,h)$ is the restriction of $\tAd_0^\Cl$ to $\Pin(V,h)$.  The
vector representation of $\Spin(V,h)$ is the common restriction of
$\Ad_0^\Cl$ or $\tAd_0^\Cl$ to $\Spin(V,h)$.
\end{definition}

\noindent The twisted vector representation gives an exact sequence:
\ben
\label{pinseqtAd}
1\longrightarrow \mG_2\hookrightarrow \Pin(V,h)\stackrel{\tAd_0^\Cl}{\longrightarrow} \O(V,h)\longrightarrow 1~~
\een
which restricts to an exact sequence: 
\ben
\label{spinseqtAd}
1\longrightarrow \mG_2\hookrightarrow \Spin(V,h)\stackrel{\tAd_0^\Cl|_{\Spin(V,h)}=\Ad_0^\Cl|_{\Spin(V,h)}}{\longrightarrow} \SO(V,h)\longrightarrow 1~~.
\een
The situation is summarized in the following commutative diagram with
exact rows and columns, where $\sq(x)=x^2$ for $x\in \R^\times$:
\begin{equation}
\label{diagram:PinGtwisted}
\scalebox{1.0}{
\xymatrix{
&1 \ar[d]&1\ar[d]& &\\
1 \ar[r] & \mG_2~\ar[d] \ar[r]~ &~\Pin(V,h) ~\ar[r]^{~~\tAd_0^\Cl} \ar[d] &\O(V,h) \ar@{=}[d]~\ar[r] &1\\
1 \ar[r]~ & ~\R^\times\ar[r]\ar[d]^{\sq}~& ~\G(V,h) ~\ar[r]^{~~\tAd_0^\Cl}\ar[d]^{|N_G|}& \O(V,h)\ar[r]&1\\
& ~~\R_{>0}~\ar[d]\ar@{=}[r]&~~\R_{>0}\ar[d]& &\\
& 1&1 & &\\
  }}
\end{equation}
The untwisted vector representation gives an exact sequence: 
\ben
\label{pinseqAd}
1\longrightarrow \mG_2\hookrightarrow \Pin(V,h)\stackrel{\Ad_0^\Cl}{\longrightarrow} {\hat \O}(V,h)\longrightarrow 1~~
\een
which restricts to \eqref{spinseqtAd} and we have a commutative diagram similar to
\eqref{diagram:PinGtwisted}. 

\subsection{Relation between the twisted and untwisted vector representations of $\G(V,h)$ and $\Pin(V,h)$ for even $d$}

\

\

\noindent The following result relates the groups $\G(V,h)$ and $\G(V,-\sigma_{p,q}h)$ when $d$ is even (cf. \cite{TrautmanPin2}):

\begin{prop}
\label{prop:AdtAdPin}
When $d$ is even, there exists an isomorphism of groups
  $\varphi:\G(V,h)\stackrel{\sim}{\rightarrow} \G(V,-\sigma_{p,q} h)$ such that $\Ad_0^\Cl\circ
  \varphi=\tAd_0^\Cl$. This restricts to an isomorphism from
  $\Pin(V,h)$ to $\Pin(V, -\sigma_{p,q} h)$ having the same property. Thus: 
\begin{enumerate}
\itemsep 0,0em
\item When $p-q\equiv_8 0,4$, there exists an isomorphism of groups
  $\varphi: \G(V,h)\stackrel{\sim}{\rightarrow} \G(V,-h)$ such that $\Ad_0^\Cl\circ
  \varphi=\tAd_0^\Cl$. This restricts to an isomorphism from
  $\Pin(V,h)$ to $\Pin(V,-h)$ having the same property.
\item When $p-q\equiv_8 2,6$, there exists a group automorphism
  $\varphi:\G(V,h)\stackrel{\sim}{\rightarrow} \G(V,h)$ such that $\Ad_0^\Cl\circ
  \varphi=\tAd_0^\Cl$. This restricts to an automorphism of
  $\Pin(V,h)$ having the same property.
\end{enumerate}
\end{prop}

\proof
Since $d$ is even, we have:
\be
\nu^2=\sigma_{p,q}=(-1)^{q+\frac{d}{2}}=(-1)^{\frac{p-q}{2}}=\twopartdef{+1}{p-q\equiv_8
  0,4}{-1}{p-q\equiv_8 2,6}~~. 
\ee
Moreover, $\nu$ anticommutes with $v$ for all $v\in V$ and hence $(\nu
v)^2=-\sigma_{p,q} v^2$. Therefore, the map $V\ni v\rightarrow \nu v\in V'\eqdef \nu V\subset \Cl(V,h)$
extends (upon identifying the vector space $V'$ with $V$) to a unital isomorphism of $\R$-algebras
$\varphi_0:\Cl(V,h)\stackrel{\sim}{\rightarrow} \Cl(V,-\sigma_{p,q}h)$,
which restricts to an isomorphism
$\varphi:\G(V,h)\stackrel{\sim}{\rightarrow} \G(V,-\sigma_{p,q}h)$. For any
non-degenerate vectors $v_1,\ldots v_k\in V$, we have:
\beqa
&& \Ad_0^\Cl(\varphi(v_1\ldots v_k))=\Ad_0^\Cl(\nu v_1)\circ \ldots \circ \Ad_0^\Cl(\nu v_k)=\nn\\
&& \Ad_0^\Cl(\nu) \circ \Ad_0^\Cl(v_1)\circ \ldots \circ \Ad_0^\Cl(\nu)\circ \Ad_0^\Cl(v_k) =\tAd_0^\Cl(v_1\ldots v_k)~~,
\eeqa
where we noticed that $\Ad_0^\Cl(\nu)=-\id_V$ since $d$ is even. Thus
$\Ad_0^\Cl(\varphi(g))=\tAd_0^\Cl(g)$ for all $g\in \G(V,h)$.
\qed

\subsection{Connected components of the pin and pseudo-orthogonal groups}

When $pq\neq0$, the group $\Spin(V,h)$ has two connected components
given by:
\be
\Spin^\pm (V,h)\eqdef \{a\in \Spin(V,h)| N(a)=\pm 1\}~~,
\ee
where $\Spin^+(V,h)$ is the connected component of the identity. The
decomposition $\Spin(V,h)=\Spin^+(V,h)\sqcup \Spin^-(V,h)$ is the
$\Z_2$-grading induced by the group morphism
$N|_{\Spin(V,h)}:\Spin(V,h)\rightarrow \mG_2$. Accordingly,
$\Pin(V,h)$ has four connected components given by:
\be
\Pin_\epsilon^\eta(V,h)=\{a\in \Pin(V,h)|\tdelta(a)=\epsilon~,~N(a)=\eta\}~~,~~(\epsilon,\eta\in \{-1,1\})~~,
\ee
where $\Pin_+^+(V,h)=\Spin^+(V,h)$ is the connected component of the
identity. The decomposition $\Pin(V,h)=\sqcup_{\epsilon,\eta\in
  \{-1,+1\}}\Pin_\epsilon^\eta(V,h)$ is the $D_4$-grading induced by
the group morphism $(\tdelta|_{\Pin(V,h)}) \times
(N|_{\Pin(V,h)})\rightarrow \mG_2\times \mG_2$.  In particular, the
component group $\pi_0(\Pin(V,h))=\Pin(V,h)/\Spin(V,h)$ is isomorphic
with $D_4$. Since $N(-1)=N(1)=1$, the Clifford norm descends through
the 2-fold covering map of \eqref{pinseqtAd} to a group morphism
$N_0:\O(V,h)\rightarrow \mG_2$.  Similarly, we have
$\tdelta(-1)=\tdelta(1)=1$ and $\tdelta$ descends to the determinant
morphism $\det: \O(V,h)\rightarrow \mG_2$. These two morphisms give a
$D_4$-grading $\O(V,h)=\sqcup_{\epsilon,\eta\in
  \{-1,+1\}}\O_\epsilon^\eta(V,h)$ which coincides with the
decomposition of $\O(V,h)$ into connected components:
\be
\O_\epsilon^\eta(V,h)=\{a\in \O(V,h)|\det a=\epsilon~,~N_0(a)=\eta\}~~.
\ee
The group $\O_+^+(V,h)$ is the connected component of the identity and
the component group $\pi_0(\O(V,h))=\O(V,h)/\SO(V,h)$ is isomorphic
with $D_4$. Accordingly, the group $\SO(V,h)=\{a\in \O(V,h)|\det
a=+1\}$ has two connected components distinguished by the morphism
$N_0$, which gives a $\Z_2$-grading of $\SO(V,h)$:
\be
\SO^\pm(V,h)=\{a\in \SO(V,h)|N_0(a)=\pm 1\}~~,
\ee
where $\SO^+(V,h)=\O_+^+(V,h)$ is the connected component of the
identity. The groups $\O(V,h)$, $\SO(V,h)$ and $\SO^+(V,h)$ are
homotopy equivalent with their maximal compact forms $\O(p)\times
\O(q)$, $\mathrm{S}[\O(p)\times \O(q)]$ and $\SO(p)\times \SO(q)$
respectively. When $p=0$ and $q\neq 0$, we have $\O(V,h)\simeq \O(q)$ while
for $p\neq 0$ and $q=0$ we have $\O(V,h)\simeq \O(p)$, hence in these cases
$\SO(V,h)$ is connected and $\O(V,h)$ has two connected components
distinguished by the determinant; similar remarks apply to
$\Spin(V,h)$ and $\Pin(V,h)$.

\subsection{Presentation of the pin group in terms of the spin group}
For any $d$, we have $\nu\in \Pin(V,h)$ and: 
\be
\Ad_0^\Cl(\nu)=(-1)^{d-1}\id_V~~,~~\tAd_0^\Cl(\nu)=-\id_V~~. 
\ee
Thus:
\be
\tAd_0^\Cl(\nu a)=\tAd_0^\Cl(\nu)\tAd_0^\Cl(a)=-\tAd_0^\Cl(a)~~\mathrm{for}~~a\in \Pin(V,h)~~.
\ee
When $d$ is even, we have $\nu\in \Spin(V,h)$ and $\Ad_0^\Cl(\nu)=-\id_V\in \SO(V,h)$. 
When $d$ is odd, we have $\nu\in \Pin_-(V,h)$, $\Ad_0^\Cl(\nu)=\id_V\in \SO(V,h)$, $\tAd_0^\Cl(\nu)=-\id_V\in \O_-(V,h)$ and every element of $\Pin_-(V,h)$ can be
written as $\nu a$ for a uniquely-determined $a\in\Spin(V,h)$. Thus:
\be
\Pin_-(V,h)=_{d=\mathrm{odd}} \Spin(V,h)\nu
\ee
and: 
\ben
\label{pinspin}
\Pin(V,h)=_{d=\mathrm{odd}}\Spin(V,h)\langle \nu \rangle\simeq\twopartdef{\Spin(V,h)\times \mG_2}{p-q\equiv_8 1,5}{\Spin(V,h)\cdot \mG_4}{p-q\equiv_8 3, 7}~~.
\een
In this presentation, the twisted adjoint representation of $\Pin(V,h)$ for odd $d$ is given as follows: 
\begin{enumerate}[1.]
\item If $p-q\equiv_8 1,5$, then: 
\be
\tAd_0(a,g)=g\Ad_0(a)~~\forall a\in \Spin(V,h)~~\mathrm{and}~~g\in \mG_2
\ee
\item If $p-q\equiv_8 3, 7$, then: 
\be
\tAd_0([a,g])=g^2\Ad_0(a)~~\forall a\in \Spin(V,h)~~\mathrm{and}~~g\in \mG_4
\ee
\end{enumerate}
Notice that $\Ad_0(-a)=\Ad_0(a)$ and $\tAd_0(-a)=\tAd_0(a)$ for all $a\in \Pin(V,h)$. 

\

\begin{remark}
Notice the equivalences: 
\ben
p-q\equiv_8 3\Longleftrightarrow q-p\equiv_8 5~~,~~p-q\equiv_8 7 \Longleftrightarrow q-p \equiv_8 1~~,
\een
which interchange the complex case in signature $(p,q)$ with the
non-simple case in signature $(q,p)$. In fact, the interchange
$p\leftrightarrow q$ corresponds to $h\leftrightarrow -h$. Since $d$ (which is assumed odd)
is invariant under this transformation, the volume forms
$\nu_{h}=\nu_{p,q}$ and $\nu_{-h}=\nu_{q,p}$ (taken with respect to
some fixed orientation of $V$) are central in $\Cl(V,h)$ and
$\Cl(V,-h)$, respectively and we have:
\be
\nu_{-h}^2=-\nu_h^2~~.
\ee
Since $\Spin(V,-h)\simeq \Spin(V,h)$, we have: 
\be
\Pin(V,-h) \simeq \twopartdef{\Spin(V,h)\cdot \mG_4}{p-q\equiv_8 1,5}{\Spin(V,h)\times \mG_2}{p-q\equiv_8 3, 7}~~.
\ee
The group $\Pin(V,-h)$ is sometimes denoted $\Pin^{-}(V,h)$ and is
used in the theory of $\Pin^-$ structures \cite{KirbyTaylor}.
\end{remark}

\subsection{The extended Clifford group}

\begin{definition}
The {\em extended Clifford group}\footnote{Sometimes called the {\em
    untwisted} Clifford group or simply the ``Clifford group'' in
  older literature such as \cite{Karoubi}.} $\G^e(V,h)$ is the
subgroup of $\Cl(V,h)^\times$ defined through:
\be
\G^e(V,h)\eqdef \{x\in \Cl(V,h)^\times|\Ad^\Cl(x)(V)=V\}~~.
\ee
\end{definition}

\

\noindent The ordinary Clifford group $\G(V,h)$ coincides with the
subgroup formed by those elements of $\G^e(V,h)$ which are homogeneous
with respect to the canonical $\Z_2$-grading of $\Cl(V,h)$. We have:
\be
\G^e(V,h)\cap \Cl_\pm(V,h)=\G_\pm(V,h)~~,~~\G^e(V,h)\cap \Cl_\hom^\times(V,h)=\G(V,h)~~.
\ee 

\begin{prop} 
\label{prop:Ge}
We have: 
\beqa
\G^e(V,h)~~&= &Z(V,h)^\times \G(V,h)~\nn\\
Z(\G^e(V,h)) &=& Z(V,h)^\times~\nn\\
Z(\G(V,h))~ &=& \G(V,h)\cap Z(V,h)^\times~~
\eeqa
and $\G^e(V,h)\simeq [\G(V,h)\times Z(V,h)^\times]/Z(\G(V,h))$. 
Namely:
\begin{enumerate}[A.]
\itemsep 0.0em
\item When $d$ is even, we have $\G^e(V,h)=\G(V,h)$.
\item When $d$ is odd, we have:
\be
Z(\G(V,h))= \R^\times \sqcup \R^\times \nu \simeq_{\Gp} \R_{>0} \times \{1, \nu,-1,-\nu\} \simeq_{\Gp} \twopartdef{\R_{>0}\times D_4}{p-q\equiv_8 1, 5}{\R_{>0}\times \mG_4}{p-q\equiv_8 3, 7}~~.
\ee
and:
\be
\G^e(V,h)\simeq_{\Gp} \twopartdef{\R_{>0}\times \G(V,h)}{p-q\equiv_8 1, 5}{\left[\U(1)\times \G(V,h)\right]/\mG_4}{p-q\equiv_8 3,7}~~.
\ee
Furthermore:
\be
Z(V,h)^\times/Z(\G(V,h))\simeq_\Gp \twopartdef{\U(\D)/ D_4 \simeq_\Gp \R_{>0}}{p-q\equiv_8 1, 5}{\U(1)/\mG_4\simeq_\Gp \U(1)}{p-q\equiv_8 3,7}
\ee
and there exists a short exact sequence:
\ben
\label{LGseq}
1\longrightarrow \G(V,h) \hookrightarrow \G^e(V,h)\longrightarrow Z(V,h)^\times/Z(\G(V,h))\longrightarrow 1~~.
\een
\end{enumerate}
\end{prop}
\begin{proof}
The extended Clifford group is generated by the Clifford group and by
the elements of $Z(V,h)^\times$ (see \cite{Karoubi}), thus
$\G^e(V,h)=Z(V,h)^\times \G(V,h)$. Let $e_1\ldots e_d$ be an
orthonormal basis of $(V,h)$. Since $e_i\in \G(V,h)\subset \G^e(V,h)$,
any $a\in Z(\G^e(V,h))$ satisfies $ae_i=e_ia$ for all $i=1\ldots d$,
which implies $av=va$ for all $v\in V$. Thus $a\in
Z(V,h)^\times$. This shows that $Z(\G^e(V,h))\subset Z(V,h)^\times$
and also that $Z(\G(V,h))\subset Z(V,h)^\times$.  Since the inverse to
the first inclusion is obvious, we conclude that
$Z(\G^e(V,h))=Z(V,h)^\times$. Since $Z(\G(V,h))\subset Z(V,h)^\times$,
we have $Z(\G(V,h))\subset \G(V,h)\cap Z(V,h)^\times$.  The inverse of
this inclusion is obvious, so we have $Z(\G(V,h))= \G(V,h)\cap
Z(V,h)^\times$.  The surjective morphism of groups given by
$\G(V,h)\times Z(V,h)^\times \ni (a,\alpha)\rightarrow \alpha a\in
\G^e(V,h)$ has kernel equal to $\{(\alpha^{-1},\alpha)|\alpha\in
Z(\G(V,h))\}\simeq Z(\G(V,h))$, which shows that $\G^e(V,h)\simeq
[\G(V,h)\times Z(V,h)^\times]/Z(\G(V,h))$. Proposition
\ref{prop:center} implies statements A. and B., where in the case
$p-q\equiv_8 1,5$ we used the isomorphisms $\D^\times\simeq
\R_{>0}\times \U(\D)$ and $\U(\D)\simeq \U^{++}(\D)\times D_4 \simeq
\R_{>0}\times D_4$ (see Appendix \ref{app:hyp}).  \qed
\end{proof}

\begin{prop}
\label{prop:GePinSpin}
The following statements hold:
\begin{enumerate}[A.]
\itemsep 0.0em
\item When $d$ is even, we have $\G^e(V,h)=\G(V,h)\simeq_{\Gp} \R_{>0}\times
  \Pin(V,h)$.
\item When $d$ is odd, we have:
\be
\G^e(V,h) \simeq_{\Gp} Z(V,h)^\times  \Pin(V,h)\simeq_{\Gp} \twopartdef{\left[ \D^\times \times \Pin(V,h)\right]/D_4 }{p-q\equiv_8 1,5}{ \left[\C^\times\times \Pin(V,h)\right]/\mG_4}{p-q\equiv_8 3, 7}
\ee
and:
\be
\G^e(V,h) \simeq_{\Gp} Z(V,h)^\times  \Spin(V,h) \simeq_{\Gp} \twopartdef{\D^\times \cdot \Spin(V,h)}{p-q\equiv_8 1,5}{\C^\times \cdot\Spin(V,h)}{p-q\equiv_8 3, 7}
\ee
\end{enumerate}
\end{prop}

\proof 
Follows from relations \eqref{pinspin} and Proposition \ref{prop:Ge}.
\qed 

\subsection{The vector representation of the extended Clifford group}

\begin{definition}
The {\em vector representation} of $\G^e(V,h)$ is the group morphism
$\Ad_0^e:\G^e(V,h)\rightarrow {\hat \O}(V,h)$ given by:
\be
\Ad_0^e(a)\eqdef \Ad^\Cl(a)|_V~~(a\in \G^e(V,h))~~.
\ee
Its restriction to $\G(V,h)$ coincides with the {\em untwisted} vector representation of $\G(V,h)$. 
\end{definition}

When $d$ is even, the vector representation of the extended Clifford group gives the exact sequence \eqref{CliffVect}: 
\be
1\longrightarrow Z(V,h)^\times=\R^\times \hookrightarrow \G^e(V,h)=\G(V,h)\stackrel{\Ad_0^e=\Ad_0^\Cl}{\longrightarrow} \O(V,h)\longrightarrow 1~~
\ee
while when $d$ is odd it gives an exact sequence (cf. \cite[Proposition (1.1.8)]{Karoubi}): 
\be
1\longrightarrow Z(V,h)^\times \hookrightarrow \G^e(V,h)\stackrel{\Ad_0^e}{\longrightarrow} \SO(V,h)\longrightarrow 1~~,
\ee
where $Z(V,h)^\times$ was given in \eqref{Ztimes}. 

\subsection{Volume grading of the extended Clifford group}

The volume grading of $\Cl(V,h)$ induces a grading of $\G^e(V,h)$. From the properties of the volume grading of $\Cl(V,h)$ we obtain the following proposition.

\begin{prop}
\label{prop:CliffVol}
The following statements hold:
\begin{enumerate}[A.]
\itemsep 0.0em
\item When $d$ is even, the volume grading of $\G^e(V,h)$ coincides
  with the canonical $\Z_2$-grading of $\G(V,h)$.
\item When $d$ is odd, the volume grading of $\G^e(V,h)$ is
  concentrated in degree zero:
\be
\G^e(V,h)^0=\G^e(V,h)~~,~~\G^e(V,h)^1=\emptyset~~.
\ee
\end{enumerate}
\end{prop}

\subsection{The improved reversion and improved Clifford norm}
\label{subsec:ImprovedNorm}

Since $\tau(Z(V,h))\subset Z(V,h)$, we have $N(Z(V,h))\subset Z(V,h)$,
which implies $N(\G^e(V,h))\subset Z(V,h)^\times$. Since
$Z(V,h)^\times$ is an Abelian group, it follows that the restriction of
$N$ gives a group morphism from $\G^e(V,h)$ to $Z(V,h)^\times$. In
general, $N(\G^e(V,h))$ is larger than $\R^\times$. Let:
\ben
\label{epsilond}
\epsilon_d\eqdef -(-1)^{\left[\frac{d}{2}\right]}=\twopartdef{-1}{d\equiv_4 0,1}{+1}{d\equiv_4 2,3}~~.
\een

\begin{definition}
The {\em improved reversion} is the unique unital anti-automorphism $\tau_e$
of $\Cl(V,h)$ which satisfies $\tau_e(v)=\epsilon_d v$ for all $v\in V$, namely:
\be
\tau_e(x)\eqdef \tau\circ \pi^{\frac{1-\epsilon_d}{2}}=\twopartdef{\ttau=\tau\circ \pi}{d\equiv_4 0,1}{\tau}{d\equiv_4 2,3}~~.
\ee
\end{definition}

\

\noindent With this definition, we have: 
\be
\tau_e(\nu)=\twopartdef{-\nu}{d\equiv_4 1,2,3}{+\nu}{d\equiv_4 0}~~.
\ee
In particular, $\tau_e(\nu)=-\nu$ when $d$ is odd ($d\equiv_4
1,3$). Thus $\tau_e$ acts as conjugation of $Z(V,h)\simeq \C$ in the
complex case $p-q\equiv_8 3, 7$ and as conjugation of $Z(V,h)\simeq
\D$ in the normal non-simple and quaternionic non-simple cases $p-q\equiv_8 1,
5$, i.e. in all cases when $Z(V,h)$ is not reduced to $\R$. 

\begin{definition}
The {\em improved Clifford norm} is the map $N_e:\Cl(V,h)\rightarrow \Cl(V,h)$ defined 
through:
\ben
\label{Ne}
N_e(x)\eqdef \tau_e(x)x=\twopartdef{\tN(x)}{d\equiv_4 0,1}{N(x)}{d\equiv_4 2,3}~~(x\in \Cl(V,h))~~.
\een
\end{definition}

\noindent Notice that $N_e(x)=x^2$ for all $x\in \R$.

\begin{prop}
\label{prop:improved}
The improved Clifford norm satisfies $N_e(Z(V,h))\subset \R$ and
\be
N_e|_{Z(V,h)}=\rM~~,
\ee
namely:
\begin{enumerate}[1.]
\itemsep 0.0 em
\item When $d$ is even, $N_e|_{Z(V,h)}$ coincides with the squared
  absolute value (and hence with the squaring function) on $Z(V,h)=\R$
\item When $p-q\equiv_8 3,7$, $N_e|_{Z(V,h)}$ coincides with the
  squared absolute value on $Z(V,h)\simeq \C$
\item When $p-q\equiv 1,5$, $N_e|_{Z(V,h)}$ coincides with the
  hyperbolic modulus on $Z(V,h)\simeq \D$.
\end{enumerate}
\end{prop}

\proof The case of even $d$ is obvious (since $Z(V,h)=\R$ in that case).  For odd $d$ and
$\alpha,\beta\in \R$, we have:
\be
N_e(\nu)=-\nu^2=\twopartdef{+1}{p-q\equiv_8 3,7}{-1}{p-q\equiv_8 1,5}
\ee
and:
\be
N_e(\alpha+\beta\nu)=\twopartdef{\alpha^2+\beta^2=|z|^2}{p-q\equiv_8 3,7}{\alpha^2-\beta^2=\rM(z)}{p-q\equiv_8 1,5}~~,
\ee 
where:
\be
z=\twopartdef{\alpha+i\beta\in \C}{p-q\equiv_8 3,7}{\alpha+j\beta\in \D}{p-q\equiv_8 1,5}~~.
\ee
\qed

\begin{prop}
$N_e$ induces a group morphism $N_e|_{\G^e(V,h)}:\G^e(V,h)\rightarrow
  \R^\times$.
\end{prop}

\begin{proof}
Since $N(\G(V,h))\subset \R^\times$ and $\tN(\G(V,h))\subset
\R^\times$, it is clear that $N_e(\G(V,h))\subset \R^\times$. The
conclusion follows from the previous proposition using the fact that
$\G^e(V,h)=Z(V,h)^\times \G(V,h)$.  \qed
\end{proof}

\noindent Composing $N_e|_{\G^e(V,h)}$ with the absolute value morphism
$|~|:\R^\times \rightarrow \R_{>0}$ gives a group morphism
$|N_e|:\G^e(V,h)\rightarrow \R_{>0}$:
\be
|N_e|(x)=|N_e(x)|~~\forall x\in \G^e(V,h)~~.
\ee 
For any $v\in V$, we have:
\be
N_e(v)=\twopartdef{-N(v)}{d\equiv_4 0, 1}{+N(v)}{d\equiv_4 2,3}
\ee
and hence: 
\ben
\label{Neres}
|N_e|_{\G(V,h)}=|N||_{\G(V,h)}~~.
\een

\subsection{The extended pin group}

\begin{definition}
The {\em extended pin group} is defined through:
\ben
\Pin^e(V,h)\eqdef \ker\left(|N_e|:\G^e(V,h)\rightarrow \R_{>0}\right)
\een
\end{definition}

\

\noindent We have a short exact sequence: 
\ben
\label{pineseq}
1\longrightarrow \Pin^e(V,h) \hookrightarrow \G^e(V,h)\stackrel{|N_e|}{\longrightarrow} \R_{>0}\longrightarrow 1~~
\een
and a commutative diagram with exact rows and columns: 
\begin{equation}
\label{diagram:PineGe}
\scalebox{1.0}{
\xymatrix{
&1 \ar[d]&1\ar[d]& &\\
1 \ar[r] & \U(Z(V,h))~\ar[d] \ar[r]~ &~\Pin^e(V,h) ~\ar[r]^{\!\!\!\!\Ad_0^e} \ar[d] &{\hat \O}(V,h) \ar@{=}[d]~\ar[r] &1\\
1 \ar[r]~ & ~Z(V,h)^\times\ar[r]\ar[d]^{|\rM|}~& ~\G^e(V,h) ~\ar[r]^{\!\!\!\!\Ad_0^e}\ar[d]^{|N_e|}& {\hat \O}(V,h)\ar[r]&1\\
& ~~\R_{>0}~\ar[d]\ar@{=}[r]&~~\R_{>0}\ar[d]& &\\
& 1&1 & &\\
  }}
\end{equation}
where $|\rM|\eqdef |~|\circ \rM$ and:
\begin{eqnarray}
\U(Z(V,h)) &\eqdef & \{x\in Z(V,h)||N_e(x)|=1\}= \{x\in Z(V,h)||\rM(x)|=1\}\nonumber\\ &=& \threepartdef{\mG_2}{p-q\equiv_8 0,2,4,6}{\U(1)}{p-q\equiv_8 3,7}{\U(\D)}{p-q\equiv_8 1,5}\, .
\end{eqnarray}
In all cases, we have $Z(V,h)^\times=\{z\in Z(V,h)|\rM(z)\neq 0\}$ and any element
$z\in Z(V,h)^\times$ can be written as $z=u\sqrt{|\rM(z)|}$, where $u\in
\U(Z(V,h))$ is uniquely determined by $z$. Thus: 
\ben
\label{GePine}
\G^e(V,h)=\R_{>0}\Pin^e(V,h)\simeq \R_{>0}\times \Pin^e(V,h)~~.
\een
In particular, $\G^e(V,h)$ is homotopy equivalent with $\Pin^e(V,h)$. 

\begin{definition}
Define: 
\beqan
&&\Spin^c(V,h)\eqdef \Spin(V,h)\cdot \U(1)\nn\\
&&\Spin^h(V,h)\eqdef \Spin(V,h)\cdot \U(\D)~~.
\eeqan
\end{definition}

\begin{prop} 
\label{prop:Pine}
We have $\Pin^e(V,h)=\U(Z(V,h))\Pin(V,h)$, namely:
\begin{enumerate}[1.]
\item In the simple normal and simple quaternionic cases ($p-q\equiv_8
  0,2,4,6$), we have $\Pin^e(V,h)=\Pin(V,h)$.
\item In the complex case ($p-q\equiv_8 3,7$), we have
  $\Pin^e(V,h)=\Spin(V,h)\U(1)\simeq \Spin^c(V,h)$.
\item In the non-simple cases ($p-q\equiv_8 1,5$), we have
  $\Pin^e(V,h)=\Spin(V,h)\U(\D)\simeq \Spin^h(V,h)$.
\end{enumerate}
\end{prop}

\proof

\

\

\noindent 1. Follows from the fact that $d$ is even and hence $Z(V,h)=\R$ and $\G^e(V,h)=\G(V,h)$ in the simple
normal and simple quaternionic cases.

\

\noindent 2. and 3. In the complex and non-simple cases, $d$ is odd
and any $x\in \G^e(V,h)$ can be written as $x=za$ for some $z\in
Z(V,h)^\times$ and $a\in \Spin(V,h)$. Since $|N_e(a)|=|N(a)|=1$, we
have $|N_e(x)|=|N_e(z)|=|\rM(z)|$, which equals $1$ iff $z\in
\U(Z(V,h))$. In both cases we have $\U(Z(V,h))\cap
\Spin(V,h)=\{-1,1\}$, which gives the conclusion.  \qed

\

\noindent Proposition \ref{prop:Ge} implies $Z(\Pin^e(V,h))=\U(Z(V,h))$. The situation is summarized in Table \ref{table:Pine}.

\begin{table}[H]
\centering
\begin{tabular}{|c|c|c|c|c|}
\hline
\!$\begin{array}{c} p-q\\ {\rm mod}~8 \end{array}$ & \! type & $Z(V,h)$ & \!$\U(Z(V,h))$ & \!$\!\Pin^e(V,h)$ \\
\hline\hline
$0,2$ &normal simple & $\R$ & $\mG_2$ & $\Pin(V,h)$\\
\hline
$3,7$ & complex simple & $\C$ & $\U(1)$ & $\Spin^c(V,h)$ \\
\hline
$4,6$ & quat. simple & $\R$ &$\mG_2$& $\Pin(V,h)$ \\
\hline
$1$ & normal non-simple & $\D$ & $\U(\D)$ & $\Spin^h(V,h)$ \\
\hline
$5$ & \!quat. non-simple\! & $\D$ & $\U(\D)$& $\Spin^h(V,h)$  \\
\hline 
\end{tabular}
\vskip 0.2in
\caption{Extended pin groups.}
\label{table:Pine}
\end{table}

Table \ref{table:ClassicalCliffordI} summarizes the relation of some
low-dimensional Clifford algebras with certain classical algebras and
our notation for the latter\footnote{Notice that other references,
  such as \cite{Karoubi}, use a convention in which $p$ and $q$ are
  interchanged.}. Table \ref{table:ClassicalCliffordII} describes the
corresponding $\Spin, \Pin$ and $\Pin^e$ groups.

\begin{table}[!htbp]
\centering
\begin{tabular}{|c|c|c|c|c|}
\hline
$\Cl(V,h)$ & name of numbers & notation & $Z(V,h)$& $\begin{array}{c}\mathrm{isomorphic} \\\mathrm{descriptions}\end{array}$ \\
\hline\hline
$\Cl_{0,1}$ &  complex  & $\C$ & $\C$ & $\C$\\
\hline
$\Cl_{1,0}$ & double/split/hyperbolic   & $\D$ & $\D$ & $\R\oplus \R$\\
\hline
$\Cl_{0,2}$ & quaternions & $\H$ & $\R$ & $\H$\\
\hline
$\Cl_{2,0}\simeq \Cl_{1,1}$ & para/split/co-quaternions & $\P$ & $\R$ & $\Mat(2,\R)$\\
\hline
$\Cl_{0,3}$ & split biquaternions &$\D_\H$  & $\D$ & $\H\otimes_\R \D\simeq \H\oplus \H $\\
\hline
$\Cl_{3,0}$ & biquaternions  & $\C_\H$  & $\C$ & $\H\otimes_\R \C=\Mat(2,\C)$\\
\hline  
\end{tabular}
	\vskip 0.2in
	\caption{Some classical Clifford algebras and their isomorphic descriptions. }
	\label{table:ClassicalCliffordI}
\end{table}

\begin{table}[!htbp]
	\centering
	\begin{tabular}{|c|c|c|c|c|c|}
		\hline
		$\Cl(V,h)$ & $\begin{array}{c}p-q\\\!\!\mod\!8\end{array}$ & type &$\Spin(V,h)$ & $\Pin(V,h)$ & $\Pin^e(V,h)$ \\
		\hline\hline
                $\Cl_{0,1}$ & $7$ & complex & $\mG_2$& $\mG_4$ & $\Spin^c_{0,1}\!\!=\!\U(1)$ \\
		\hline
		$\Cl_{1,0}$ & $1$ & normal non-simple & $\mG_2$ & $D_4$  & $\Spin^h_{1,0}\!\!=\!\U(\D)$\\
		\hline
		$\Cl_{0,2}$ & $6$ & quat. simple & $\U(1)$ & $\Pin_{0,2}$ & $\Pin_{0,2}$\\
		\hline
		$\Cl_{2,0}$ & $2$ & normal simple & $\U(1)$ & $\Pin_{2,0}$ & $\Pin_{2,0}$ \\
		\hline
		$\Cl_{0,3}$  & $5$ &  quat. non-simple& $\Sp(1)$ & $\Pin_{0,3}$  & $\Spin^h_{3,0}\!\!=\!\Sp(1)\!\cdot\! \U(\D)$\\
		\hline
		$\Cl_{3,0}$ & $3$ & complex & $\Sp(1)$ & $\Pin_{3,0}$ & $\Spin^c_{3,0}\!\!=\!\Sp(1)\!\cdot\! \U(1)$\\
		\hline  
	\end{tabular}
	\vskip 0.2in
	\caption{The groups $\Spin(V,h)$, $\Pin(V,h)$ and $\Pin^e(V,h)$ for some classical Clifford algebras.}
	\label{table:ClassicalCliffordII}
\end{table}

\section{Enlarged spinor groups and their elementary representations} 
\label{sec:enlargedspinor}

In this section, we discuss certain enlargements of $\Spin(V,h)$ which,
together with the spin and pin groups, will arise later as canonical
models of the reduced Lipschitz group $\cL$ of irreducible real
representations of $\Cl(V,h)$ for various dimensions and
signatures. Depending on the value of $p-q\!\mod 8$, $\cL$ turns out to
be isomorphic with one of the groups $\Spin(V,h)$, $\Pin(V,h)$,
$\Spin^q(V,h)$, $\Pin^q(V,h)$ or $\Spin^o_\pm(V,h)$ discussed below,
while certain natural representations of $\cL$ are isomorphic with the
elementary representations discussed in this section. The groups
$\Spin^q(V,h)$ and their elementary representations were studied in
\cite{Nagase} for the case $p=d$, $q=0$. To our knowledge, the groups
$\Spin^o_\pm(V,h)$ and their elementary representations were not 
considered before in this context; they are studied in detail in
\cite{spino}, to which we refer the reader for more information.

\subsection{The group $\Pin^q(V,h)$ and its elementary representations}
\label{sec:pinq}
\begin{definition}
Define: 
\ben
\Pin^q(V,h)\eqdef \Pin(V,h)\cdot \Sp(1)=[\Pin(V,h)\times \Sp(1)]/\{-1,1\}~~.
\een
\end{definition}

\

\noindent Let $\Ad_\bullet:\Sp(1)=\U(\H)\rightarrow \SO(\Im\H)= \SO(3,\R)$ be the adjoint
representation\footnote{This coincides with the vector representation of $\Spin(3)\simeq \Sp(1)\simeq \SU(2)$, $\Ad_\bullet$ being the 
double covering morphism $\Spin(3)\rightarrow \SO(3,\R)$.} of $\U(\H)=\Sp(1)$:
\be
\Ad_\bullet(q)(u)=quq^{-1}~~\forall q\in \U(\H)~~,~~\forall u\in \Im \H~~,
\ee
where $\Im \H=\R^3$ is endowed with the canonical scalar product. We have: 
\be
\Ad_\bullet(q)=\Ad_{\H}(q)|_{\Im \H}~~\forall q\in \U(\H)=\Sp(1)~~,
\ee
where $\Ad_{\H}:\U(\H)\rightarrow \Aut_\Alg(\H)$ is the morphism of groups given by: 
\be
\Ad_\H(q)(q')\eqdef q q'q^{-1}~~\forall q\in \U(\H)=\Sp(1)~~\forall q'\in \H~~.
\ee
The latter gives a four-dimensional representation over $\R$ which
decomposes as a direct sum of the trivial representation (whose
underlying subspace is $\R 1_\H$) and the representation $\Ad_\bullet$
(which is supported on $\Im \H$).

\begin{definition}
The {\em vector representation of $\Pin^q(V,h)$} is the
group morphism $\lambda:\Pin^q(V,h)\rightarrow {\hat \O}(V,h)$ given by:
\be
\lambda([a,q])\eqdef \Ad_0(a)~~\forall [a,q]\in \Pin^q(V,h)~~.
\ee
The {\em twisted vector representation of $\Pin^q(V,h)$} is the
group morphism $\tlambda:\Pin^q(V,h)\rightarrow \O(V,h)$ given by:
\be
\tlambda([a,q])\eqdef \tAd_0(a)~~\forall [a,q]\in \Pin^q(V,h)~~.
\ee
The {\em characteristic representation of $\Pin^q(V,h)$} is the
group morphism $\mu:\Pin^q(V,h)\rightarrow \SO(3,\R)$ given by:
\be
\mu([a,q])\eqdef \Ad_\bullet(q)~~\forall [a,q]\in \Pin^q(V,h)~~,
\ee
The {\em basic representation of $\Pin^q(V,h)$} is the group
morphism $\rho\eqdef \lambda\times \mu:\Pin^q(V,h)\rightarrow
{\hat \O}(V,h)\times \SO(3,\R)$:
\ben
\rho([a,q])\eqdef (\Ad_0(a),\Ad_\bullet(q))~~.
\een
The {\em twisted basic representation of $\Pin^q(V,h)$} is the group
morphism $\trho\eqdef \tlambda\times \mu:\Pin^q(V,h)\rightarrow
\O(V,h)\times \SO(3,\R)$:
\ben
\trho([a,q])\eqdef (\tAd_0(a),\Ad_\bullet(q))~~.
\een
\end{definition}

\noindent We have exact sequences:

\ben
\label{Pinqseq}
1\longrightarrow \Z_2\hookrightarrow \Pin^q(V,h)\stackrel{\rho}{\longrightarrow} {\hat \O}(V,h)\times \SO(3,\R)\longrightarrow 1~~,
\een
and: 
\ben
\label{Pinqtseq}
1\longrightarrow \Z_2\hookrightarrow \Pin^q(V,h)\stackrel{\trho}{\longrightarrow} \O(V,h)\times \SO(3,\R)\longrightarrow 1~~,
\een
where $\Z_2=\{[-1,1]=[1,-1],[1,1]=[-1,-1]\}$. 

\begin{prop}
\label{prop:rhotrho}
Let $d$ be even and $\varphi:\Pin(V,h)\stackrel{\sim}{\rightarrow}
\Pin(V,-\sigma_{p,q}h)$ be the isomorphism of Proposition
\ref{prop:AdtAdPin}. Then the isomorphism of groups
$\theta:\Pin^q(V,h)\stackrel{\sim}{\rightarrow}\Pin^q(V,-\sigma_{p,q}h)$
defined through:
\be
\theta([a,q])\eqdef [\varphi(a),q]~~\forall a\in \Pin(V,h)~~\mathrm{and}~~q\in \U(\H)~~,
\ee
satisfies $\rho\circ \theta=\trho$.
\end{prop}

\proof Follows from Proposition \ref{prop:AdtAdPin}. \qed

\subsection{The group $\Spin^q(V,h)$ and its elementary representations}
\label{sec:spinq}

Define: 
\ben
\Spin^q(V,h)\eqdef \Spin(V,h)\cdot \Sp(1)=[\Spin(V,h)\times \Sp(1)]/\{-1,1\}~~.
\een
In the case $p=d,q=0$, this group was studied in \cite{Nagase}. 

\begin{definition}
The {\em vector representation} $\lambda:\Spin^q(V,h)\rightarrow
\SO(V,h)$ is the restriction of any of the vector representations of
$\Pin^q(V,h)$. The {\em characteristic representation}
$\mu:\Spin^q(V,h)\rightarrow \SO(3,\R)$ is the restriction of the
characteristic representation of $\Pin^q(V,h)$. The {\em basic
  representation} $\rho\eqdef \lambda\times
\mu:\Spin^q(V,h)\rightarrow \SO(V,h)\times \SO(3,\R)$ is the restriction
of any of the basic representations of $\Pin^q(V,h)$.
\end{definition}

\

\noindent The sequences \eqref{Pinqseq} and \eqref{Pinqtseq} restrict to the same exact sequence:
\ben
\label{Spinqseq}
1\longrightarrow \Z_2\hookrightarrow \Spin^q(V,h)\stackrel{\rho}{\longrightarrow} \SO(V,h)\times \SO(3,\R)\longrightarrow 1~~.
\een

\subsection{The group $\Spin_\alpha^o(V,h)$ and its elementary representations}
\label{sec:spino}

In this subsection, we assume that $d=p+q$ is odd. Let $\alpha\in \{-1,1\}$ be a sign factor and define: 
\be
\Pin_2(\alpha)\eqdef \twopartdef{\Pin_{2,0}}{\alpha=+1}{\Pin_{0,2}}{\alpha=-1}~~.
\ee

\begin{definition}
Define: 
\be
\Spin^o_\alpha(V,h)=\Spin(V,h)\cdot \Pin_2(\alpha)\eqdef [\Spin(V,h)\times \Pin_2(\alpha)]/\{-1,1\}~~.
\ee
\end{definition}

\

\noindent Let $\Ad_0^{(2)}:\Pin_2(\alpha)\rightarrow \O(2,\R)$ and
$\tAd_0^{(2)}:\Pin_2(\alpha)\rightarrow \O(2,\R)$ be the untwisted and
twisted vector representations of $\Pin_2(\alpha)$. In signatures
$(2,0)$ and $(0,2)$, Proposition \ref{prop:AdtAdPin} gives an
automorphism of $\Pin_2(\alpha)$ which exchanges the two
representations $\Ad_0^{(2)}$ and $\tAd_0^{(2)}$.
Notice that $\Spin(2)\simeq \SO(2,\R)\simeq \U(1)$ and that $\U(1)/\{-1,1\}\simeq \U(1)$. 

\begin{definition}
\label{def:SpinoBasicReps}
\

\begin{enumerate}[1.]
\itemsep 0.0em
\item The {\em vector representation} of $\Spin^o_\alpha(V,h)$ is the
  group morphism $\lambda:\Spin^o_\alpha(V,h)\rightarrow \O(V,h)$
  defined through:
\be
\lambda([a,u])\eqdef \det(\Ad_0^{(2)}(u))\Ad_0(a)~~\forall [a,u]\in \Spin^o_\alpha(V,h)~~.
\ee
\item The {\em characteristic representation} of $\Spin^o_\alpha(V,h)$
  is the group morphism $\mu:\Spin^o_\alpha(V,h)\rightarrow\O(2,\R)$
  defined through:
\be
\mu([a,u])\eqdef \Ad_0^{(2)}(u)~~\forall [a,u]\in \Spin^o_\alpha(V,h)~~.
\ee
\item The basic representation is the group morphism
  $\rho\eqdef \lambda\times \mu:\Spin^o_\alpha(V,h)\rightarrow
  \mathrm{S}[\O(V,h)\times \O(2,\R)]$:
\be
\rho([a,u])\eqdef (\det(\Ad_0^{(2)}(u))\Ad_0(a),\Ad_0^{(2)}(u))~~.
\ee
\end{enumerate}
\end{definition}

\

\noindent We have a short exact sequence: 
\be
1\longrightarrow \Spin(V,h) \longrightarrow \Spin^o_\alpha(V,h)\stackrel{\mu}{\longrightarrow}\O(2,\R)\longrightarrow 1~~.
\ee
Since $d$ is odd, we also have short exact sequences:
\be
1\longrightarrow \U(1) \longrightarrow \Spin^o_\alpha(V,h)\stackrel{\lambda}{\longrightarrow} \O(V,h)\longrightarrow 1
\ee
(where $\U(1)=\Spin(2)\subset \Pin_2(\alpha)\subset \Spin^o_\alpha(V,h)$) and: 
\be
1\longrightarrow \Z_2 \longrightarrow \Spin^o_\alpha(V,h)\stackrel{\rho}{\longrightarrow} \mathrm{S}[\O(V,h)\times \O(2,\R)]\longrightarrow 1~~,
\ee
where $\Z_2=\{[1,1]=[-1,-1],[1,-1]=[-1,1]\}\subset \Spin_\alpha^o(V,h)$. One can identify $\Pin_2(\alpha)$ with the abstract group $\O_2(\alpha)$ defined below.

\begin{definition}
Let $\O_2(\alpha)$ be the compact non-Abelian Lie group with
underlying set $\U(1)\times \Z_2$ and composition given by: 
\beqan
\label{Ualpha}
&& (z_1,\hat{0})(z_2, \hat{0})=(z_1z_2, \hat{0})\, , \quad (z_1,\hat{0})(z_2,\hat1)=(z_1z_2,\hat1)\, ,\nn\\
&& (z_1,\hat1)(z_2,\hat{0})=(z_1\bar{z}_2,\hat1)\, , \quad (z_1,\hat1)(z_2,\hat1)=(\alpha z_1\bar{z}_2,\hat{0})\, ,
\eeqan where $\mathbb{Z}_{2} = \left\{\hat{0},\hat{1}\right\}$. The
unit in $\O_2(\alpha)$ is given by $1\equiv (1,{\hat 0})$.
\end{definition}

\noindent The group $\U(1)$ embeds into $\O_2(\alpha)$ as the {\em
  non-central} subgroup $\U(1)\times \{{\hat 0\}}$.  The element
$\bc\eqdef (1,{\hat 1})\in \O_2(\alpha)$ satisfies
$\bc^2=(\alpha,{\hat 0})=\alpha 1$ and $\bc^{-1}=(\alpha,{\hat
  1})=\alpha \bc$. Thus $\bc$ has order two when $\alpha=1$ and order
four when $\alpha=-1$. This element generates a subgroup $\Gamma_c$ of
$\O_2(\alpha)$ which is isomorphic with $\Z_2$ when $\alpha=+1$ and
with $\mG_4$ when $\alpha=-1$. Together with $\bc$, the subgroup
$\U(1)$ generates $\O_2(\alpha)$. In fact, $\O_2(\alpha)=\{z|z\in
\U(1)\}\sqcup\{z\bc|z\in \U(1)\}$ is the group generated by $\U(1)$ and
$\bc$ with the relations $\bc^2=\alpha 1$ and $\bc z={\bar z}\bc$. 
We have:
\be
\Ad(\bc)(x)=K(x)\, , \quad \forall\, x\in \O_2(\alpha)\, ,
\ee
where $K:\O_2(\alpha)\rightarrow \O_2(\alpha)$ is the {\em
  conjugation automorphism}, given by:
\be
K(z,\kappa)=({\bar z},\kappa)\, , \quad \forall\, z\in \U(1)\, , \quad \kappa\in \Z_2\, .
\ee
Notice that $K(z)={\bar z}$ for $z\in \U(1)$, $K(\bc)=\bc$ and $K^2=\id_{\O_2(\alpha)}$. The group $\O_2(\alpha)$ admits a $\Z_2$-grading
with homogeneous components:
\be
\O_2(\alpha)^+=\{(z,{\hat 0})|z\in \U(1)\}\simeq \U(1)\, , \quad \O_2(\alpha)^-=\{(z,{\hat 1})|z\in \U(1)\}=\U(1) \bc\, .
\ee
These coincide with the connected components of $\O_2(\alpha)$. We
have $Z(\O_2(\alpha))=\{-1,1\}=\mG_2$. 

\begin{definition}
The {\em abstract determinant} is the grading morphism
$\eta_\alpha:\O_2(\alpha)\rightarrow \mG_2$ of $\O_2(\alpha)$:
\be
\eta_\alpha(x)\eqdef (-1)^{\pr_2(x)}~~,
\ee
where $\pr_2(z,\kappa)\eqdef \kappa$ for any $(z,\kappa)\in \O_2(\alpha)$. 
\end{definition}

\noindent We have $\eta_\alpha(z)=1$, $\eta_\alpha(\bc)=-1$ and a short exact sequence:
\ben
\label{ext}
1\longrightarrow \U(1) \longrightarrow \O_2(\alpha)\stackrel{\eta_\alpha}{\longrightarrow} \mG_2\longrightarrow 1~~.
\een
Moreover:
\begin{enumerate}[1.]
\itemsep 0.0em
\item For $\alpha=+1$, the sequence \eqref{ext} splits (a splitting
  morphism $\theta:\mG_2\rightarrow \O_2(\alpha)$ being given by
  $\theta(1)=1$ and $\theta(-1)=\bc$) and we have\footnote{This is the
    well-known isomorphism between $\Pin_{2,0}\simeq \O_2(+)$ and the
    orthogonal group $\O(2,\R)$.} $\O_2(+)\simeq_{\Gp} \O(2,\R)$ by an
  isomorphism which identifies $\U(1)$ with $\SO(2,\R)$ and $\bc$ with
  reflection of $\R^2$ with respect to some axis.
\item For $\alpha=-1$, the sequence \eqref{ext} presents $\O_2(-)$ as
  a {\em non-split} extension of $\Z_2$ by $\U(1)$.  In particular, we
  have $\O_2(-)\not \simeq \O(2,\R)$.
\end{enumerate}

\begin{definition}
The {\em squaring morphism} is the surjective group morphism
$\sigma_\alpha:\O_2(\alpha)\rightarrow \O_2(+)$ given by:
\be
\sigma_\alpha(z,\kappa)\eqdef (z^2,\kappa)~~,~~\forall (z,\kappa)\in \O_2(\alpha)~~.
\ee
\end{definition}

\noindent We have a short exact sequence: 
\be
1\longrightarrow \{-1,1\} \hookrightarrow \O_2(\alpha)\stackrel{\sigma_\alpha}{\longrightarrow} \O_2(+)\longrightarrow 1~~.
\ee
In particular, $\O_2(+)\simeq \Pin_{2,0}$ and $\O_2(-)\simeq
\Pin_{0,2}$ are inequivalent central extensions of $\O(2,\R)$ by
$\Z_2$. The reflection:
\ben
\label{C0}
C_0=\left[\begin{array}{cc}1 & 0 \\ 0 & -1 \end{array}\right]\in \O_-(2,\R)
\een 
of $\R^2$ with respect to the horizontal axis (=real axis of $\C\equiv \R^2$) gives isomorphisms of groups 
$\Phi_0^{(\pm)}:\O_2(+)\stackrel{\sim}{\rightarrow} \O(2,\R)$ through the formula:
\ben
\label{PhiC}
\Phi_0^{(\pm)}(e^{\bi\theta},{\hat 0})=R(\pm \theta)\, , \quad \Phi_0^{(\pm)}(e^{\bi\theta},{\hat 1})=R(\pm\theta) C_0\, ,
\een
where $\theta\in \R$ and: 
\ben
\label{rot}
R(\theta)=\left[\begin{array}{cc}\cos(\theta) & -\sin(\theta) \\ \sin(\theta) & \cos(\theta) \end{array}\right]\in \SO(2,\R)~~.
\een

\noindent Let $e_1,e_2$ be the canonical basis of $\R^2$ and
$\nu_2(\alpha)=e_1e_2$ be the Clifford volume element of
$\Cl_2(\alpha)\eqdef \twopartdef{\Cl_{2,0}}{\alpha=+1}{\Cl_{0,2}}{\alpha=-1}$ with respect
to the standard orientation of $\R^2$. The following result is proved
in \cite{spino}:

\begin{prop}
\label{prop:Pin2O2}
There exists an isomorphism of $\Z_2$-graded groups $\psi_\alpha: \O_2(\alpha)\stackrel{\sim}{\rightarrow}\Pin_2(\alpha)$ which 
satisfies: 
\be
\psi_\alpha(\bi)=\nu_2=e_1e_2~~\mathrm{and}~~\psi_\alpha(\bc)=e_1~~.
\ee
Moreover, the untwisted vector representation of $\Pin_2(\alpha)$ agrees
with the squaring morphism $\sigma_\alpha$ through the isomorphisms
$\psi_\alpha$ and $\Phi_0^{(-\alpha)}$:
\ben
\label{murep}
\Ad_0^{(2)}\circ \psi_\alpha =\Phi_0^{(-\alpha)}\circ \sigma_\alpha~~.
\een
and the abstract determinant agrees with the grading morphism
$\det\circ \Ad_0^{(2)}$ of $\Pin_2(\alpha)$:
\ben
\label{etarep}
\det\circ \Ad_0^{(2)}\circ \psi_\alpha=\eta_\alpha~~.
\een
\end{prop}

\subsection{Adapted $\Spin^o$ groups}
\label{sec:spinoadapted}

\begin{definition}
Let $(V,h)$ be a quadratic vector space belonging to the complex case,
i.e. such that the signature $(p,q)$ of $(V,h)$ satisfies $p-q\equiv_8
3,7$ (in particular, $d=p+q$ is odd). Let $\alpha_{p,q}\eqdef (-1)^{\frac{p-q+1}{4}}$. Then the {\em
  adapted $\Spin^o$ group} of $(V,h)$ is defined through:
\be
\Spin^o(V,h)\eqdef \Spin^o_{\alpha_{p,q}}(V,h)=\twopartdef{\Spin^o_-(V,h)}{p-q\equiv_8 3}{\Spin^o_+(V,h)}{p-q\equiv_8 7}~~.
\ee

\end{definition}

\subsection{The canonical spinor group and its elementary representations}
\label{sec:cangroup}

It is convenient to introduce a group depending on $p-q\!\mod 8$ which,
as we shall see later on, provides a canonical presentation of the
reduced Lipschitz group of real Clifford irreps.

\begin{definition}
The {\em canonical spinor group} $\Lambda(V,h)$ of $(V,h)$ is defined
as follows:
\begin{enumerate}[1.]
\itemsep 0.0em
\item In the normal simple case, we set $\Lambda(V,h)\eqdef
  \Pin(V,h)$.
\item In the complex case, we set $\Lambda(V,h)\eqdef
  \Spin^o(V,h)$.
\item In the quaternionic simple case, we set $\Lambda(V,h)\eqdef
  \Pin^q(V,h)$.
\item In the normal non-simple case, we set $\Lambda(V,h)\eqdef
  \Spin(V,h)$.
\item In the quaternionic non-simple case, we set $\Lambda(V,h)\eqdef
  \Spin^q(V,h)$.
\end{enumerate}
\end{definition}

\noindent The situation is summarized in Table \ref{table:cL}.

\begin{table}[H]
\centering
\begin{tabular}{|c|c|c|c|c|c|c|c|}
\hline
$\begin{array}{c} p-q\\ {\rm mod}~8 \end{array}$  & type  & $\Lambda(V,h)$   \\
\hline\hline
$0,2$ & normal  simple   & $\Pin(V,h)$ \\
\hline
$3,7$ & complex  & $\Spin^o(V,h)$   \\
\hline
$4,6$ & quaternionic simple & $\Pin^q(V,h)$ \\
\hline
$1$ & normal  non-simple & $\Spin(V,h)$   \\
\hline
$5$ & quaternionic  non-simple & $\Spin^q(V,h)$    \\
\hline 
\end{tabular}
\vskip 0.2in
\caption{Canonical spinor groups}
\label{table:cL}
\end{table}

\begin{definition}
The {\em vector representation} $\lambda$ of $\Lambda(V,h)$ is defined
as follows:
\begin{enumerate}[1.]
\itemsep 0.0em
\item In the normal simple case, $\lambda\eqdef \Ad_0^\Cl:\Pin(V,h)\rightarrow
  \O(V,h)$ is the {\em untwisted} vector representation of
  $\Pin(V,h)$.
\item In the complex case,
  $\lambda:\Spin^o(V,h)\rightarrow \O(V,h)$ is the
  vector representation of $\Spin^o_{\alpha_{p,q}}(V,h)$.
\item In the quaternionic simple case, $\lambda:\Pin^q(V,h)\rightarrow
  \O(V,h)$ is the {\em untwisted} vector representation of $\Pin^q(V,h)$.
\item In the normal non-simple case, $\lambda\eqdef \Ad_0^\Cl:\Spin(V,h)\rightarrow
  \SO(V,h)$ is the vector representation of $\Spin(V,h)$.
\item In the quaternionic non-simple case,
  $\lambda:\Spin^q(V,h)\rightarrow \SO(V,h)$ is the vector
  representation of $\Spin^q(V,h)$.
\end{enumerate}
\end{definition}

\begin{definition}
The {\em characteristic representation} $\mu$ of $\Lambda(V,h)$ is
defined as follows:
\begin{enumerate}[1.]
\itemsep 0.0em
\item In the normal simple case, $\mu:\Pin(V,h)\rightarrow 1$ is the
  trivial one-dimensional representation.
\item In the complex case, $\mu:\Spin^o(V,h)\rightarrow
  \O(2,\R)$ is the characteristic representation of
  $\Spin^o_{\alpha_{p,q}}(V,h)$.
\item In the quaternionic simple case, $\mu:\Pin^q(V,h)\rightarrow
  \SO(3,\R)$ is the characteristic representation of $\Pin^q(V,h)$.
\item In the normal non-simple case, $\mu:\Spin(V,h)\rightarrow 1$
  is the trivial one-dimensional representation.
\item In the quaternionic non-simple case,
  $\mu:\Spin^q(V,h)\rightarrow \SO(3,\R)$ is the characteristic
  representation of $\Spin^q(V,h)$.
\end{enumerate}
\end{definition}

\begin{definition}
The {\em basic representation} of $\Lambda(V,h)$ is the representation
$\rho=\lambda\times \mu$, namely:
\begin{enumerate}[1.]
\itemsep 0.0em
\item In the normal simple case, $\rho=\lambda:\Pin(V,h)\rightarrow \O(V,h)$
  is the {\em untwisted} vector representation of $\Pin(V,h)$.
\item In the complex case,
  $\rho:\Spin^o(V,h)\rightarrow \mathrm{S}[\O(V,h)\times \O(2,\R)]$ is
  the basic representation of $\Spin^o_{\alpha_{p,q}}(V,h)$.
\item In the quaternionic simple case, $\rho:\Pin^q(V,h)\rightarrow
  \O(V,h)\times \SO(3,\R)$ is the {\em untwisted} basic representation of $\Pin^q(V,h)$.
\item In the normal non-simple case, $\rho:\Spin(V,h)\rightarrow
  \SO(V,h)$ is the vector representation of $\Spin(V,h)$.
\item In the quaternionic non-simple case,
  $\rho:\Spin^q(V,h)\rightarrow \SO(V,h)\times \SO(3,\R)$ is the basic
  representation of $\Spin^q(V,h)$.
\end{enumerate}
\end{definition}

\section{Clifford representations over $\R$}
\label{sec:Cliffordrep}

In this section, we discuss finite-dimensional real Clifford
representations, the notion of weak faithfulness for such
representations as well as certain natural subspaces associated to
them. Real Lipschitz groups (to be introduced in the next section)
arise most naturally as the automorphism groups of weakly faithful
real Clifford representations in a certain category which has more
morphisms than the usual category of representations. Therefore, we
start by introducing that category.

\subsection{The unbased category of Clifford representations over $\R$}
\label{subsec:ClRep}

\begin{definition}
A {\em Clifford representation over $\R$} is a morphism of unital
algebras $\gamma: \Cl(V,h)\rightarrow \End_\R(S)$, where $S$ is a
finite-dimensional $\R$-vector space.
\end{definition}

\noindent Let $\gamma:\Cl(V,h)\rightarrow \End_\R(S)$ and
$\gamma':\Cl(V',h')\rightarrow \End_\R(S')$ be two Clifford
representations. 

\begin{definition}
\label{DefClMf}
A {\em morphism} from $\gamma$ to $\gamma'$ is a pair $(f_0,f)$ such
that:
\begin{enumerate}
\item $f_0:V\rightarrow V'$ is an isometry from $(V,h)$ to $(V',h')$
\item $f:S\rightarrow S'$ is an $\R$-linear map
\item $\gamma'(\Cl(f_0)(x))\circ f=f\circ \gamma(x)$ for all $x\in
  \Cl(V,h)$.
\end{enumerate}
A {\em based morphism} is a morphism $(f_0,f)$ such that $f_0=\id_V$.
A (not necessarily based) isomorphism from $\gamma$ to itself is
called an {\em automorphism}.
\end{definition}

In our language, a morphism of representations in the traditional
sense corresponds to a {\em based} morphism. Since $\Cl(V,h)$ is
generated by $V$ while $\Cl(V',h')$ is generated by $V'$, condition
(3) is equivalent with the weaker requirement:
\be
\gamma'(f_0(v))\circ f=f\circ \gamma(v)~~~~\forall v\in V~~,
\ee
and can also be written as (see the diagram in Figure \ref{fig:repdiagram}): 
\be
R_f\circ \gamma'\circ f_0=L_f \circ \gamma|_V~~\mathrm{or}~~R_f\circ \gamma'\circ \Cl(f_0)=L_f \circ \gamma~~,
\ee
where $L_f:\End_\R(S)\rightarrow \Hom_\R(S,S')$ and $R_f:\End_\R(S')\rightarrow \Hom_\R(S,S')$ are defined through:
\be
L_f(\varphi)\eqdef f\circ \varphi~~,~~R_f(\psi)=\psi\circ f~~,~~\forall \varphi\in \End_\R(S)~~\mathrm{and}~~\psi\in \End_\R(S')~~.
\ee

\begin{figure}[H]
\label{fig:repdiagram}
\ben
\scalebox{1.2}{
\xymatrix{
\Cl(V',h')~  \ar[rr]^{\gamma'}
&& ~\End_\R(S') ~ \ar[d]^{R_f} \\
&& \Hom_\R(S,S')\\
\Cl(V,h) \ar[uu]^{\Cl(f_0)} \ar[rr]_{\gamma} && \End_\R(S) \ar[u]^{L_f} 
}}
\een
\caption{Commutativity of the diagram is the condition that $(f_0,f)$ gives an unbased morphism of Clifford representations.}
\end{figure}

\noindent With this definition, Clifford representations form a
category denoted $\ClRep$, were compatible morphisms $(f_0,f)$ and
$(f_0',f')$ compose pairwise, i.e.  $(f_0',f')\circ (f_0,f)\eqdef
(f_0'\circ f_0,f'\circ f)$. The forgetful functor
$\Pi:\ClRep\rightarrow Cl$ which takes $\gamma$ into $\Cl(V,h)$ and
$(f_0,f)$ into $\Cl(f_0)$ is a fibration whose fiber above $\Cl(V,h)$
is the usual category $\Rep(\Cl(V,h))$ of representations of
$\Cl(V,h)$ (whose morphisms are the based morphisms of
representations). Isomorphisms in $\Rep(\Cl(V,h))$ are the usual equivalences of
representations. Hence equivalences of representations of Clifford
algebras coincide with based isomorphisms of $\ClRep$; in particular,
any isomorphism class of Clifford representations in the category 
$\ClRep$ decomposes as a disjoint union of equivalence classes. The
surjective functor $F=\Cl^{-1}\circ \Pi:\ClRep\rightarrow \Quad$ sends
$\Cl(V,h)$ to $(V,h)$ and $(f_0,f)$ to $f_0$. A morphism $(f_0,f)$ is
an isomorphism in $\ClRep$ iff both $f_0$ and $f$ are bijective.

\begin{prop}
\label{prop:Ad}
Let $(f_0,f):\gamma\stackrel{\sim}{\rightarrow} \gamma'$ be an isomorphism of Clifford
representations. Then $\Ad(f)(\gamma(V))=\gamma'(V')$ and
$\gamma'\circ f_0=\Ad(f)\circ \gamma|_V$, where the unital isomorphism
of algebras $\Ad(f):\End_\R(S)\rightarrow \End_\R(S')$ is defined
through:
\be
\Ad(f)(\varphi)\eqdef f\circ \varphi \circ f^{-1}
\ee
for all $\varphi\in \End_\R(S)$. 
\end{prop}

\proof Follows immediately from the fact that $\Cl(f_0)|_V=f_0$. \qed

\

\noindent When $(f_0,f)$ is an isomorphism, condition 3. in Definition
\ref{DefClMf} becomes:
\ben
\label{f0f}
\gamma'\circ \Cl(f_0) =\Ad(f)\circ \gamma~~,
\een
being equivalent with the condition $\gamma'\circ f_0=\Ad(f)\circ \gamma|_V$, which states that 
$f$ {\em implements} the isometry $f_0:(V,h)\rightarrow (V',h')$ at the level of the representation spaces. 

\subsection{Clifford image, vector image and Schur algebra}

\begin{definition}
The {\em Clifford image} of $\gamma$ is the unital subalgebra $C(\gamma)\eqdef
\gamma(\Cl(V,h))\subset \End_\R(S)$. The {\em vector image} is the
subspace $W(\gamma)\eqdef \gamma(V)\subset C(\gamma)$.
\end{definition}

Notice that the Clifford image coincides with the unital subalgebra of
$\End_\R(S)$ generated by the vector image and hence $C(\gamma)$ is
uniquely determined by $W(\gamma)$. Every element $w\in W(\gamma)$ can
be written as $w=\gamma(v)$ for some $v\in V$, but this presentation
need not be unique\footnote{It is unique when $\gamma$ is weakly faithful, see below.}.

\begin{definition}
The {\em Schur algebra} $\S(\gamma)$ is the centralizer algebra of
$C(\gamma)$ inside the algebra $(\End_\R(S),\circ)$. It consists of
all $\R$-linear endomorphisms of $S$ which commute with every element
of $C(\gamma)$.
\end{definition}
Since $C(\gamma)$ is generated by $W(\gamma)$, we have: 
\begin{eqnarray}
\S(\gamma) &=& \{a\in \End_\R(S)|a w=wa~~~\forall w\in W(\gamma)\}\nonumber\\ &=&\{a\in \End_\R(S)|a\gamma(v)=\gamma(v)a~~\forall v\in V\}\, .
\end{eqnarray}
Since $\S:=\S(\gamma)$ is a unital subalgebra of $\End_\R(S)$, we can
view $S$ as a left $\S$-module upon defining the left multiplication with
scalars through:
\be
s \xi\eqdef s(\xi)~~\forall s\in \S~~,~~\xi\in S~~,~~
\ee
where $s(\xi)$ denotes the result of applying the operator $s$ to
$\xi$.  Also notice that $\End_\R(S)$ is an $\S$-bimodule (with
external multiplications given by composition from the left and from
the right with operators belonging to $\S$).  To simplify notation, we
will often denote composition of operators by juxtaposition. Let:
\begin{eqnarray}
\End_\S(S) &\eqdef &\{a\in \End_\R(S)|a(s\xi)=sa(\xi)~~\forall \xi\in S~~\mathrm{and}~~s\in\S\}\nonumber\\ &=&\{a\in \End_\R(S)|as=sa~~\forall s\in \S\}\, ,
\end{eqnarray}
denote the unital algebra of $\S$-linear maps (endomorphisms of $S$ as
an $\S$-module) and $\Aut_\S(S)$ denote the group of invertible $\S$-linear
maps ($\S$-module automorphisms). By the definition of $C(\gamma)$, we
have:
\be
C(\gamma) \subset \End_{\S}(S)~~.
\ee

\begin{remark}
Recall that $\Cl(V,h)$ is a semisimple $\R$-algebra. This implies that
any finite-dimensional $\Cl(V,h)$-module (i.e. any Clifford
representation $\gamma:\Cl(V,h)\rightarrow \End_\R(S)$) decomposes as
a direct sum of inequivalent irreps
$\gamma_i:\Cl(V,h)\rightarrow \End_\R(S_i)$ ($i=1\ldots n$) with
multiplicities $m_i$, namely $S=\oplus_{i=1}^n{S_i\otimes_\R U_i}$ with $U_i=\R^{m_i}$ and 
$\gamma(x)=\oplus_{i=1}^n 
\gamma_i(x)\otimes_\R \id_{U_i}$. Schur's lemma implies the decomposition:
\ben
\label{SchurDec}
\S(\gamma)=\End_{\Cl(V,h)}(S)=\oplus_{i=1}^n \S(\gamma_i)\otimes_\R \End_\R(U_i)~~,
\een
where $\S(\gamma_i)$ are division algebras over $\R$, hence each
$\S(\gamma_i)$ is isomorphic with one of the $\R$-algebras $\R$, $\C$ or
$\H$.
\end{remark}

\subsection{Pseudocentralizer, anticommutant subspace and Schur pairing}
\label{subsec:pseudocentralizer}

\begin{definition}
The {\em anticommutant subspace} of the Clifford representation
$\gamma:\Cl(V,h)\rightarrow \End_\R(S)$ is the following subspace of
$\End_\R(S)$:
\begin{eqnarray}
A(\gamma) &\eqdef &\{a\in \End_\R(S)|aw=-wa~~\forall w\in W(\gamma)\} \nonumber\\ 
&=& \{a\in \End_\R(S)|a\gamma(v)=-\gamma(v)a~~\forall v\in V\}\nonumber \\
&=& \{a\in \End_\R(S)|a\gamma(x)=\gamma(\pi(x))a~~\forall x\in \Cl(V,h)\}\, .
\end{eqnarray}
\end{definition}
Let $A:=A(\gamma)$. We have $\S A\subset A$ and $A\S\subset A$, so $A$ is an
$\S$-bimodule (a submodule of the $\S$-bimodule $\End_\R(S)$). For any
$a_1,a_2\in A$, we have $a_1a_2\in \S$, so the composition of
$\End_\R(S)$ induces an $\R$-bilinear map:
\ben
\label{m}
m:A\times A\rightarrow \S~~. 
\een
In particular, $a\in A$ implies $a^2\in \S$, so $A$ consists of square
roots of elements of $\S$ (taken in $\End_\R(S)$). These observations
imply that the subspace $\S+A$ is a unital subalgebra of $\End_\R(S)$.

\begin{definition}
The {\em pseudocentralizer} of $\gamma$ is the unital subalgebra $\T(\gamma)\eqdef \S(\gamma)+A(\gamma)$ of $\End_\R(S)$.
\end{definition}

\begin{definition}
The {\em Schur pairing} of $\gamma$ is the symmetric $\R$-bilinear
pairing $\fp_\gamma:A(\gamma)\times A(\gamma)\rightarrow \S(\gamma)$
obtained by symmetrizing the composition map \eqref{m}:
\be
\fp_\gamma(a_1,a_2)\eqdef \frac{1}{2}(a_1a_2+a_2a_1)~~,~~\forall a_1,a_2\in A(\gamma)~~.
\ee
\end{definition}

\noindent Notice that $\S^\times \cap A=\emptyset$ if $W(\gamma)\neq 0$. Hence:
\be
\S\cap A\subset \S\setminus \S^\times~~\mathrm{if}~~W(\gamma)\neq 0~~.
\ee
In particular, we have $\S\cap A=\{0\}$ when $W(\gamma)\neq 0$ and
$\S$ is a division algebra (for example, when $\gamma$ is an
irreducible representation with $S\neq 0$).

\

\begin{remark}
Let $(f_0,f):\gamma\stackrel{\sim}{\rightarrow} \gamma'$ be an
isomorphism of Clifford representations. Then
$\Ad(f)(W(\gamma))=W(\gamma')$. Since $\Ad(f)$ is a unital
$\R$-algebra isomorphism from $\End_\R(S)$ to $\End_\R(S')$, this
implies $\Ad(f)(C(\gamma))=C(\gamma')$ and:
\be
\Ad(f)(\S(\gamma))=\S(\gamma')~~,~~\Ad(f)(A(\gamma))= A(\gamma')~~,~~\Ad(f)(\T(\gamma))=\T(\gamma')~~.
\ee
Hence restriction gives bijective maps: 
\beqan
&& \Ad_s(f)\eqdef\Ad(f)|_{\S(\gamma)}:\S(\gamma)\rightarrow \S(\gamma')\nn\\
&&\Ad_A(f)\eqdef \Ad(f)|_{A(\gamma)}:A(\gamma)\rightarrow A(\gamma)\\
&&\Ad_\T(f)\eqdef \Ad(f)|_{T(\gamma)}:\T(\gamma)\rightarrow \T(\gamma)\nn~~.
\eeqan
Notice that $\Ad_s(f)$ and $\Ad_\T(f)$ are unital isomorphisms of
$\R$-algebras while $(\Ad_s(f),\Ad_A(f))$ is a twisted isomorphism of
left modules (see Appendix \ref{app:twisted}) from the
$\S(\gamma)$-module $A(\gamma)$ to the $\S(\gamma')$-module
$A(\gamma')$. Since these maps behave well under composition, this
gives functors $\S:\ClRep^\times\rightarrow \Alg^\times$,
$A:\ClRep^\times\rightarrow \TwMod^\times$ (where $\TwMod$ is the
category of left modules over unital associative $\R$-algebras and
twisted morphisms between such) and $\T:\ClRep^\times \rightarrow
\Alg^\times$ which respectively associate to $\gamma$ the objects
$\S(\gamma)$, $A(\gamma)$ and $\T(\gamma)$ and to $f$ the morphisms
$\S(f)=\Ad_s(f)$, $A(f)=(\Ad_s(f), \Ad_A(f))$ and $\T(f)=\Ad_\T(f)$.
When $\gamma'=\gamma$ and $(f_0,f)\in \Aut(\gamma)$, we have
$\Ad_s(f)\in\Aut_{\Alg}(\S(\gamma))$, $\Ad_A(f)\in
\Aut_{\S(\gamma)}^\tw(A(\gamma))$ and $\Ad_\T(f)\in
\Aut_{\Alg}(\T(\gamma))$. The fact that $\Ad_\T(f)$ is an algebra
automorphism implies $\fp_\gamma\circ (\Ad_A(f)\otimes_\R
\Ad_A(f))=\Ad_s(f)\circ \fp_\gamma$ and hence $\Ad_A(f)\in
\Aut_{\S(\gamma)}^\tw(A(\gamma),\fp_\gamma)$ (see Appendix
\ref{app:twisted} for notation).
\end{remark}

\subsection{The pinor volume element and the volume grading of $\End_\R(S)$}
\label{subsec:EndVolGrading}

Let us fix an orientation of $V$ and let $\nu$ denote the
corresponding Clifford volume element.

\begin{definition}
The element $\omega_\gamma\eqdef \gamma(\nu)\in
C(\gamma)\subset \End_\R(S)$ is called the {\em pinor volume element}
of $\gamma$ determined by the chosen orientation of $V$.
\end{definition}

\noindent The pinor volume element $\omega:=\omega_\gamma$ satisfies:
\ben
\label{omegarels}
\omega^2=(-1)^{q+\frac{d(d-1)}{2}}\id_S~~,~~\omega w=(-1)^{d-1}w\omega~~\forall w\in W~~.
\een
Thus $\omega \in \S$ for odd $d$ and $\omega\in A$ for even $d$. The
first relation implies that $\Ad(\omega)$ is an involutive
automorphism of the algebra $\End_\R(S)$:
\be
\Ad(\omega)^2=\id_{\End_\R(S)}
\ee
and hence induces a $\Z_2$-grading $\End_\R(S)=\End_\R^0(S)\oplus \End_\R^{1}(S)$ of this algebra, where: 
\beqa
&&\End_\R^0(S)\eqdef \ker(\Ad(\omega)-\id_{\End_\R(S)})=\{a\in \End_\R(S)|a\omega=\omega a\}\nn\\
&& \End_\R^1(S)\eqdef \ker(\Ad(\omega)+\id_{\End_\R(S)})=\{a\in \End_\R(S)|a\omega=-\omega a\}~~.
\eeqa
In particular, $\End_\R^0(S)$ is a unital subalgebra of $\End_\R(S)$. 
Since $\omega=\gamma(\nu)$, we have: 
\be
\Ad(\omega)\circ \gamma=\gamma\circ \Ad(\nu)~~,
\ee
which gives: 

\begin{prop}
\label{prop:PinVol}
The morphism of $\R$-algebras $\gamma:\Cl(V,h)\rightarrow \End_\R(S)$
is homogeneous of degree zero with respect to the $\Z_2$-gradings
induced by $\nu$ on $\Cl(V,h)$ and by $\omega$ on $\End_\R(S)$:
\be
\gamma(\Cl^\kappa(V,h))\subset \End_\R^\kappa(S)~~\forall \kappa\in \Z_2~~.
\ee
\end{prop}

\

\noindent When $d$ is even, we have $\Cl^0(V,h)=\Cl_+(V,h)$ and
$\Cl^1(V,h)=\Cl_-(V,h)$ and the proposition gives $C_+(\gamma)\eqdef
\gamma(\Cl_+(V,h))\subset \End_\R^0(S)$ and $C_-(\gamma)\eqdef
\gamma(\Cl_-(V,h))\subset \End_\R^1(S)$.  In this case, we have
$C(\gamma)=C_+(\gamma)\oplus C_-(\gamma)$ and $C(\gamma)$ is a
homogeneous subalgebra of $\End_\R(S)$ with respect to the
$\Z_2$-grading introduced above. When $d$ is odd, we have
$\Cl(V,h)=\Cl^0(V,h)$ and the proposition gives
$C(\gamma)\subset \End_\R^0(S)$.

\subsection{Weakly faithful and rigid Clifford representations}

\begin{definition}
A Clifford representation $\gamma:\Cl(V,h)\rightarrow \End_\R(S)$ is
called {\em weakly faithful} if the restriction $\gamma_0\eqdef
\gamma|_V:V\rightarrow \End_\R(S)$ is an injective map. It is called {\em rigid} 
if $\gamma_V=\id_V$. 
\end{definition}

When $\gamma$ is weakly faithful, we can use the injection $\gamma|_V$
to identify $V$ with the subspace $W(\gamma)\subset \End_\R(S)$. Let
$\ClRep_w$ (respectively $\ClRep_r$) denote the full sub-categories of
$\ClRep$ whose objects are the weakly faithful (respectively rigid)
Clifford representations and $\ClRep_w^\times$ (respectively
$\ClRep_r^\times$) denote the corresponding unit groupoids (which are
full sub-groupoids of $\ClRep^\times$). Notice that any faithful or
rigid Clifford representation is weakly faithful. When $\gamma$ and
$\gamma'$ are weakly faithful and $(f_0,f):\gamma\rightarrow \gamma'$
is an {\em isomorphism} of Clifford representations, Proposition
\ref{prop:Ad} shows that we have $\Ad(f)(\gamma(V))=\gamma'(V')$ and
that $f_0$ is uniquely determined by $f$ through the relation:
\ben
\label{AdRelation}
f_0=(\gamma'|_{V'})^{-1}\circ \Ad(f)\circ \gamma|_V~~.  \een
It is easy to see that the converse also holds, so we have:
\begin{prop}
\label{prop:AdRelation}
Assume that $\gamma$ and $\gamma'$ are weakly faithful. Then any
isomorphism $(f_0,f):\gamma\rightarrow \gamma'$ is determined by the
linear isomorphism $f: S\rightarrow S'$. Namely, we have
$\Ad(f)(\gamma(V))=\gamma'(V')$ and $f_0$ is given by relation
\eqref{AdRelation}. Conversely, any linear isomorphism $f:S\rightarrow
S'$ which satisfies $\Ad(f)(\gamma(V))=\gamma'(V')$ determines an
isomorphism of quadratic spaces $f_0:(V,h)\rightarrow (V',h')$ through
relation \eqref{AdRelation} and we have $(f_0,f)\in
\Hom_{\ClRep^\times}(\gamma,\gamma')$.
\end{prop}
\noindent In view of this, we denote isomorphisms of weakly faithful
Clifford representations only by $f$ (since $f$ determines $f_0$ in
this case) and we identify $\Hom_{\ClRep_w^\times}(\gamma,\gamma')$
with a subset of the set $\Isom_\R(S,S')$ of linear isomorphisms from
$S$ to $S'$:
\beqan
\label{isomwf}
\Hom_{\ClRep_w^\times}(\gamma,\gamma') &\equiv& \{f\in \Isom_\R(S,S')|\Ad(f)(\gamma(V))=\gamma'(V')\}~\nn\\
&& \mathrm{for}~\gamma,\gamma'\in \Ob (\ClRep_w)~~.
\eeqan

\begin{remark}
One can show that the full inclusion functor
$\ClRep_r^\times\hookrightarrow\ClRep_w^\times$ is an equivalence of
categories. This equivalence of categories can be used to show
equivalence of our approach with the real version of the formalism of
``spin spaces'' which was used in \cite{FriedrichTrautman} for the
case of complex Clifford representations.
\end{remark}

\section{Real Lipschitz groups and their elementary representations}
\label{sec:elementaryrep}

In this section, we introduce the real Lipschitz group of a weakly
faithful real Clifford representation, which coincides with the
automorphism group of the latter in the category $\ClRep$ introduced
in the previous section. We also study certain elementary
representations of real Lipschitz groups. The results of this section
apply to any weakly faithful real Clifford representation $\gamma$, which 
need not be irreducible.

\subsection{The Lipschitz group of a weakly faithful real Clifford representation}

When $\gamma:\Cl(V,h)\rightarrow \End_\R(S)$ is a weakly faithful
real Clifford representation, restriction gives a linear isomorphism:
\be
\gamma|_V:V\stackrel{\sim}{\rightarrow} W(\gamma)~~
\ee
which can be used to transport $h$ to a symmetric and non-degenerate
bilinear form $g_\gamma:W(\gamma)\times W(\gamma)\rightarrow \R$:
\be
g_\gamma(w_1,w_2)\eqdef h((\gamma|_V)^{-1}(w_1),(\gamma|_V)^{-1}(w_2))~~\forall w_1,w_2\in W(\gamma)~~.
\ee
Thus $(W(\gamma),g_\gamma)$ is a quadratic space and
$\gamma|_V:(V,h)\stackrel{\sim}{\rightarrow} (W(\gamma),g_\gamma)$ is
an invertible isometry.

\begin{definition}
Let $\gamma$ be a weakly faithful real Clifford representation. The group
$\L(\gamma)\eqdef \{a\in \Aut_\R(S)|\Ad(a)(W(\gamma))\subset
W(\gamma)\}$ is called the {\em real Lipschitz group} of $\gamma$.
\end{definition}

\noindent For simplicity, we will often denote $W(\gamma)$ by
$W$, $g_\gamma$ by $g$, $\L(\gamma)$ by $\L$ etc.

\subsection{The vector representation of the Lipschitz group}

\noindent For any $a\in \L$ and any $w\in W$, we have $\Ad(a)(w)\in W$ and
$\Ad(a)(w)^2=\Ad(a)(w^2)$. Since $w^2=g(w,w)\id_S$ and
$\Ad(a)(w)^2=g(\Ad(a)(w),\Ad(a)(w))\id_S$, this implies $\Ad(a)|_W\in
\O(W,g)$.

\begin{definition}
The group morphism $\Ad_0^\gamma: \L(\gamma)\rightarrow \O(W(\gamma),g_\gamma)$ given by:
\be
\Ad_0^\gamma(a)\eqdef \Ad(a)|_W
\ee 
is called the {\em vector representation} of $\L(\gamma)$. 
\end{definition}

\

\noindent It is easy to see that the group $\S(\gamma)^\times$ of units of
$\S(\gamma)$ coincides with the kernel of $\Ad_0^\gamma$:
\ben
\label{kerAd0}
\ker(\Ad_0^\gamma)=\S(\gamma)^\times\subset \L~~.
\een

\begin{prop}
Let $\gamma$ be a weakly faithful real Clifford representation.  Then
the Lipschitz group $\L(\gamma)$ is naturally isomorphic with the
automorphism group $\Aut_{\ClRep_w}(\gamma)$ of $\gamma$ in the
category $\ClRep_w$. Therefore, the isomorphism class of
$\L(\gamma)$ depends only on the isomorphism class of $\gamma$ in the
category $\ClRep_w$.
\end{prop}

\begin{proof}
Any $a\in \L$ induces an invertible isometry $a_0\in \O(V,h)$ through
relation \eqref{AdRelation}, namely:
\ben
\label{a0}
a_0=(\gamma|_V)^{-1}\circ \Ad_0^\gamma(a)\circ (\gamma|_V)\in \O(V,h)~~,
\een
which implies:
\ben
\label{AdCl}
\gamma \circ \Cl(a_0) =\Ad(a)\circ \gamma~~.
\een
Thus $(a_0,a)$ is the unique automorphism of $\gamma$ in the category
$\ClRep$ whose second component equals $a$. Conversely, we have $a\in
\L(\gamma)$ for any $(a_0,a)\in
\Aut_{\ClRep_w}(\gamma)=\Aut_{\ClRep}(\gamma)$ (see Proposition
\ref{prop:Ad}) and $a_0$ is determined by $a$ through relation
\eqref{a0} (see Proposition \ref{prop:AdRelation}). Hence the map
$F_1:\Aut_{\ClRep}(\gamma)\rightarrow \L(\gamma)$ given by
$F_1(a_0,a)=a$ is an isomorphism of groups which allows us to identify
$\L(\gamma)$ with $\Aut_{\ClRep}(\gamma)$. \qed
\end{proof}

\noindent For what follows, we fix a weakly faithful Clifford
representation $\gamma:\Cl(V,h)\rightarrow \End_\R(S)$ and set
$\L:=\L(\gamma)$, $C:= C(\gamma)$, $W:=W(\gamma)$, $g:=g_\gamma$,
$\S:=\S(\gamma)$, $A:=A(\gamma)$ and $\Ad_0:=\Ad_0^\gamma$. Fixing an
orientation on $V$, we orient $W$ such that
$\gamma|_V:V\stackrel{\sim}{\rightarrow} W$ is orientation-preserving
and let $\nu$ and $\omega=\gamma(\nu)$ denote the corresponding
Clifford volume and pinor volume element. Since the quadratic spaces
$(V,h)$ and $(W,g)$ are isometric, we will sometimes identify them
using the isometry $\gamma|_V$, in which case we view the vector
representation of $\L$ as a group morphism $\Ad_0:\L\rightarrow
\O(V,h)$.

\

\begin{remark}
An element $w\in W$ is invertible in $\End_\R(S)$ iff it is
nondegenerate in the quadratic space $(W,g)$. In this case, its
inverse $w^{-1}=\frac{1}{g(w,w)} w$ also belongs to $W$. If $w\in W$
is a unit vector, then so is $w^{-1}$.
\end{remark}

\subsection{Volume grading of the Lipschitz group}
\label{subsec:LipVolGrading}

Let $\det:\O(W,g)\rightarrow \mG_2$ be the determinant morphism and
$\L=\L^0\sqcup \L^1$ be the $\Z_2$-grading of $\L$ induced by the
group morphism $\det\circ \Ad_0:\L\rightarrow \mG_2$:
\beqa
&& \L^0\eqdef \{a\in \L|\det(\Ad_0(a))=+1\}=\Ad_0^{-1}(\SO(W,g))~~,~~\nn\\
&& \L^1\eqdef \{a\in \L|\det(\Ad_0(a))=-1\}=\Ad_0^{-1}(\O_-(W,g))~~.
\eeqa

\begin{definition}
The subgroup $\L^0\subset \L$ is called the {\em special Lipschitz group}.
\end{definition}

\begin{prop}
\label{Lgrading}
We have: 
\ben
\label{VolL}
\Ad(\omega)(a)=\det(\Ad_0(a))a~~,~~\forall a\in \L
\een
and:
\ben
\label{L01}
\L^0=\L\cap \End^0_\R(S)~~,~~\L^1=\L\cap \End^1_\R(S)~~.
\een
\end{prop}

\proof If $e_1\ldots e_d$ is an oriented orthonormal basis of $(W,g)$,
then: 
\be
\epsilon_1\eqdef (\gamma|_V)^{-1}(e_1),\ldots, \epsilon_d\eqdef
(\gamma|_V)^{-1}(e_d)\, ,
\ee 
is an oriented orthonormal basis of
$(V,h)$. Thus $\nu=\epsilon_1\ldots \epsilon_d$ and hence
$\omega=\gamma(\nu)=e_1\ldots e_d$. For any $a\in \L$, we have:
\be
\omega'\eqdef \Ad(a)(\omega)=\Ad_0(a)(e_1)\ldots
\Ad_0(a)(e_d)=e'_1\ldots e'_d~~, 
\ee
where $e'_k\eqdef \Ad_0(a)(e_k)\in W$
form an orthonormal basis of $(W,g)$ (because $\Ad_0(a)\in
\O(W,g)$). Thus $\epsilon'_1\eqdef (\gamma|_V)^{-1}(e'_1),\ldots,
\epsilon'_d\eqdef (\gamma|_V)^{-1}(e'_d)$ is an orthonormal basis of
$(V,h)$ which satisfies $\epsilon'_k=a_0(\epsilon_k)$, where $a_0\in
\O(V,h)$ is given by \eqref{a0}. Hence $\nu'\eqdef \epsilon'_1\ldots
\epsilon'_d=\det(a_0)\epsilon_1\ldots \epsilon_d=\det(a_0)\nu$, which
implies $\omega'=\gamma(\nu')=\det(a_0) \omega$. Relation \eqref{a0}
gives $\det\Ad_0(a)=\det (a_0)\in \{-1,1\}$ and hence
$\omega'=\det(\Ad_0(a))\omega$. Thus
$\Ad(a)(\omega)=\det(\Ad_0(a))\omega$, which is equivalent with
\eqref{VolL} upon using the relation $\Ad_{0}(a) = \Ad_{0}(a^{-1}) \, \forall a \in L$. Hence:
\ben
\label{AdVol}
\Ad(\omega)(a)=(-1)^\alpha a~~ \forall a\in \L^\alpha~~(\alpha\in \Z_2)~~.
\een
and hence $\L^\alpha\subset \End_\R^\alpha(S)$ for all $\alpha\in \Z_2$.  Since
$\L=\L^0\sqcup \L^1$ and $\End_\R^0(S)\cap \End_\R^1(S)=\{0\}$ while
$0\not \in \L$, this implies \eqref{L01}.
\qed

\

\begin{remark}
Notice that $\omega\in \L$ since $\Ad(\omega)(w)=(-1)^{d-1}w\in W$ for
all $w\in W$. Since $\Ad(\omega)(\omega)=+\omega$, we have
$\omega\in \End_\R^0(S)$ and relations \eqref{L01} imply $\omega\in
\L^0$.
\end{remark}

\begin{prop}
\label{prop:mreflection}
Let $w\in W$ be a non-degenerate vector. Then $w\in \L$ and $\Ad_0(w)$
equals {\em minus} the orthogonal reflection of $(W,g)$ in the hyperplane
orthogonal to $w$:
\ben
\label{mreflection}
\Ad_0(w)=-R_w
\een
\end{prop} 

\proof Follows by direct computation using the relations
$w^{-1}=\frac{1}{g(w,w)} w$ and $wx+xw=2g(x,w)$ for $x\in W$ as well as 
relation \eqref{reflection}.
\qed

\

\begin{remark}
Let $w\in W$ be non-degenerate. Since $\Ad(\omega)(w)=(-1)^{d-1}w$,
relations \eqref{L01} give:
\begin{enumerate}[1.]
\itemsep 0.0em
\item When $d$ is even, then $w\in \L^1$
\item When $d$ is odd, then $w\in \L^0$.
\end{enumerate}
\end{remark}

\subsection{Image subgroups of the Lipschitz group}

\begin{prop}
We have $\gamma(\G^e(V,h))\subset \L$, where $\G^e(V,h)$ is the extended Clifford group. 
\end{prop}

\begin{proof}
Since $\gamma$ is a unital morphism of $\R$-algebras, we have
$\gamma(\Cl(V,h)^\times)\subset C^\times\subset \Aut_\R(S)$. Moreover,
we have:
\ben
\label{gammaAd}
\Ad(\gamma(x))(\gamma(y))=\gamma(\Ad^\Cl(x)(y))~~\forall x\in \Cl(V,h)^\times~~\mathrm{and}~~y\in \Cl(V,h)~~,
\een
which also reads: 
\be
R_\gamma\circ \Ad\circ \gamma=L_\gamma\circ \Ad^\Cl~~.
\ee
Since $\gamma|_V:V\stackrel{\sim}{\rightarrow} W$ is a bijection, any
element $w\in W$ can be written as $w=\gamma(v)$ with
$v=\gamma^{-1}(w)\in V$. Applying \eqref{gammaAd} to $y=v$ gives:
\be
\Ad(\gamma(x))(w)=\gamma(\Ad^\Cl(x)(v))~~\forall x\in \Cl(V,h)^\times~~.
\ee
When $x\in \G^e(V,h)$, we have $\Ad^\Cl(x)(v)\in V$ and the relation
above gives $\Ad(\gamma(x))(w)\in \gamma(V)=W$, which implies
$\gamma(x)\in \L$. \qed
\end{proof}

\begin{definition}
The {\em image extended Clifford group} of $\gamma$ is the subgroup
$\gamma(\G^e(V,h))\subset \L$.  The {\em image Clifford group} of
$\gamma$ is the subgroup $\gamma(\G(V,h))\subset
\gamma(\G^e(V,h))$. The {\em image pin group} of $\gamma$ is the
subgroup $\gamma(\Pin(V,h))\subset\gamma(\G(V,h))$. The {\em image
  spin group} of $\gamma$ is the subgroup $\gamma(\Spin(V,h))\subset
\gamma(\Pin(V,h))$.
\end{definition}

Recall that $\G(V,h)$ is generated by the non-degenerate vectors of
$(V,h)$ while $\Pin(V,h)$ is generated by the unit vectors of
$(V,h)$. Since $\gamma|_V$ is an invertible isometry from $(V,h)$ to
$(W,g)$, we have:

\begin{prop}
The image Clifford group $\gamma(\G(V,h))$ coincides with the subgroup
of $\Aut_\R(S)$ generated by all non-degenerate vectors of $(W,g)$,
while the image pin group $\gamma(\Pin(V,h))$ coincides with the
subgroup of $\Aut_\R(S)$ generated by all unit vectors of $(W,g)$
\end{prop}

\noindent The $\Z_2$-grading of $\L$ induces $\Z_2$-gradings on all
image subgroups defined above.  For example, we have
$\gamma(\G^e(V,h))^\kappa=\gamma(\G^e(V,h))\cap \L^\kappa$,
$\gamma(\Pin(V,h))^\kappa=\gamma(\Pin(V,h))\cap \L^\kappa$ for
$\kappa\in \Z_2$.  Relation \eqref{L01} implies:
\beqa
&& \gamma(\G^e(V,h))^\kappa=\gamma(\G^e(V,h))\cap \End^\kappa_\R(S)~~,~~\nn\\
&& \gamma(\G(V,h))^\kappa=\gamma(\G(V,h))\cap \End^\kappa_\R(S)~~,~~\nn\\
&& \gamma(\Pin(V,h))^\kappa=\gamma(\Pin(V,h))\cap \End^\kappa_\R(S)~~.
\eeqa
When $d$ is even, the gradings of $\gamma(\G(V,h))$ and
$\gamma(\Pin(V,h))$ coincide with those induced by the canonical
$\Z_2$-grading of $\Cl(V,h)$. When $d$ is odd, we have:
\be
\gamma(\G(V,h))=\gamma(\G(V,h))^0\subset \L^0~~,~~\gamma(\Pin(V,h))=\gamma(\Pin(V,h))^0\subset \L^0~~.
\ee

\begin{remark}
In general, $\L^1$ can be non-empty even when $d$ is odd (since $\L$
need not be generated by invertible elements from $W$). We will see
later that this indeed happens for certain irreducible Clifford
representations. On the other hand, $\G^e(V,h)$ is generated by
$Z(V,h)^\times$ and $\G(V,h)$ and we have
$\gamma(Z(V,h)^\times)\subset \S^\times\subset \L^0$.  This implies
$\gamma(\G^e(V,h))\subset \L^0$ when $d$ is odd. When $d$ is even, we
have $\G^e(V,h)=\G(V,h)$ and $\gamma(\G^e(V,h))=\gamma(\G(V,h))$.
\end{remark} 

\subsection{Twisting elements and twisted image pin groups}

\begin{definition}
We say that $\gamma$ {\em admits twisting elements} if the
intersection $A(\gamma)\cap \Aut_\R(S)$ is non-empty. In this case, a
{\em twisting element} of $\gamma$ is an invertible element of the
anticommutant subspace $A(\gamma)$, i.e. an element of the set $A(\gamma)\cap
\Aut_\R(S)$. A twisting element $\mu\in A(\gamma)\cap \Aut_\R(S)$ is called
    {\em nondegenerate} if the endomorphism $\id_S+\mu\in \End_\R(S)$
    is invertible and {\em special} if $\mu^2=-\id_S$.
\end{definition}

\begin{prop}
Any special twisting element is non-degenerate.
\end{prop}

\begin{proof}
If $\mu$ is a special twisting element, then
$\frac{1}{2}(\id_S-\mu)(\id_S+\mu)=\id_S$ and hence $\id_S+\mu$ is
invertible with inverse $\frac{1}{2}(\id_S-\mu)$. \qed
\end{proof}

\noindent For any twisting element $\mu$ and any $w\in W$, we have
$\Ad(\mu)(w)=-w$. Thus twisting elements form a subset of $\L$. This
set is invariant under multiplication with elements of
$\R^\times$. The subset of special twisting elements is invariant
under changes of sign ($\mu\rightarrow -\mu$). 

\begin{definition}
Let $\mu$ be a twisting element. The {\em $\mu$-twisted image pin
  group of $\gamma$} is the sub-group $\Pin_\mu(\gamma)\subset
\Aut_\R(S)$ generated by the elements $\mu w$, where $w$ runs over the
unit vectors of $(W,g)$.
\end{definition}

\noindent Any element of $\Pin_\mu(\gamma)$ can be brought to the form $
(-1)^{\frac{k(k-1)}{2}}\mu^k w_1\ldots w_k$, where $w_1,\ldots,w_k$
are unit vectors of $(W,g)$, but this presentation need not be unique.

\

\begin{prop}
Let $\mu$ be a {\em special} twisting element. Then the $\mu$-twisted image
pin group is isomorphic with the image pin group, namely:
\ben
\label{PinPinmu}
\Ad(\id_S+\mu)(\gamma(\Pin(V,h)))=\Pin_\mu(\gamma)~~ 
\een
and we have: 
\ben
\label{AdPinmu}
\Ad(a)(W)=W~~\forall~~a\in \Pin_\mu(\gamma)~~.
\een
Moreover, the isomorphism of groups $\varphi_\mu\eqdef
\Ad(\id_S+\mu)|_{\gamma(\Pin(V,h))}:$ $
\gamma(\Pin(V,h))\stackrel{\sim}{\rightarrow}\Pin_\mu(\gamma)$
satisfies:
\ben
\label{AdAdgamma}
\Ad_W \circ \varphi_\mu\circ\gamma|_{\Pin(V,h)}=\Ad(\gamma|_V)\circ \tAd_0~~,
\een
where the group morphism $\Ad_W:\Pin_\mu(\gamma)\rightarrow
\O(W,g)$ is defined through $\Ad_W(a)=\Ad(a)|_W$ for all $a\in
\Pin_\mu(\gamma)$ and $\tAd_0:\Pin(V,h)\rightarrow \O(V,h)$ is the
twisted vector representation of $\Pin(V,h)$.  Hence the
representation $\Ad_W\circ \varphi_\mu\circ \gamma|_{\Pin(V,h)}$ of
$\Pin(V,h)$ is equivalent with the twisted vector representation.
\end{prop}

\begin{remark}
In the proposition, $\gamma|_V$ is viewed as an invertible isometry
from $(V,h)$ to $(W,g)$, so $\Ad(\gamma|_V)$ is a unital isomorphism
of algebras from $\End_\R(V)$ to $\End_\R(W)$ which restricts to an
isomorphism of groups from $\O(V,h)$ to $\O(W,g)$.
\end{remark}

\proof Let $\mu$ be a special twisting element. Then the map $V\ni
v\rightarrow \mu \gamma(v)\in \End_\R(S)$ satisfies $(\mu
\gamma(v))^2=\gamma(v)^2=h(v,v)\id_S$ and hence induces a unital
morphism of algebras $\gamma_\mu:\Cl(V,h)\rightarrow \End_\R(S)$ such
that $\gamma_\mu(v)=\mu\gamma(v)$ for all $v\in V$. The identity $\mu
\gamma(v)=(\id_S+\mu)\gamma(v)(\id_S+\mu)^{-1}=\Ad(\id_S+\mu)(\gamma(v))$
implies that this representation is equivalent with $\gamma$, namely:
\ben
\label{Gammamu}
\gamma_\mu=\Ad(\id_S+\mu)\circ \gamma~~.
\een
Since $\gamma_\mu(\Pin(V,h))=\Pin_\mu(\gamma)$, relation
\eqref{Gammamu} implies \eqref{PinPinmu}. For every $w\in W$ and any
unit vector $u$ of $(W,g)$, we have:
\ben
\label{AdAdtw}
\Ad(\mu u)(w)=\mu u w (\mu u)^{-1}=\mu u w u^{-1} \mu^{-1}=-u w u^{-1}=R_u~~.
\een
This implies $\Ad(\mu u)(W)\subset W$ and shows that \eqref{AdPinmu}
holds.  Taking $u=\gamma(v)$ and $w=\gamma(v')$ with
$v=(\gamma|_V)^{-1}(u)\in V$ a unit vector of $(V,h)$ and
$v'=(\gamma|_V)^{-1}(w)\in V$ in \eqref{AdAdtw} gives:
\be
(\Ad\circ \varphi_\mu\circ \gamma)(v)(w) = \Ad(\mu\gamma(v))( \gamma(v^{\prime}))=\gamma|_V(\tAd_0(v)(v^{\prime}))\, ,
\ee 
which gives \eqref{AdAdgamma}.
\qed

\subsection{Surjectivity conditions for the vector representation}

\begin{prop} 
Let $\mu$ be any twisting element and $w\in W$ be a non-degenerate
vector. Then $\mu w\in \L$ and $\Ad_0(\mu w)$ equals the
$g$-orthogonal reflection of $W$ in the hyperplane orthogonal to
$w$:
\be
\Ad_0(\mu w)=+R_w~~.
\ee
In particular, we have $\Pin_\mu(\gamma)\subset \L$.
\end{prop}

\proof It suffices to consider the case when $w$ is a unit vector, so
we can assume $g(w,w)=\epsilon \in \{-1,1\}$. Then
$w^2=g(w,w)\id_S=\epsilon\id_S$, so $w$ is invertible and
$w^{-1}=\epsilon w$. For any $x\in W$, we compute:
\be
\Ad(\mu w)(x)=\Ad(\mu)(wxw^{-1})=\epsilon \Ad(\mu)(wxw)=-\epsilon wxw =-\Ad(w)(x)=+R_w~~,
\ee
where we used \eqref{mreflection}.
\qed

\begin{thm}
\label{OrthogonalCover}
The following statements hold: 
\begin{enumerate}[1.]
\itemsep 0.0em
\item The image of the vector representation of $\L^0$ equals the
  special orthogonal group of $(W,g)$, so we have a short exact
  sequence:
\ben
\label{L0ext}
1\longrightarrow \S^\times \hookrightarrow \L^0\stackrel{\Ad_0}{\longrightarrow} \SO(W,g)\longrightarrow 1~~.
\een
\item Suppose that $d$ is even or that $\gamma$ admits a twisting
  element. Then the vector representation of $\L$ is surjective and we
  have a short exact sequence:
\ben
\label{Lext}
1\longrightarrow \S^\times \hookrightarrow \L\stackrel{\Ad_0}{\longrightarrow} \O(W,g)\longrightarrow 1~~.
\een
\end{enumerate}
\end{thm}

\begin{proof}
Since $\L$ contains all non-degenerate vectors of $(W,g)$, Proposition
\ref{prop:mreflection} shows that the group $\Ad_0(\L)$ contains all
minus reflections $\Ad_0(w)=-R_w$ of $\O(W,g)$. By the definition of
the grading of $\L$, we have $\Ad_0(\L^0)\subset \SO(W,g)$ and
$\Ad_0(\L^1)\subset \O_-(W,g)$, where $\O_-(W,g)\subset \O(W,g)$
denotes the set of improper orthogonal transformations. Thus
$\Ad_0(\L^0)$ and $\Ad_0(\L^1)$ are disjoint. We distinguish the cases:

\begin{enumerate}[1.]
\itemsep 0.0em
\item When $d$ is even, the minus reflections have determinant $-1$
  and they generate $\O(W,g)$, thus $\Ad_0(\L)=\O(W,g)$, which
  establishes part of the second statement.
\item When $d$ is odd, minus reflections have determinant $+1$ and
  generate $\SO(W,g)$, thus $\SO(W,g)\subset \Ad_0(\L)$. Since
  $\Ad_0(\L^0)\subset \SO(W,g)$ and $\Ad_0(\L^1)\subset \O_-(W,g)$ are
  disjoint, we must have $\SO(W,g)=\Ad_0(\L^0)$, which establishes the
  first statement.
\end{enumerate}
\noindent Suppose now that $d$ is arbitrary but that $\gamma$ admits a
twisting element $\mu$. Then the previous proposition shows that $\L$
contains $\Pin_\mu(\gamma)$ and that $\Ad_0(\L)$ contains all
reflections $\Ad(\mu w)=R_w$ of $(W,g)$, which generate $\O(W,g)$ by
the Cartan-Dieudonne theorem. This completes the proof of the second
statement. \qed
\end{proof}

\subsection{The Schur representation}

\begin{pdef}
\label{SchurRep}
For any $a\in \L$, we have $\Ad(a)(\S)=\S$. Thus $\Ad$ induces a {\em Schur representation}:
\be
\Ad_s:\L\rightarrow \Aut_\Alg(\S)~~,~~\Ad_s(a)\eqdef \Ad(a)|_\S~~(a\in \L)
\ee
through unital algebra automorphisms of $\S$. Furthermore, we have: 
\ben
\label{kerAds}
\ker\Ad_s=\L\cap \End_\S(S)
\een
\end{pdef}

\begin{proof}
Given $a\in \L$, any vector $w\in W$ can be written as
$w=\Ad(a)(w')$ for some $w'\in W$, namely $w'=\Ad(a^{-1})(w)$.  For
any $s\in \S$, we have:
\begin{eqnarray}
\Ad(a)(s)w &=& \Ad(a)(s)\Ad(a)(w')=\Ad(a)(sw')=\Ad(a)(w's)\nonumber\\ &=& \Ad(a)(w')\Ad(a)(s)=w\Ad(a)(s)\, ,
\end{eqnarray}
where we used the fact that $s\in \S$ commutes with $w'\in W$. This
shows that $\Ad(a)(s)$ commutes with any $w\in W$ and hence that
$\Ad(a)(s)$ belongs to $\S$. Thus $\Ad(a)(\S)\subset \S$. Since
$\Ad(a)$ is $\R$-linear and bijective while $\S$ is finite-dimensional
over $\R$, we in fact have $\Ad(a)(\S)=\S$. The fact that $\Ad_s(a)$
is a unital morphism of algebras is obvious, as is the remaining
statement.  \qed
\end{proof}

\subsection{The anticommutant representation}
\label{subsec:anticommutantrep}

\begin{prop}
The anticommutant subspace is invariant under the vector
representation of $\L$:
\be
\Ad(a)(A)=A~~,~~\forall a\in \L~~
\ee
and hence $\Ad$ induces a linear representation $\Ad_A:\L\rightarrow
\Aut_\R(A)$, $\Ad_A(a)\eqdef \Ad(a)|_A$ in the space $A$. Moreover,
$\Ad_A$ is a representation of $\L$ through $\Ad_s$-twisted
$\S$-module automorphisms which are twisted-orthogonal with respect to
the Schur pairing $\fp$, so the following relations hold for all $a\in \L$:
\ben
\label{AdS}
\Ad_A(a)(s x)=\Ad_s(a)(s)\Ad_A(a)(x)~~\forall s\in \S~~\mathrm{and}~~x\in A
\een
and:
\ben
\label{Adrho}
\fp(\Ad_A(a)(x_1),\Ad_A(a)(x_2))=\Ad_s(a)(\fp(x_1,x_2))~~,~~\forall x_1,x_2\in A~~.
\een
\end{prop}

\noindent We refer the reader to Appendix \ref{app:twisted} for the notion of twisted module automorphism
and for the notation used in the proposition. 

\begin{proof} 
For any $a\in \L$, $x\in A$ and $w\in W$, we have: 
\ben
\label{Adrel}
\Ad(a)(x)w=ax\Ad(a^{-1})(w)a^{-1}=-a\Ad(a^{-1})(w)xa^{-1}=-w\Ad(a)(x)~~, 
\een
where in the middle equality we used the fact that $a^{-1}\in \L$
(since $\L$ is a group), which implies that $\Ad(a^{-1})(w)$ belongs
to $W$ and hence that it anticommutes with $x\in A$ (by the definition
of $A$). Relation \eqref{Adrel} shows that $\Ad(a)(x)$ anticommutes
with $w$ for any $w\in W$ and hence $\Ad(a)(x)$ belongs to $A$ for any
$x\in A$, thus $\Ad(a)(A)=A$ for any $a\in \L$. Relation \eqref{AdS}
is obvious.  For any $x_1,x_2\in A$, we have:
\be
\fp(\Ad(a)(x_1),\Ad(a)(x_2))=\frac{1}{2}\left[\Ad(a)(x_1x_2)+\Ad(a)(x_2x_1)\right]=\Ad(a)(\fp(x_1,x_2))~~,
\ee
which gives \eqref{Adrho} since $\fp(x_1,x_2)\in \S$.
\qed
\end{proof}

\begin{definition}
The group morphism: 
\be
\Ad_A:\L\rightarrow \Aut_\S^\tw(A,\fp)
\ee 
is called the {\em anticommutant representation} of $\L$.
\end{definition}

\subsection{Adapted pairings and Lipschitz norms}
\label{subsec:adaptedpairings}

\begin{definition}
A non-degenerate $\R$-bilinear pairing $\cB:S\times S\rightarrow \R$ is
called {\em adapted} to $\gamma$ if it has the following properties:
\begin{enumerate}[(a)]
\item $\cB$ is symmetric or skew-symmetric, that is:
\be
\cB(\xi,\xi')=\sigma_\cB\cB(\xi',\xi)~~\forall \xi,\xi'\in S\, ,
\ee
where $\sigma_\cB\in \{-1,1\}$ is called the {\em symmetry} of $\cB$.
\item $w^t=\epsilon_\cB w$ for all $w\in W$, where
  $~^t:\End_\R(S)\rightarrow \End_\R(S)$ denotes the $\cB$-transpose
  and $\epsilon_\cB\in \{-1,1\}$ is called the {\em type} of $\cB$.
\end{enumerate} 
\end{definition}

\

\begin{remark}
Since an adapted pairing is either symmetric or skew-symmetric, the
$\cB$-transpose is an involutive unital $\R$-algebra anti-automorphism
of $\End_\R(S)$. Notice that an admissible bilinear pairing (in the
sense of \cite{AC1,AC2}) of an irreducible Clifford representation is
adapted; in fact, an admissible pairing is an adapted pairing which has 
a definite ``isotropy'' (see op. cit.).
\end{remark} 

\begin{definition}
We define $\O(S,\cB)$ to be the automorphism group of $(S,\cB)$, namely:
\begin{equation}
\O(S,\cB) = \left\{ a\in \End_{\mathbb{R}}(S)\,\,|\,\, \cB(a x, a y) = \cB(x,y)\right\}\, ,
\end{equation}
for all $x,y\in S$. It consists of all $\cB$-orthogonal invertible linear operators acting in $S$:
\be
\O(S,\cB)=\left\{ a\in \End_{\mathbb{R}}(S)\,\,|a^t\circ a=\id_S\right\}~~.
\ee
\end{definition}
Let $\cB$ be an adapted pairing on $S$. 

\begin{definition}
The {\em Lipschitz norm} determined by $\cB$ is the map
$\cN_\cB:\End_\R(S)\rightarrow \End_\R(S)$ defined through:
\be
\cN_\cB(a)\eqdef a^t\circ a~~(a\in \End_\R(S))~~.
\ee
\end{definition}

\noindent Notice that $\cN_\cB(a)^t=\cN_\cB(a)$, so the image of the
Lipschitz norm consists of $\cB$-symmetric linear operators. We have
$\cN_\cB(w)=\epsilon_\cB w^2=\epsilon_\cB g(w,w)\id_S$ for all $w\in
W$.  Notice that an operator $a\in \End_\R(S)$ is $\cB$-orthogonal if
and only if $\mathcal{N}_\cB(a)=\id_S$; in particular, the
intersection:
\be
W\cap \O(S,\cB)=\left\{w\in W|g(w,w)=\epsilon_\cB\right\}
\ee
coincides with the set of unit vectors of signature equal to $\epsilon_\cB$ and we have:
\ben
\label{LB}
\L\cap \O(S,\cB)=\{a\in \L\,\, |\,\, \cN_\cB(a)=\id_S\}~~.
\een
However, the restriction $\cN_\cB|_\L$ does {\em not} generally give a
group morphism, because the elements of $\cN_\cB(\L)$ need not commute
with those of $\L$ or $\L^t$. As we shall see below, the situation is somewhat
better for the image extended Clifford group $\gamma(\G^e(V,h))$.

\begin{definition}
The {\em modified reversion} determined by $\cB$ is the unital
anti-automorphism of $\Cl(V,h)$ given by:
\be
\tau_\cB(x)=\tau\circ \pi^{\frac{1-\epsilon_\cB}{2}}=\twopartdef{\tau}{\epsilon_\cB=+1}{\ttau=\tau\circ \pi}{\epsilon_\cB=-1}~~.
\ee
The {\em modified Clifford norm} determined by $\cB$ is the map
$N_\cB:\Cl(V,h)\rightarrow \R$ defined through:
\be
N_\cB(x)\eqdef \tau_\cB(x)x=\twopartdef{N}{\epsilon_\cB=+1}{\tN}{\epsilon_\cB=-1}~~(x\in \Cl(V,h))~~.
\ee
\end{definition}

\

\noindent Notice that $\tau_\cB(v)=\epsilon_\cB v$ for all $v\in V$. For any $x\in \Cl(V,h)$, we have: 
\be
\gamma(x)^t=\gamma(\tau_\cB(x))~~,
\ee
which implies:
\ben
\label{cNgamma}
\cN_\cB\circ \gamma=\gamma\circ N_\cB~~.
\een
In particular, we have $\cN_\cB(C(\gamma))\subset C(\gamma)$. 

\

\begin{prop}
We have $\cN_\cB(\gamma(\G(V,h)))\subset \R^\times\id_S$ and:
\ben
\label{cNinc}
\cN_\cB(\gamma(\G^e(V,h)))\subset \gamma(N_\cB(\G^e(V,h)))\subset Z(\S)\cap \S^\times=Z(\S)^\times~~.
\een
In particular, we have $\cN_\cB(\gamma(\G^e(V,h))\subset \R^\times\id_S$
if $N_\cB$ coincides with the improved Clifford norm.
\end{prop}

\

\begin{proof}
Since $N(\G(V,h))\subset \R^\times$ and $\tN(\G(V,h))\subset
\R^\times$, we have $N_\cB(\G(V,h))\subset \R^\times$ and
\eqref{cNgamma} gives $\cN_\cB(\gamma(\G(V,h)))\subset
\R^\times\id_S$. The first inclusion in \eqref{cNinc} also follows
from \eqref{cNgamma}. Since $N(\G^e(V,h))\subset Z(V,h)^\times$ and 
$\tN(\G^e(V,h))\subset Z(V,h)^\times$, we have $N_\cB(\G^e(V,h))\subset
Z(V,h)^\times$. Since $\gamma(Z(V,h))\subset Z(\S)$, we obtain the
second inclusion of \eqref{cNinc}.  \qed
\end{proof}

\

\noindent Since $N$ and $\tN$ are equal on $\Spin(V,h)$, we have
$N_\cB=N$ on $\Spin(V,h)$ and \eqref{cNgamma} gives:
\be
\cN_\cB(\gamma(a))=\twopartdef{+\id_S}{a\in \Spin^+(V,h)}{-\id_S}{a\in \Spin^-(V,h)}
\ee
In particular, we have $\gamma(\Spin^+(V,h))\subset \O(S,\cB)$ and
hence any adapted pairing is invariant under the action of the
subgroup $\gamma(\Spin^+(V,h))\subset \L$. Notice that an adapted
pairing {\em need not} be invariant under the $\gamma$-action of the
full spin group $\Spin(V,h)$. Also notice that the full subgroup
$\L\cap \O(S,\cB)$ consisting of those elements of $\L$ which preserve
an adapted pairing $\cB$ can be strictly larger than
$\gamma(\Spin^+(V,h))$, even when $\cB$ is an admissible pairing of an
irreducible Clifford representation.

\
\begin{remark}
In general, the set $\cN_\cB(\L)$ is larger that $Z(\S)^\times$. 
The decomposition \eqref{SchurDec} implies:
\be
Z(\S)=\oplus_{i=1}^n Z(\S_i)\id_{U_i}\simeq_\Alg \oplus_{i=1}^n Z(\S_i)~~,~~Z(\S)^\times \simeq_\Gp \oplus_{i=1}^n Z(\S_i)^\times~~,
\ee
where $\S_i\eqdef \S(\gamma_i)$ are the Schur algebras of the
inequivalent irreducible components of $\gamma$. We have: 
\be
Z(\S_i)^\times\simeq \threepartdef{\R^\times}{~,~\S_i\simeq_\Alg \R}{\C^\times}{~,~\S_i\simeq_\Alg \C}{\R^\times}{~,~\S_i\simeq_\Alg \H}~~.
\ee
\end{remark}

\section{Lipschitz groups of irreducible real Clifford representations}
\label{sec:irreps}

In this section, we study the Lipschitz groups of irreducible real
Clifford representations (all of which turn out to be weakly faithful
and to form a single unbased isomorphism class in every signature) as
well as their elementary representations. In particular, we show that
the reduced Lipschitz groups of such representations (which we call
{\em elementary reduced Lipschitz groups}) are isomorphic with the
canonical spinor groups introduced in Section \ref{sec:cangroup} and
that their elementary representations agree with those of the
canonical spinor groups. This shows, in particular, that the canonical
spinor group of $(V,h)$ arises naturally as the Lipschitz group of the
unique unbased isomorphism class of the irreps of $(V,h)$, which is
always weakly faithful. Notice that this provides a unifying
perspective on various extended spinor groups arising in spin
geometry, while also including the groups $\Spin^o(V,h)$, new in the literature, and identifying the precise spinor group which is relevant when
considering irreducible real Clifford representations in every
dimension and signature. One should compare this with the traditional
approach, where a spinor group is chosen apriori, without worrying
about irreducibility of the corresponding Clifford representation.  In
our approach, irreducibility of the Clifford representation is the
central feature of interest.

\subsection{Basics}

\begin{definition}
A {\em real pin representation} is an irreducible finite-dimensional
real Clifford representation $\gamma:\Cl(V,h)\rightarrow \End_\R(S)$,
where $V\neq 0$ and $S\neq 0$.
\end{definition}

\noindent 
A real pin representation
$\gamma:\Cl(V,h)\rightarrow \End_\R(S)$ is faithful iff $\Cl(V,h)$ is
simple as an associative $\R$-algebra, which happens when $p-q\not
\equiv_8 1,5$ (the {\em simple case}). The pinor volume element
$\omega=\gamma(\nu)$ is proportional to $\id_S$ iff we are in the {\em
  non-simple case} $p-q\equiv_8 1,5$. In the simple case, all pin
representations of $\Cl(V,h)$ are equivalent. In the non-simple case,
$\Cl(V,h)$ admits two inequivalent irreducible representations, which
can be realized in the same space $S$. In each of these irreps, the
Clifford volume element $\nu\in \Cl(V,h)$ defined by a given orientation
of $V$ satisfies:
\be
\omega\eqdef \gamma(\nu)=\epsilon_\gamma \id_S~~,
\ee
where $\epsilon_\gamma\in \{-1,1\}$ is a sign factor called the {\em
  signature} of the irrep $\gamma$.  The two irreps are distinguished
by the value of $\epsilon_\gamma$ and we denote them by
$\gamma_\pm:\Cl(V,h)\rightarrow \End_\R(S)$ (where
$\epsilon_{\gamma_\pm}=\pm 1$). We have:
\ben
\label{gammapm}
\gamma_+=\gamma_-\circ \pi~~,
\een
where $\pi:\Cl(V,h)\rightarrow \Cl(V,h)$ is the parity automorphism,
which satisfies $\pi(\nu)=-\nu$ since $d=\dim V=p+q$ is odd in the
non-simple case.  Though inequivalent, these two irreps are isomorphic
in the category $\ClRep$ through the isomorphism
$(f_0,f)=(-\id_V,\id_S)$, where $-\id_V\in \O_-(V,h)$. Indeed, we have
$\pi|_V=-\id_V$ hence $\Cl(-\id_V)=\pi$ and \eqref{gammapm} reads
$\Ad_{\id_S}\circ \gamma_+=\gamma_-\circ \Cl(-\id_V)$
(cf. \eqref{f0f}), which shows that $(f_0,f):\gamma_+\rightarrow
\gamma_-$ is an isomorphism in $\ClRep$. Notice that $(f_0,f)\circ
(f_0,f)=\id_{\gamma_+}$, so $(f_0,f)^{-1}=(f_0,f)$.
The kernel of $\gamma_\epsilon$ is given
by:
\ben
\label{kernel}
\ker \gamma_\epsilon = \{x\in \Cl(V,h)|x\nu=-\epsilon x\}~~.
\een
and we have\footnote{In the non-simple case, we have $\nu^2=+1$ and
  multiplication with $\nu$ gives a non-unital involutive algebra
  endomorphism of $\Cl(V,h)$.} $\dim (\ker \gamma_\epsilon)=\dim
(\im\gamma_\epsilon)=\frac{1}{2}\dim \Cl(V,h)=2^{d-1}$.

\begin{prop}
Let $(V,h)$ be a quadratic space. Then all real irreducible 
representations of $\Cl(V,h)$ are weakly faithful. Moreover, there
exists a single isomorphism class of such representations in the
category $\ClRep$, which is uniquely determined by the isomorphism
class of $(V,h)$ and hence by the signature of $h$. In the simple
cases, this isomorphism class is also an equivalence class of
representations. In the non-simple cases, this isomorphism class
decomposes into two equivalence classes of representations, each of
which is determined by the signature of $h$.
\end{prop}

\begin{proof} 
Injectivity of $\gamma|_V$ can fail only when $\gamma$ is not
faithful and $V\cap \ker\gamma\neq \{0\}$, which by relation
\eqref{kernel} can happen only when $\dim V=2$ since right
multiplication with $\nu$ maps $V$ into the subspace of $\Cl(V,h)$
which corresponds to $\wedge^{d-1} V$ through the
Riesz-Chevalley-Crumeyrolle isomorphism. However, $\gamma$ is always
faithful when $\dim V=2$, since in this case $p-q\in \{-2,0,2\}$ and
hence $p-q\equiv_8 0,2, 6$, which corresponds to the simple case. 
The remaining statements follow from the discussion above. \qed
\end{proof}

\begin{prop}
\label{prop:gammaNS}
In the simple case, $\gamma$ gives a bijection between $\Cl(V,h)$ and
the Clifford image $C=\im \gamma$. In the non-simple case, we have
$C=C_+=C_-$ and $\gamma$ restricts to linear bijections between
$\Cl_\pm(V,h)$ and $C$. In this case, the restriction of $\gamma$ to
$\Cl_+(V,h)$ is a unital isomorphism of algebras from $\Cl_+(V,h)$ to $C$.
\end{prop}

\begin{proof} 
In the simple case, $\gamma$ is faithful and hence induces a
bijection between $\Cl(V,h)$ and $C\eqdef \gamma(\Cl(V,h))$.  In the
non-simple case, we have $\omega=\gamma(\nu)=\epsilon \id_S$, where
$\epsilon\eqdef \epsilon_\gamma$. Recall that $C_\pm\eqdef
\gamma(\Cl_\pm(V,h))$ and hence $C=\gamma(\Cl(V,h))=C_++C_-$.  The
dimension of $V$ is odd and the right multiplication
$R_\nu:\Cl(V,h)\rightarrow \Cl(V,h)$ with the Clifford volume element
$\nu$ maps $\Cl_\pm(V,h)$ into $\Cl_\mp(V,h)$. This implies that the
right multiplication $R_\omega:\End_\R(S)\rightarrow \End_\R(S)$ with
$\omega$ satisfies $R_\omega(C_\pm)=C_\mp$. Since $\omega=\epsilon
\id_S$, we have $R_\omega(C_\pm)=C_\pm$ and we conclude that
$C=C_+=C_-$ since $C_\pm$ are subspaces of $\End_\R(S)$. Since
$\nu^2=1$ in the simple cases, we have $(R_\nu)^2=\id_{\Cl(V,h)}$.
Hence the linear map $R_\nu:\Cl_+(V,h)\rightarrow \Cl_-(V,h)$ is
bijective and thus
$\dim_\R\Cl_+(V,h)=\dim_\R\Cl_-(V,h)=\frac{1}{2}\dim_\R \Cl(V,h)=\dim
(\im\gamma)=\dim C$, which implies that $\gamma$ restricts to
bijections between $\Cl_\pm(V,h)$ and $C$. Since
$C_+=\gamma(\Cl_+(V,h))$ and $\Cl_+(V,h)$ is a unital subalgebra of
$\Cl(V,h)$, the unital morphism of algebras $\gamma$ restricts to a
unital isomorphism of algebras from $\Cl_+(V,h)$ to $C_+$. \qed
\end{proof}

\subsection{The Schur algebra}

\begin{definition}
Let $U$ be an oriented three-dimensional Euclidean vector space. The
{\em quaternion algebra} of $U$ is the normed unital associative
$\R$-algebra $\H_U$ whose underlying set equals $\R\oplus U$ and whose
multiplication is defined through:
\be
(q_0+\bq)(q_0'+\bq')\eqdef q_0q_0' -(\bq,\bq') +q_0\bq'+q_0'\bq+\bq\times \bq'~~\forall q_0,q'_0\in \R,~\bq,\bq'\in U~~,
\ee
where $(~,~)$ and $\times$ denote scalar and vector products of
$U$. The norm of $q:=q_0+\bq\in \H_U$ is defined through:
\be
||q||_{U}\eqdef \sqrt{c_U(q)q}=\sqrt{q_0^2+||\bq||^2}~~,
\ee
where $||~||$ is the norm of $U$ and $c_U:\H_\U\rightarrow \H_U$ is the
conjugation of $\H_U$, i.e. the unital anti-automorphism given by:
\be
c_U(q_0+\bq)=q_0-\bq~~.
\ee
\end{definition}

\noindent The standard algebra $\H$ of quaternions is the quaternion
algebra of $\R^3$, when the latter is endowed with its canonical
scalar product and orientation. Any quaternion algebra $\H_U$ is of
course isomorphic with $\H$ as a unital associative normed algebra
through some (non-unique) isomorphism which takes $\Im\H=\R^3$ into
$U$. The following result characterizes the Schur algebra of pin
representations (see, for example, \cite{gf}):

\begin{prop}
\label{Schur}
The Schur algebra of pin representations is as follows:
\begin{enumerate}[1.]
\itemsep 0.0em
\item In the normal (simple or non-simple) case, we have $\S=\R\id$,
  which we identify with $\R$ through the isomorphism $\R\ni
  x\stackrel{\sim}{\longrightarrow} x\id_S\in \S$. In this case, we
  set $c=\id_\S$ and endow $\S$ with the norm induced from $\R$.
\item In the complex case, we have $\S=\R\id_S\oplus \R\omega$, which
  we identify with $\C$ through the isomorphism $\C\ni
  z=x+iy\stackrel{\simeq}{\longrightarrow} x+y\omega \in \S$ (in this
  case, we have $\nu^2=-1$ and hence $\omega^2=-\id_S$).  Accordingly,
  we let $J\eqdef \omega$ denote the imaginary unit of $\S$. In
  particular, $\S$ is a normed $\R$-algebra whose norm and conjugation
  $c$ do not depend on the choice of orientation of $V$ (and hence are
  invariant under the change $\omega\rightarrow -\omega$). The
  subspace $\Im \S\eqdef \R\omega$ is also independent of the choice
  of orientation of $V$.
\item In the quaternionic (simple or non-simple) case, we have a
  direct sum decomposition $\S=\R\id_S\oplus U$ of the underlying
  $\R$-vector space of $\S$ (where $U=U(\gamma)$ is an oriented
  Euclidean vector space determined by $\gamma$) and $\S$ is
  isomorphic with the quaternion algebra of $U$ through the map
  $\H_U\ni (q_0,\bq)\rightarrow (q_0\id_S,\bq)\in \S$. In particular,
  $\S$ has a natural structure of normed $\R$-algebra, hence there
  exists a (non-unique) unital isomorphism of normed algebras
  $f:\H\stackrel{\sim}{\rightarrow} \S$ such that $f(\Im \H)=U$.  The
  conjugation $c=c_U$ of $\S$ is uniquely determined by $\gamma$.
\end{enumerate}
\end{prop}

\noindent In the quaternionic case, we set $\Im\S\eqdef U$.  The
isomorphisms of the proposition map the groups of unit norm elements
$\U(\R),\U(\C)$ and $\U(\H)$ to the corresponding subgroup of $\S$,
which we denote by $\U(\S)$. In the complex case, the isomorphism
$\C\simeq \S$ of Proposition \ref{Schur} depends on the choice of
orientation of $V$; changing that orientation amounts to postcomposing
that isomorphism with the conjugation $c$ of $\S$ or to precomposing
it with the conjugation of $\C$. The following proposition clarifies
the role of the isomorphism $f$ in the quaternionic case.

\begin{prop}
Let $\S$ be a unital $\R$-algebra such that $\S\simeq_\Alg \H$ and let
$m:\S\times \S\rightarrow \S$, $m(q_1,q_2)\eqdef q_1q_2$ denote the
multiplication map of $\S$. Then there exists a surjective map
$\Phi:\Isom_\Alg(\H,\S)\rightarrow B(\S)$, where:
\begin{enumerate}[(a)]
\item $\Isom_\Alg(\H,\S)$ is the set of unital isomorphisms of
  $\R$-algebras $f:\H\stackrel{\sim}{\rightarrow}\S$
\item $B(\S)$ is the set of $\R$-subspaces $\U\subset \S$ such that
  $\S=\R 1_\S\oplus U$, endowed with a Euclidean scalar product
  $(~,~)$ and orientation such that the following condition is
  satisfied:
\begin{enumerate}[(C)]
\item The restriction $m|_{U\times U}:U\times U\rightarrow \S$ of the
  multiplication map has the form: \be
  m_U(\bs,\bs')=-(\bs,\bs')+\bs\times \bs'~~, \ee 
  where $\times:U\times U\rightarrow U$ is the vector product of $U$.
\end{enumerate}
\end{enumerate}
Namely, we have $\Phi(f)\eqdef f(\Im \H)$ and the scalar
product and orientation of $\Phi(f)$ are induced by $f$ from those of $\Im
\H$. For any $U\in B(\S)$, the preimage $\Phi^{-1}(U)$ is a torsor for
the right action of $\Aut_\Alg(\H)$:
\be
f\rightarrow f\circ \varphi~~,~~\varphi\in \Aut_\Alg(\H)\simeq_{\Gp} \SO(3,\R)~~.
\ee
and hence we have a bijection $B(\S)\simeq_{\Set}
\Isom_\Alg(\H,\S)/\SO(3,\R)$. For any $f\in \Phi^{-1}(U)$ and any
$q\in \H$ we have $f({\bar q})=c_U(f(q))$, where $c_U:\S\rightarrow \S$ is
the involutive unital anti-automorphism of the algebra $\S$ given by
$c_U=\id_{\R1_\S}\oplus (-\id_U)$.
\end{prop}

\begin{proof}
Surjectivity of $\Phi$ follows by picking an oriented
orthonormal basis $(e_1,e_2,e_3)$ of $U$ and noticing that
$U=f(\Im\H)$ for the unique $\R$ linear map $f:\H\rightarrow \S$ which
satisfies $f(1_\H)=1_\S$ and $f(\epsilon_i)=e_i$, where
$\epsilon_1,\epsilon_2,\epsilon_3$ is a canonically-oriented
orthonormal basis of $\R^3$. It is easy to see that this map is a
unital isomorphism of $\R$-algebras.  That $\Phi^{-1}(U)$ is an
$\SO(3,\R)$-torsor is obvious.  \qed
\end{proof}

\begin{remark}
Notice that $m$ is completely determined by the Euclidean
scalar product of $U$ and by its orientation through the formula: 
\be
m(s_0+\bs, s_0'+\bs')\eqdef s_0s_0' -(\bs,\bs') +s_0\bs'+s_0'\bs+\bs\times_U \bs'~~\forall s_0,s'_0\in \R~~\forall \bs,\bs'\in U~~.
\ee
Conversely, $m$ determines both the scalar product and orientation of
$U$ though condition (C).
\end{remark}

\subsection{The anticommutant subspace}

For the remainder of this section, we fix a real pin representation
$\gamma:\Cl(V,h)\rightarrow \End_\R(S)$. Let $\alpha_{p,q}\in
\{-1,1\}$ be defined as follows:
\begin{enumerate}[1.]
\itemsep 0.0em
\item For the normal simple case:
\be
\alpha_{p,q}\eqdef \sigma_{p,q}=(-1)^{\frac{p-q}{2}}=\twopartdef{+1}{p-q\equiv_8 0}{-1}{p-q\equiv_8 2}
\ee
\item For the complex case:
\be
\alpha_{p,q}\eqdef (-1)^{\frac{p-q+1}{4}}=\twopartdef{-1}{p-q\equiv_8 3}{+1}{p-q\equiv_8 7}~~
\ee
\item For the quaternionic simple case: 
\be
\alpha_{p,q}\eqdef \sigma_{p,q}=(-1)^{\frac{p-q}{2}}=\twopartdef{+1}{p-q\equiv_8 4}{-1}{p-q\equiv_8 6}~~. 
\ee
\end{enumerate}

\begin{prop}
\label{Alist}
The following statements hold:
\begin{enumerate}[I.]
\item In the non-simple cases, we have $A=0$.
\item In the simple cases, $A$ is a rank one free $\S$-module. Namely,
  there exists an element $u\in A$ such that:
\begin{enumerate}[(a)]
\itemsep 0.0em
\item $u$ is a basis of $A$ over $\S$
\item $u$ satisfies $u^2=\alpha_{p,q} \id_S$ (in particular, $u$ is invertible)
\item $u$ satisfies:
\be
\Ad_s(u)=\threepartdef{\id_\S}{for~the~normal~simple~case}{c}{for~the~complex~case}{\id_\S}{for~the~quaternionic~simple~case}
\ee
\end{enumerate}
In the normal simple and quaternionic simple cases, there exist only
two elements $u\in A$ with these properties, namely $u=\pm \omega$.
In the complex case, any two elements $u\in A$ which have these
properties are related by $u'=s u$ where $s=e^{\theta J} \id_S\in \U(\S)$ (
$\theta\in \R$) corresponds to a complex number of unit modulus under
the isomorphism $\S\simeq_\Alg \C$ of Proposition \ref{Schur}.
\end{enumerate}
\end{prop}

\begin{proof} 
Consider first the non-simple cases. Then $x\in A$ must satisfy
$x \gamma(v)=-\gamma(v) x$ for all $v\in V$, which implies
$x\gamma(\nu)=-\gamma(\nu) x$ since $d$ is odd. On the other hand, in
these cases we have $\omega=\gamma(\nu)=\epsilon \id_S$ (with
$\epsilon \in \{-1,1\}$), so the relation $x\omega=-\omega x$ becomes
$2\epsilon x=0$, which implies $x=0$. We conclude that $A=0$ in the
non-simple cases.

Consider now the simple cases. It was shown in \cite{gf} that the element: 
\ben
\label{udef}
u\eqdef \threepartdef{\omega}{for~the~normal~simple~case}{D}{for~the~complex~case}{\omega}{for~the~quaternionic~simple~case}
\een
(where $D$ was defined in \cite{gf}) satisfies
conditions (a),(b) and (c). Let $u'\in A$ be another element
satisfying these three conditions. Then $u'=s u$ for some $s\in
\S^\times$ and hence $(u')^2=(su)^2=susu=s\Ad_s(u)(s) u^2$. Since
$u^2=(u')^2=\alpha_{p,q}\id_S$, this gives $s\Ad_s(u)(s)=\id_S$. In
the normal simple and normal quaternionic cases, we have
$\Ad_s(u)=\id_\S$, so the previous relation gives $s^2=\id_S$ and
hence $s\in \{-\id_S,\id_S\}$ since the only square roots of unity in
the algebras $\C$ and $\H$ are $-1$ and $+1$ (because $\R$ and $\H$
are associative division algebras). It is clear that $-u$ satisfies
$(a)$, $(b)$ and $(c)$. In the complex case, we have $\Ad_s(u)=c$, so
the relation above becomes $sc(s)=\id_S$. This shows that $s$
corresponds to a complex number of unit modulus under the isomorphism
$\S\simeq_\Alg \C$ of Proposition \ref{Schur}. In this case, it is obvious that
$e^{\theta J}u$ satisfies the three conditions for any $\theta\in \R$.
\qed
\end{proof}

\begin{cor}
\label{TwistingElements}
In the simple cases, any element $u$ satisfying conditions (a), (b),
(c) of Proposition \ref{Alist} is a twisting element for
$\gamma$. Moreover:
\begin{enumerate}[1.]
\itemsep 0.0em
\item In the normal simple case, $u=\omega$ is a special twisting
  element iff $p-q\equiv_8 2$.
\item In the complex case, $u=D$ is a special twisted element iff
  $p-q\equiv_8 3$.
\item In the quaternionic simple case, $u=\omega$ is a special
  twisting element iff $p-q\equiv_8 6$.  When $p-q\equiv_8 4$, the
  element $\mu=J\omega\in A$ is a special twisting element for any
  $J\in \Im \S\cap \U(\S)$ (we have $J^2=-\id_S$).
\end{enumerate}
\end{cor}

\proof Follows immediately from Proposition \ref{Alist}. \qed 

\subsection{The Schur pairing in the simple case}

Assume that $(V,h)$ belongs to the simple case and let $u\in A$ be an
element having the properties given in Proposition \ref{Alist}. Then
the Schur pairing $\fp:A\times A\rightarrow \S$ can be identified with
a symmetric $\R$-bilinear map $\fp_u:\S\times \S\rightarrow \S$ given
by:
\be
\fp_u(s_1, s_2 )\eqdef \fp(s_1 u, s_2 u)=\frac{1}{2}\left[(s_1 u)(s_2 u)+ (s_2 u) (s_1 u)\right]~~,
\ee
(see Appendix \ref{app:twisted}). Identifying $\S$ with $\R$, $\C$ or
$\H$ as in Proposition \ref{Schur}, we find:

\begin{prop}
\label{prop:SchurPairing}
In the simple cases, the Schur pairing can be identified with one of
the following:
\begin{enumerate}[1.]
\itemsep 0.0em
\item In the normal simple case $p-q\equiv_8 0,2$, the Schur pairing
  can be identified with the map $\fp_u:\R\times \R\rightarrow \R$
  given by:
\be
\fp_u(x_1,x_2)= \alpha_{p,q} x_1 x_2~~,~~\mathrm{where}~~\alpha_{p,q}=\sigma_{p,q}=(-1)^{\frac{p-q}{2}}~~.
\ee
\item In the complex case $p-q\equiv_8 3,7$, the Schur pairing can be
  identified with the map $\fp_u:\C\times \C\rightarrow \C$ given by:
\be
\fp_u(z_1,z_2)=\alpha_{p,q} \mathrm{Re}({\bar z_1}z_2)\in \R\subset \C~~,~~\mathrm{where}~~\alpha_{p,q}=(-1)^{\frac{p-q+1}{4}}~~.
\ee 
\item In the quaternionic simple case $p-q\equiv_8 4,6$, the Schur
  pairing can be identified with the map $\fp_u:\H\times \H\rightarrow
  \H$ given by:
\be
\fp_u(q_1,q_2)=\alpha_{p,q}\frac{1}{2}(q_1q_2+q_2q_1)~~,~~\mathrm{where}~~\alpha_{p,q}=\sigma_{p,q}=(-1)^{\frac{p-q}{2}}~~.
\ee
\end{enumerate}
\end{prop}

\subsection{Pseudocentralizers of pin representations}

\begin{prop}
For any real pin representation, we have $\S\cap A=0$. Hence the
pseudocentralizer $\T$ of $\gamma$ is a $\Z_2$-graded algebra with
homogeneous subspaces $\T^0=\S$ and $\T^1=A$. In the simple cases, we
have $A=\S u$ (where $u$ is as in Proposition \ref{Alist}),
which gives the following unital isomorphisms of $\R$-algebras:
\begin{enumerate}
\itemsep 0.0em
\item In the normal simple case: 
\be
\T\simeq_{\Alg_{\Z_2}} \twopartdef{\Cl_{1,0}\simeq_\Alg \D}{p-q\equiv_8 0}{\Cl_{0,1}\simeq_\Alg \C}{p-q\equiv_8 2}
\ee
\item In the complex case:
\be
\T\simeq_{\Alg_{\Z_2}} \twopartdef{\Cl_{0,2}\simeq_{\Alg} \H}{p-q\equiv_8 3}{\Cl_{2,0}\simeq \Cl_{1,1}\simeq_\Alg \P}{p-q\equiv_8 7}~~,
\ee
\item In the quaternionic simple case: 
\be
\T\simeq_{\Alg_{\Z_2}} \twopartdef{\Cl_{0,3}\simeq_\Alg \D_\H}{p-q\equiv_8 4}{\Cl_{3,0}\simeq_\Alg \C_\H}{p-q\equiv_8 6}~~,
\ee
\end{enumerate}
In the non-simple cases, we have $\T=\T^0=\S$.
\end{prop}

\begin{proof}
Let $x\in \S\cap A$. Then $xw=-wx=wx$ for all $w\in W=\gamma(V)$,
which implies $wx=0$.  Since $W\simeq V\neq 0$ by our general
assumption, there exists a non-degenerate vector $w_0$ in $(W,g)$.
Since $w_0$ is invertible in the algebra $\End_\R(S)$, the relation
$w_0x=0$ implies $x=0$. Thus $\S\cap A=0$.  Since $\T=\S+A$, we have
$\T=\S\oplus A$, which is a $\Z_2$-grading of the algebra $\T$ by the
discussion of Subsection \ref{subsec:pseudocentralizer}. The remaining 
statements follow from Proposition \ref{Alist}.
\qed
\end{proof}

\noindent The situation is summarized in Table \ref{table:Pseudocenter}

\begin{table}[H]
\centering
\begin{tabular}{|c|c|c|c|c|c|c|c|}
\hline
$\begin{array}{c} p-q\\ {\rm mod}~8 \end{array}$  & type  & $\S$ & $u$ & $u^2$ & $\Ad_s(u)$ &$\T$  \\
\hline\hline
$0$ & normal  simple   & $\R$ & $\omega$ & $+1$ & $\id_\S$ & $\Cl_{1,0}\simeq \D$   \\
\hline
$2$ & normal  simple   & $\R$ & $\omega$ & $-1$ & $\id_\S$ &$\Cl_{0,1}\simeq \C$   \\
\hline
$3$ & complex  simple  & $\C$ & $D$ & $-1$ & $c$ &$\Cl_{0,2}\simeq \H$  \\
\hline
$7$ & complex simple  & $\C$ & $D$ & $+1$ & $c$ & $\Cl_{2,0}\simeq \P$  \\
\hline
$4$ & quaternionic simple & $\H$ &$\omega$ & $+1$ & $\id_\S$ & $\Cl_{0,3}\simeq \D_\H$ \\
\hline
$6$ & quaternionic  simple & $\H$ &$\omega$ & $-1$ & $\id_\S$ &$\Cl_{3,0}\simeq \C_\H$ \\
\hline
$1$ & normal  non-simple & $\R$  & $--$ & $--$ & $--$ &$\Cl_{0,0}\simeq \R$ \\
\hline
$5$ & quaternionic  non-simple & $\H$  & $--$& $--$ & $--$ & $\Cl_{0,2}\simeq \H$    \\
\hline 
\end{tabular}
\vskip 0.2in
\caption{Pseudocentralizers of pin representations}
\label{table:Pseudocenter}
\end{table}

\subsection{The Schur representation}

\begin{definition}
Let:
\ben
\label{Lgamma}
\L_\gamma\eqdef \twopartdef{\gamma(\G^e(V,h)}{p-q\not\equiv_ 8 1,5~(\mathrm{the~simple~case})}{\gamma(\G_+(V,h))}{p-q\equiv_8 1,5~(\mathrm{the~non~simple~case})}~~.
\een
\end{definition}

\begin{prop}
\label{KAds}
For any real pin representation, the Clifford image is given by
$C=\End_\S(S)$ and the kernel of the Schur representation of $\L$
equals $\L_\gamma$:
\ben
\label{kAds}
\ker \Ad_s=\L_\gamma=\L\cap C=\L\cap C^\times~~.
\een
Furthermore:
\begin{enumerate}
\itemsep 0.0em
\item In the simple case, $\gamma$ restricts to an isomorphism of
  groups between $\G^e(V,h)$ and $\L_\gamma$.
\item In the non-simple case, we have $\L=\L^0$ and $\gamma$ restricts
  to an isomorphism of groups between $\G_+(V,h)$ and $\L_\gamma$.
\end{enumerate}
\end{prop}

\begin{proof}
The fact that $C=\End_\S(S)$ is well-known (see, for example,
\cite{gf}). Thus \eqref{kerAds} gives $\ker\Ad_s=\L\cap
C$. Let:
\be
C^\times\eqdef \{a\in C|\exists b\in C: ab=ba=\id_S\}
\ee
denote the group of invertible elements of the subalgebra $C$ of
$\End_\R(S)$.  Since $C=\End_\S(S)$, we have\footnote{If $a\in C\cap
  \Aut_\R(S)=\End_\S(S)\cap \Aut_\R(S)$, then the relation $as=sa$
  with $s\in \S$ implies $a^{-1}s=sa^{-1}$ and hence
  $a^{-1}\in \End_\S(S)=C$, where $a^{-1}$ is the inverse of $a$ in
  $\End_\R(S)$. Thus $C\cap \Aut_\R(S)\subset C^\times$. The opposite
  inclusion is obvious.} $C\cap \Aut_\R(S)=C^\times$, i.e. an element
of $C$ is invertible in $\End_\R(S)$ iff it is invertible in
$C$. Since $\L\subset \Aut_\R(S)$, this implies $\L\cap C\subset
C^\times$ and hence $\L\cap C=\L\cap C^\times$.  To show that $\L\cap
C^\times=\L_\gamma$, we distinguish the cases:
\begin{enumerate}
\item In the simple case, $\gamma$ restricts to unital isomorphism of
  algebras between $\Cl(V,h)$ and $C$ and between $V$ and $W\eqdef
  \gamma(V)\subset C$. In particular, we have
  $C^\times=\gamma(\Cl(V,h)^\times)$. Using the definitions of $\L$
  and $\G^e(V,h)$, this implies:
\beqa
\L\cap C^\times &=&\{a\in C^\times|\Ad(a)(W)=W\}=\nn\\
&=&\gamma(\{x\in \Cl(V,h)^\times|\Ad(x)(V)=V\})=\gamma(\G^e(V,h))~~,
\eeqa
showing that $\ker\Ad_s=\L\cap C=\L\cap C^\times$ equals $\L_\gamma$.  It
is clear that $\gamma$ restricts to an isomorphism of groups between
$\G^e(V,h)$ and $\ker\Ad_s=\L_\gamma=\gamma(\G^e(V,h))$.
\item In the non-simple case, $\gamma$ restricts to a unital
  isomorphism of algebras from $\Cl_+(V,h)$ to $C=\End_\S(S)$ (see
  Proposition \ref{prop:gammaNS}) and to an isomorphism of vector
  spaces from $V$ to $W=\gamma(V)$. In particular, we have
  $C^\times=\gamma(\Cl_+(V,h)^\times)$. Using the definitions of $\L$
  and $\G_+(V,h)$, this implies:
\beqa
\L\cap C^\times &=&\{a\in C^\times|\Ad(a)(W)=W\}=\nn\\
&=& \gamma(\{x\in \Cl_+(V,h)^\times|\Ad(x)(V)=V\})=\gamma(\G_+(V,h))~~, 
\eeqa
showing that $\ker\Ad_s=\L_\gamma$.  We have
$\omega=\gamma(\nu)=\epsilon_\gamma \id_S$, hence
$\Ad(\omega)=\id_{\End_\R(S)}$ and the volume grading of $\End_\R(S)$
(see Subsection \ref{subsec:EndVolGrading}) is given by
$\End^0_\R(S)=\End_\R(S)$ and $\End_\R^1(S)=0$. Thus the volume
grading of the Lipschitz group (see Subsection
\ref{subsec:LipVolGrading}) is given by $\L^0=\L\cap \End^0_\R(S)=\L$
and $\L^1=\L\cap \End^1_\R(S)=\emptyset$. Since $\G_+(V,h)\subset
\Cl_+(V,h)$, it follows that $\gamma$ restricts to an isomorphism of
groups between $\G_+(V,h)$ and $\ker
\Ad_s=\L_\gamma=\gamma(\G_+(V,h))$.
\end{enumerate}
\qed
\end{proof}

\noindent We have:
\be
\Aut_\Alg(\S)=\threepartdef{\{\id_\S\}}{\S\simeq \R}{\{\id_\S, c \}}{\S\simeq \C}{\U(\S)/\{-\id_S,\id_S\}}{\S\simeq \H}~~,
\ee
For the third entry, recall that $\H$ is central simple as an $\R$-algebra, so all
of its $\R$-algebra automorphisms are inner by the Noether-Skolem
theorem. We have an exact sequence:
\be
1\longrightarrow \R^\times \hookrightarrow \H^\times \stackrel{\Ad}{\longrightarrow} \Aut_\Alg(\H)\longrightarrow 1~~,
\ee
which restricts to an exact sequence: 
\be
1\longrightarrow \{-1,1\} \hookrightarrow \U(\H) \stackrel{\Ad}{\longrightarrow} \Aut_\Alg(\H)\longrightarrow 1~~
\ee
and hence induces isomorphisms $\Aut_\Alg(\H)\simeq
\H^\times/\R^\times\simeq \U(\H)/\{-1,1\}$.  We have
$\U(\H)=\Sp(1)\simeq \Spin(3)$ and hence $\Aut_\Alg(\H)\simeq
\Spin(3)/\{-1,1\}\simeq \SO(3,\R)$.  The isomorphism
$\Aut_\Alg(\H)\simeq \SO(3,\R)$ takes $\alpha\in \Aut_\Alg(\H)$ into
$\alpha|_{\Im \H}\in \SO(\Im\H)$.

\begin{thm}
\label{thm:PinSchur}
For any pin representation, we have a short exact sequence:
\ben
\label{SchurSeq}
1\longrightarrow \L_\gamma \hookrightarrow \L\stackrel{\Ad_s}{\longrightarrow} \Aut_\Alg(\S)\longrightarrow 1~~.
\een
Moreover: 
\begin{enumerate}[1.]
\itemsep 0.0em
\item In the normal simple case, we have $\Aut_\Alg(\S)=\Aut_\Alg(\R)\simeq_{\Gp} 1$ and:
\beqa
\L~&=&\L_\gamma=\gamma(\G(V,h))\simeq_{\Gp} \G(V,h)\\
\L^0&=&\gamma(\G_+(V,h))\simeq_{\Gp} \G_+(V,h)\\
\L^1&=&\gamma(\G_-(V,h))\simeq_{\Set} \G_-(V,h)~~.
\eeqa
\item In the complex case, we have $\Aut_\Alg(\S)\simeq \Aut_\Alg(\C)=\{\id_S,c\}\simeq_{\Gp}\Z_2$ and: 
\beqa
\L_\gamma &=& \L^0=\gamma(\G^e(V,h))\simeq_{\Gp} \G^e(V,h)=\G(V,h)\U(1)\nn\\
\L^1 &=& u\L^0=\L^0u\simeq_{\Set}\L^0~~.
\eeqa
(where $u\in A$ is as in Proposition
  \ref{Alist}) and:
\ben
\label{AdsComplex}
\Ad_s(a)=\twopartdef{\id_\S}{a\in \L^0}{c}{a\in \L^1}~~.  \een
\item In the quaternionic simple case, we have $\Aut_\Alg(\S)\simeq \Aut_\Alg(\H)\simeq_{\Gp}\SO(3,\R)$ and:
\beqa
\L_\gamma&=&\gamma(\G(V,h))\simeq_{\Gp} \G(V,h)\nn\\
\L~&=&\gamma(\G(V,h))\U(\S)\simeq_{\Gp} \gamma(\G(V,h))\cdot \U(\S)\simeq_{\Gp} \G(V,h)\cdot \Sp(1)\nn\\
\L^0&=&\gamma(\G_+(V,h))\U(\S)\simeq_{\Gp} \gamma(\G_+(V,h))\cdot \U(\S) \simeq_{\Gp} \G_+(V,h)\cdot \Sp(1)\nn\\
\L^1&=&\gamma(\G_-(V,h))\U(\S)\simeq_{\Set} \gamma(\G_-(V,h))\times \U(\S)\simeq_{\Set} \G_-(V,h)\times \Sp(1)~~.
\eeqa
\item In the normal non-simple case, we have  $\Aut_\Alg(\S)=\Aut_\Alg(\R)\simeq_{\Gp} 1$ and:
\be
\L=\L_\gamma=\L^0=\gamma(\G_+(V,h))\simeq_{\Gp} \G_+(V,h)~~.
\ee
\item In the quaternionic non-simple case, we have $\Aut_\Alg(\S)=\Aut_\Alg(\H)\simeq_{\Gp}\SO(3,\R)$:
\beqa
\L_\gamma&=&\gamma(\G_+(V,h))\simeq_{\Gp} \G_+(V,h)\nn\\
\L~&=&\L^0=\gamma(\G_+(V,h))\U(\S)\simeq_{\Gp} \gamma(\G_+(V,h))\cdot \U(\S)\simeq_{\Gp} \G_+(V,h)\cdot \Sp(1)~~.
\eeqa
\end{enumerate}
\end{thm}

\begin{proof}
The sequence \eqref{SchurSeq} follows from \eqref{kAds} if we show
that $\Ad_s$ is surjective. For this, we distinguish the cases:
\begin{enumerate}[A.]
\itemsep 0.0em
\item The normal simple case. We have
  $\Aut_\Alg(\R)=\{\id_\R\}\simeq_{\Gp} 1$ so surjectivity is
  automatic.  Proposition \ref{prop:Ge} shows that $\G^e(V,h)=\G(V,h)$
  (since $d$ is even in the normal simple case) while Proposition
  \ref{KAds} gives $\L_\gamma=\gamma(\G^e(V,h))=\gamma(\G(V,h))\simeq_{\Gp}
  \G(V,h)$.  Since $\Aut_\Alg(\R)\simeq_{\Gp} 1$, we have
  $\L_\gamma=\ker \Ad_s=\L$ and hence $\L=\L_\gamma=\gamma(\G(V,h))$.
  In this case, $\gamma$ is a unital isomorphism of algebras from
  $\Cl(V,h)$ to $\End_\R(S)$, which implies (since
  $\omega=\gamma(\nu)$ and $d$ is even) that we have
  $\End_\R^0(S)=\gamma(\Cl_+(V,h))$ and
  $\End_\R^1(S)=\gamma(\Cl_-(V,h))$ (cf. Subsection
  \ref{subsec:ClVolGrading}). Thus
  $\L^0=\L\cap \End_\R^0(S)=\gamma(\G(V,h)\cap
  \Cl_+(V,h))=\gamma(\G_+(V,h))\simeq_{\Gp} \G_+(V,h)$ and
  $\L^1=\L\cap \End_\R^1(S)=\gamma(\G(V,h))\cap
  \Cl_-(V,h))=\gamma(\G_-(V,h))\simeq_\Set \G_-(V,h)$.
\item The complex case. In this case, we have $\Ad_s(u)=c$ and
  $\Aut_\Alg(\S)=\{\id_S,c\}$.  Since $u\in \L$, it follows that
  $\Ad_s$ is surjective and \eqref{SchurSeq} holds. Recalling that
  $J=\omega=\gamma(\nu)$ (see Proposition \ref{Schur}), we have
  $\L^0=\L\cap \ker (\Ad(\omega)-\id_S)=\L\cap \ker(\Ad(J)-\id_S)=
  \L\cap \End_\S(S)=\L\cap C=\ker
  \Ad_s=\L_\gamma=\gamma(\G^e(V,h))\simeq_{\Gp} \G^e(V,h)$, where we
  used Proposition \ref{KAds}. The relation $\Ad_s(ua)=\Ad_s(u)\circ
  \Ad_s(a)=c\circ \Ad_s(a)$ ($a\in \L$) together with the sequence
  \eqref{SchurSeq} and the fact that $c^2=\id_\S$ implies $u\L^0=\L^1$
  and also gives relation \eqref{AdsComplex}.  Similarly, the relation
  $\Ad_s(au)=\Ad_s(a)\circ \Ad_s(u)=\Ad_s(a)\circ c$ for $a\in \L$
  implies $\L^0u=\L^1$. Thus $\L^1=u\L^0=\L^0u$ and \eqref{AdsComplex}
  holds.
\item The normal non-simple case. That $\L=\L^0$ and
  $\L_\gamma=\gamma(\G_+(V,h))\simeq_{\Gp} \G_+(V,h)$ follows from
  Proposition \ref{KAds}. In this case, we have $\Aut_\Alg(\R)=\{\id_\R\}$ 
and hence $\L_\gamma=\ker\Ad_s=\L$. 
\item The quaternionic (simple or non-simple) case. In this case,
  every $\varphi\in \Aut_\Alg(\S)$ has the form $\varphi=\Ad(s)$ for
  some $s\in \U(\S)$ (because all automorphisms of $\S\simeq \H$ are
  inner). Since $\U(\S)\subset \S^\times\subset \L$, we have
  $\varphi\in \Ad_s(\L)$, which proves that $\Ad_s$ is surjective and
  hence \eqref{SchurSeq} holds. We have $\S\cap
  C=\S\cap \End_\S(S)=Z(\S)=\R\id_S$ since $\S\simeq \H$ and
  $Z(\H)=\R$.  This implies $\S^\times \cap C=\R^\times\id_S$ and
  $\U(\S)\cap C=\{-\id_S,\id_S\}$. The map $f:\L_\gamma \times
  \U(\S) \rightarrow \L$ given by:
\be
f(a_0,s)\eqdef s a_0=a_0 s~~
\ee
is a morphism of groups since every element of $\S$ commutes with
every element of $\L_\gamma\subset C$. Given $a\in \L$, we can write
$\Ad_s(a)\in \Aut_\Alg(\S)$ as $\Ad_s(a)=\Ad_s(s)$ for some $s\in
\U(\S)\subset \L$. Then $\Ad_s(s^{-1}a)=\id_\S$ and hence $s^{-1}a\in
\ker\Ad_s=\L_\gamma$ (see Proposition \ref{KAds}), thus
$a=sa_0=f(a_0,s)$ for some $a_0\in \L_\gamma$. This shows that $f$ is
surjective and hence $\L=\U(\S) \L_\gamma=\L_\gamma\U(\S)$. On the
other hand, $f(s,a_0)=\id_S$ implies $a_0=s^{-1}\in \U(\S)\cap
\L_\gamma\subset \U(\S) \cap C=\{-\id_S,\id_S\}$, which gives $\ker
f=\{(-1, -\id_S),(1,\id_S)\}$ upon noticing that
$f(-a_0,-s)=f(a_0,s)$. Thus $f$ induces an isomorphism $\L\simeq
\L_\gamma \cdot \U(\S)$ and we have
$\L=\L_\gamma\U(\S)=\U(\S)\L_\gamma\simeq \L_\gamma\cdot \U(\S)$.

\

In the quaternionic simple case, Propositions \ref{KAds} gives
$\L_\gamma=\gamma(\G^e(V,h))\simeq \G^e(V,h)$ while Proposition
\ref{prop:Ge} gives $\G^e(V,h)=\G(V,h)$ since $d$ is even.  In this
case, we have $\omega=\gamma(\nu)\in \L_\gamma$. Since $\gamma$
restricts to a bijection between $\G(V,h)$ and $\L_\gamma$ while
$\U(\S)\subset \End_\R^0(S)$, we have
$\L^\kappa=\L\cap \End_\R^\kappa(S)=(\L_\gamma\cap \End_\R^\kappa(S))\U(\S)$ for all $\kappa\in \Z_2$,
which gives $\L^0=\gamma(\G_+(V,h))\U(\S)\simeq_{\Gp} \gamma(\G_+(V,h))\cdot
\U(\S)$ and $\L^1=\gamma(\G_-(V,h))\U(\S)\simeq_{\Set}
\gamma(\G_-(V,h))\times \U(\S)$ since $\{-1,1\}\subset \G_+(V,h)$. We
thus obtain the statements at point 3. In the quaternionic non-simple
case, Proposition \ref{KAds} gives $\L_\gamma=\gamma(\G_+(V,h))\simeq
\G_+(V,h)$ and $\L=\L^0$, so we obtain the statements at point 5.
\end{enumerate}
\qed
\end{proof}

\noindent For $\alpha\in \{-1,1\}$, let: 
\be
\G_2(\alpha)=\twopartdef{\G_{2,0}\simeq_{\Gp} \R_{>0}\times \Pin_2(+)\simeq_{\Gp} \R_{>0}\times \O_2(+)}{\alpha=+1}{\G_{0,2}\simeq_{\Gp} \R_{>0}\times \Pin_2(-)\simeq_{\Gp} \R_{>0}\times \O_2(-)}{\alpha=-1}~~,
\ee
where $\O_2(\alpha)$ is the isomorphic model of $\Pin_2(\alpha)$ discussed in Section \ref{sec:spino} and $\G_{p,q}$
denotes the ordinary Clifford group of $\R^{p,q}$.

\begin{prop}
\label{prop:Lc}
In the complex case, we have an isomorphism of groups: 
\ben
\label{Lc}
\L\simeq_{\Gp} [\G_+(V,h)\times \G_2(\alpha_{p,q})]/\{(\lambda 1,\lambda^{-1} 1)|\lambda\in \R^\times\}\simeq_{\Gp} \R_{>0}\times \Spin^o(V,h)~~,
\een
where $\alpha_{p,q}=(-1)^{\frac{p-q+1}{4}}=\twopartdef{-1}{p-q\equiv_8
  3}{+1}{p-q\equiv_8 7}$ and $\Spin^o(V,h)\eqdef
\Spin^o_{\alpha_{p,q}}(V,h)$ is the adapted $\Spin^o$ group of $(V,h)$
defined in Subsection \ref{sec:spinoadapted}.
\end{prop}

\begin{proof}
In the complex case, $\gamma$ is faithful and hence induces a unital isomorphism of $\R$-algebras
from $\Cl(V,h)$ to $C$, which restricts to a unital isomorphism of
$\R$-algebras between $Z(V,h)=\R\oplus \R\nu$ and $\S=\R\oplus \R\omega$ (this
isomorphism takes $\nu$ into $\omega=J$). Let $s:\C
\stackrel{\sim}{\rightarrow} \S$ be the unital isomorphism of
$\R$-algebras which takes $1$ into $\id_\S$ and $i$ into $J$.  Since $d$ is odd, 
we have $\G_-(V,h)=\nu \G_+(V,h)$ and hence
$\G^e(V,h)=Z(V,h)^\times \G(V,h)=Z(V,h)^\times \G_+(V,h)$. Thus 
$\gamma(\G^e(V,h))=\S^\times \gamma(\G_+(V,h))$. We have
$Z(V,h)^\times\cap \G_+(V,h)=\R^\times$. Theorem \ref{thm:PinSchur}
implies that any $b\in \L$ is of the form $s(z) \gamma(a)$ or $s(z)
\gamma(a) u$ for some $z\in \C^\times$ and some $a\in \G_+(V,h)$. Hence 
the map $\varphi:\G_+(V,h)\times \C^\times \times \Z_2\rightarrow \L$ given by:
\be
\varphi(a,z,{\hat 0})=s(z)\gamma(a)~~,~~\varphi(a,z,{\hat 1})=s(z)\gamma(a)u
\ee
is surjective. We have:
\begin{eqnarray}
 (s(z_1)\gamma(a_1))(s(z_2)\gamma(a_2)) &=& s(z_1z_2)\gamma(a_1a_2)\, , \nonumber\\
 (s(z_1)\gamma(a_1))(s(z_2)\gamma(a_2)u) &=& s(z_1z_2)\gamma(a_1a_2) u\, ,\nonumber\\
(s(z_1)\gamma(a_1)u)(s(z_2)\gamma(a_2)) &=& s(z_1\bar{z_2})\gamma(a_1a_2) u\, , \nonumber\\ (s(z_1)\gamma(a_1)u)(s(z_2)\gamma(a_2)u) &=& \alpha_{p,q} s(z_1\bar{z_2})\gamma(a_1a_2)\, ,
\end{eqnarray}
and $\varphi(1,1,{\hat 0})=\id_S$, which shows that $\varphi$ is a
morphism of groups from $\G_+(V,h)\times \G_2(\alpha_{p,q})$ to $\L$
(notice that $\G_2(\pm)\simeq_{\Set}\C^\times\times \Z_2$). We have
$\ker\varphi=\{(a,z,{\hat 0})\in \G_+(V,h)\times \C^\times\times \Z_2|
\gamma(a)=s(z^{-1})\}=\{(\lambda,\lambda^{-1},{\hat 0})\equiv(\lambda
1_{\G_+(V,h)},\lambda^{-1}1_{\G_2(\alpha_{p,q})}|\lambda\in
\R^\times\}$, since $\gamma(\G_+(V,h))\cap
\S^\times=\gamma(\G_+(V,h)\cap Z(V,h)^\times)=\R^\times\id_S$ and
$\gamma$ and $s$ are injective. This gives the first isomorphism in
\eqref{Lc}. The second isomorphism follows from
$\G_+(V,h)\simeq_{\Gp}=\R_{>0}\times \Spin(V,h)$ and
$\G_2(\pm)\simeq_{\Gp} \R_{>0}\times \Pin_2(\pm)$. \qed
\end{proof}

\begin{prop}
Assume that we are in the complex case with $p-q\equiv_8 7$. Then the
sequence \eqref{SchurSeq} splits and we have $\L\simeq_{\Gp}
\G^e(V,h)\rtimes_\varphi \Z_2\simeq_{\Gp} \R_{>0}\times
(\Spin^c(V,h)\rtimes \Z_2)$, where $\varphi:\Z_2\rightarrow
\Aut_\Gp(\G^e(V,h))$ is the group morphism given by:
\be
\varphi({\hat 0})=\id_{\G^e(V,h)}~,~\varphi({\hat 1})=\pi|_{\G^e(V,h)}~~
\ee
and $\pi$ is the parity automorphism of $\Cl(V,h)$. 
\end{prop}

\begin{proof}
In this case, we have $u^2=+\id_S$ and the map
$\psi:\Aut_\Alg(\S)=\{\id_S,c\}\rightarrow \L$ given by
$\psi(\id_S)=\id_S$ and $\psi(c)=u$ is a group morphism which splits
the sequence \eqref{SchurSeq}. Thus $\L\simeq
\gamma(\G^e(V,h))\rtimes_{{\hat \varphi}} \{\id_S,c\}$, where ${\hat
  \varphi}:\{\id_S,c\}\rightarrow \Aut_{\Gp}(\gamma(\G^e(V,h)))$ is
the morphism of groups given by ${\hat
  \varphi}(\id_S)=\id_{\G^e(V,h)}$ and ${\hat
  \varphi}(c)=\Ad(u)|_{\gamma(\G^e(V,h))}$. Since $\Ad(u)(w)=-w$ for
all $w\in W=\gamma(V)$ and $\gamma$ is injective, we have $\Ad(u)\circ
\gamma=\gamma\circ \pi$, which shows that $\L\simeq
\G^e(V,h)\rtimes_\varphi \Z_2$. The second isomorphism given in the
statement follows from Proposition \ref{prop:GePinSpin}.\qed
\end{proof}

\begin{remark} 
The isomorphism $\L\simeq_{\Gp} \R_{>0}\times (\Spin^c(V,h)\rtimes
\Z_2)$ also follows from Proposition \ref{prop:Lc} and from the fact
that $\Spin^o(V,h)\simeq_{\Gp}\Spin^c(V,h)\rtimes \Z_2$ when
$p-q\equiv_8 7$ (see \cite{spino}).
\end{remark}

\subsection{The canonical pairing}

Let $\gamma:\Cl(V,h)\rightarrow \End_\R(S)$ be a pin representation.  
Recall that a $\gamma$-admissible pairing on $S$ is a $\gamma$-adapted pairing (in the 
sense of Subsection \ref{subsec:adaptedpairings}) which has a definite ``isotropy''. See 
\cite{AC1,AC2} and \cite{gf} for details on admissible pairings. Recall that 
$c:\S\rightarrow \S$ is the $\R$-linear map which corresponds to conjugation when 
$\S\simeq_{\Alg} \C,\H$, while $c=\id_\S$ when $\S\simeq_{\Alg} \R$.

\begin{pdef}
\label{prop:canpairing}
Up to rescaling by a non-zero real number, there exists a single
$\gamma$-admissible pairing $\cB_e$ (called {\em canonical pairing})
which has the following properties, where $\epsilon_e$ is the type of
$\cB_e$ and $~^{t_e}:\End_\R(S)\rightarrow \End_\R(S)$ denotes
$\cB_e$-transposition:
\begin{enumerate}[(a)]
\itemsep 0.0em
\item We have: 
\be
\epsilon_e=\twopartdefmod{+\epsilon_d=-(-1)^{\left[\frac{d}{2}\right]}}{~\mathrm{in~the~simple~cases}}{-\epsilon_d=+(-1)^{\left[\frac{d}{2}\right]}}{~\mathrm{in~the~non-simple~cases}}~~,
\ee
where $\epsilon_d$ was defined in \eqref{epsilond}.
\item $s^{t_e}=c(s)~~\forall s\in \S$ (which amounts to $\U(\S)\subset
  \O(S,\cB_e)$).
\item In the complex case, we have $u^{t_e}=u^{-1}$, i.e. $u\in
  \O(S,\cB_e)$, where $u$ is as in Proposition \ref{Alist}.
\end{enumerate}
\end{pdef}

\begin{proof} We distinguish the cases: 

\begin{enumerate}[1.]
\itemsep 0.0em
\item Normal simple case ($p-q\equiv_8 0,2$). In this case, $d$ is
  even and there exist two admissible pairings $\cB_\pm$ (each
  determined up to rescaling by a non-zero real number) which are
  distinguished by their type $\epsilon_{\pm}=\pm 1$ (see
  \cite{AC1,AC2, gf}). Hence we must have
  $\cB_e=\cB_{(-1)^{1+\left[\frac{d}{2}\right]}}$. In this case,
  $\S=\R\id_S$ and $c=\id_\S$, thus condition (b) is trivially
  satisfied.
\item Normal non-simple case ($p-q\equiv_8 1$). In this case, $d$ is
  odd and there exists a single (up to scale) admissible pairing
  $\cB$, whose type is given by
  $\epsilon=(-1)^{\left[\frac{d}{2}\right]}$ (see \cite{AC1,AC2,
    gf}). Hence $\cB_e=\cB$ satisfies condition (a). It
  obviously also satisfies condition (b).
\item Complex case ($p-q\equiv_8 3,7$). Up to scale, there exist four
  admissible pairings $\cB_\kappa$ ($\kappa=0\ldots 3$) whose types
  equal $\epsilon_\kappa=(-1)^{1+\left[\frac{\kappa}{2}\right]}$ (see
  \cite{AC1,AC2, gf}). Hence we must have
  $\cB_e=\cB_{\kappa_e}$ for some $\kappa_e\in \{0,1,2,3\}$.
  Condition (a) requires:
\ben
\label{c1}
\left[\frac{\kappa_e}{2}\right]\equiv_2 \left[\frac{d}{2}\right]=\frac{d-1}{2}~~, 
\een
where we used the fact that $d$ is odd. On the other hand, relations \cite[(3.16)]{gf} give: 
\ben
\label{c2}
u^{-t_\kappa}=(-1)^{\kappa+\frac{p-q+1}{4}} u~~,~~J^{-t_\kappa}= (-1)^{\left[\frac{\kappa}{2}\right]+\left[\frac{d}{2}\right]} J~~,
\een
where $J=\omega=\gamma(\nu)$ and $~^{t_\kappa}$ denotes
$\cB_\kappa$-transpose. Thus condition (c) requires $\kappa_e\equiv_2
\frac{p-q+1}{4}$. This is equivalent with
$\kappa_e=2\left[\frac{\kappa_e}{2}\right]+t$, where $t\eqdef
\frac{p-q+1}{4} \mod 2=\twopartdef{1}{p-q\equiv_8 3}{0}{p-q\equiv_8
  7}$.  Thus \eqref{c1} gives (recall that $d>0$): 
\be
\kappa_e=d+t-1 \!\mod 4=\twopartdef{d ~\mod 4}{~~p-q\equiv_8 3}{d\!-\!1\!\!\mod 4}{~~p-q\equiv_8 7}~~. 
\ee
Let $\cB_e=\cB_{\kappa_e}$, where $\kappa_e$ is given by this choice. Then
conditions (a) and (c) are satisfied. On the other hand, relation
\eqref{c1} and the second relation in \eqref{c2} give $J^{-t_e}=J$,
which (since $J^2=-\id_S$) amounts to $J^{t_e}=-J$, hence condition
(b) is also satisfied.
\item Quaternionic simple case ($p-q\equiv_8 4,6$). In this case, $d$
  is even and there exist (up to scale) eight admissible pairings
  $\cB_k^\epsilon$ (where $\epsilon \in \{-1,1\}$ and $\kappa\in
  \{0,1,2,3\}$), of which only the two so-called fundamental pairings
  (see \cite[Subsection 3.4]{gf}) $\cB_0^+$ and
  $\cB_0^-$ satisfy condition (b). Indeed, let $J_1, J_2, J_3\in \S$
  be the elements of $\S$ which correspond to the quaternion units
  through the isomorphism of Proposition \ref{Schur}. Then it was
  shown in Subsection 3.4.3 of op. cit. that $J_k$ $(k=1,2,3$) are
  $\cB_0^\pm$-orthogonal and hence satisfy condition (b). Together
  with relation \cite[(3.38)]{gf}, this implies that the
  $\cB_k^\epsilon$-transpose of $J_l$ equals $J_kJ_lJ_k$, which
  implies that only the fundamental pairings satisfy condition
  (b). The fundamental pairings have type $\epsilon_{\cB_0^\pm}=\pm
  1$, so condition (a) requires
  $\cB_e=\cB_0^{(-1)^{1+\left[\frac{d}{2}\right]}}$.
\item Quaternionic non-simple case $(p-q\equiv_8 5$). In this case $d$
  is odd and there exist four admissible pairings and a similar
  argument using the results of \cite[Subsection
    3.4]{gf} shows that only the basic pairing $\cB_0$
  satisfies condition (b). Hence we must take $\cB_e=\cB_0$.  The type
  of the basic pairing is
  $\epsilon_{0}=(-1)^{\left[\frac{d}{2}\right]}$, so condition (a)
  is satisfied.
\end{enumerate}
\qed
\end{proof}

\begin{remark}
Since $\tau(\nu)=(-1)^{\left[\frac{d}{2}\right]}\nu$, we have
$\omega^{t_e}=\epsilon_e^d (-1)^{\left[\frac{d}{2}\right]}\omega$ and
condition (a) implies:
\ben
\label{epsilone}
\omega^{t_e}=\threepartdef{(-1)^{\left[\frac{d}{2}\right]}\omega}{p-q\equiv_8 0,2,4,6}{-\omega}{p-q\equiv_8 3,7}{+\omega}{p-q\equiv_8 1,5}~~.
\een
Since $\omega=\sigma_{p,q}\omega^{-1}$, this gives:
\ben
\label{omegate}
\omega^{t_e}=\beta_{p,q}\omega^{-1}~~,~~\mathrm{where}~~\beta_{p,q}\eqdef \threepartdef{(-1)^p}{p-q\equiv_8 0,2,4,6}{-\sigma_{p,q}}{p-q\equiv_8 3,7}{+\sigma_{p,q}}{p-q\equiv_8 1,5}~~.
\een
\end{remark}

\subsection{The canonical Lipschitz norm}

\begin{definition}
The {\em canonical Lipschitz norm} $\cN_e$ is the Lipschitz norm
determined by the canonical pairing $\cB_e$.
\end{definition}

\begin{remark}
\label{rem:norms}
In the simple cases, we have $\tau_{\cB_e}=\tau_e$ and
hence $N_{\cB_e}=N_e$. In the non-simple cases, we have
$\tau_{\cB_e}=\tau_e\circ \pi$. Together with the results of Subsection 
\ref{subsec:ImprovedNorm} and with the fact that $\U(\S)\subset \O(S,\cB_e)$, this implies:
\begin{enumerate}[1.]
\itemsep 0.0em
\item In the simple cases, we have $\cN_e\circ \gamma=\gamma\circ N_e$.
\item In the non-simple cases, we have  $\cN_e\circ \gamma|_{\Cl_+(V,h)}=\gamma\circ N|_{\Cl_+(V,h)}$.
\item In all cases, we have $\cN_e|_\S=\rM$, where $\rM$ is the square of the norm of $\S$. 
\end{enumerate}
\end{remark} 

\begin{prop}
\label{prop:cNe}
For any pin representation, we have $\cN_e(\L)\subset \R^\times$, 
hence the restriction of $\cN_e$ gives a group morphism:
\be
\cN_e:\L\rightarrow \R^\times\id_S\simeq \R^\times~~.
\ee
The restriction of
$\cN_e$ to $\L_\gamma$ is determined as follows:
\begin{enumerate}[1.]
\itemsep 0.0em
\item In the simple cases, we have $\cN_e\circ
  \gamma|_{\G^e(V,h)}=N_e|_{\G^e(V,h)}$.
\item In the non-simple cases, we have $\cN_e\circ
  \gamma|_{\G_+(V,h)}=N|_{\G_+(V,h)}$.
\end{enumerate}
Moreover, in the complex case we have $\cN_e(u)=\id_S$ while in the quaternionic cases we have $\cN_e|_\S=\rM$, 
where $u$ is as in Proposition \ref{Alist}.
\end{prop}

\begin{proof} We consider each case in turn.

\begin{enumerate}[1.]
\item Normal simple case. In this case, $\L=\gamma(\G(V,h))$ (see
  Theorem \ref{thm:PinSchur}) and the conclusion follows because
  $N_e(\G(V,h))\subset \R^\times$ (see Proposition \ref{prop:improved})
  using Remark \ref{rem:norms}.
\item Normal non-simple case. In this case, $\L=\gamma(\G_+(V,h))$ by
  Theorem \ref{thm:PinSchur} and the conclusion follows from
  $N(\G_+(V,h))\subset \R^\times$ using Remark \ref{rem:norms}.
\item Complex case. By Remark \ref{rem:norms} and Proposition \ref{prop:improved}, we have
  $\cN_e(\gamma(\G^e(V,h))\subset \gamma(N_e(\G^e(V,h))\subset \R^\times\id_S$. On the other hand, we 
have $\cN_e(u)=u^{t_e}u=\id_S$. For any $a\in \gamma(\G^e(V,h))$, we thus have
  $\cN_e(a)\in \R^\times\id_S$ and:
\be
\cN_e(au)=u^{t_e}a^{t_e}au=\pi(a)^{t_e}\pi(a)=\pi(a^{t_e}a)=\pi(\cN_e(a))=\cN_e(a)\in \R^\times\id_\S~~.
\ee
This implies $\cN_e(\L)\subset \R^\times \id_S$ by Theorem \ref{thm:PinSchur}.
\item Quaternionic (simple or non-simple) case. In this case, we have
  $\cN_e(s)=s^{t_e}s=c(s)s=|s|^2$ for all $s\in \S$, where $|~|$ is
  the norm on $\S$ which corresponds to the quaternionic norm through
  the isomorphism of Proposition \ref{Schur}. Thus $\cN_e(\S)\subset
  \R_{\geq 0}$. By Theorem \ref{thm:PinSchur}, we have
  $\L=\U(\S)\L_\gamma$. Consider the subcases:
\begin{enumerate}[(a)]
\item In the quaternionic simple case, $d$ is even and
  $\L_\gamma=\gamma(\G(V,h))$. Thus $\cN_e(\L)\subset
  \cN_e(\S)\gamma(N_e(\G(V,h)))\subset \R^\times$, where we used
  Remark \ref{rem:norms}.
\item In the quaternionic non-simple case, we have
  $\L_\gamma=\gamma(\G_+(V,h))$. Since $N(\G_+(V,h))\subset
  \R^\times$, we again conclude $\cN_e(\L)\subset \R^\times$ using
  Remark \ref{rem:norms}.
\end{enumerate}
\end{enumerate}
The remaining statements follow from  Remark \ref{rem:norms} and from relation \eqref{Lgamma}. \qed
\end{proof}

\noindent The composition of $\cN_e|_{\L}$ with the absolute value morphism
$|~|:\R^\times \rightarrow \R_{>0}$ gives a morphism of groups
$|\cN_e|:\L\rightarrow \R_{>0}$.

\subsection{The reduced Lipschitz group}
\label{subsec:ReducedLipschitz}

\begin{definition}
The {\em reduced Lipschitz group} $\cL$ is the kernel of the group
morphism $|\cN_e|:\L\rightarrow \R_{>0}$.
\end{definition}

\noindent We have a short exact sequence: 
\ben
\label{cLseq}
1\longrightarrow \cL \longrightarrow \L\stackrel{|\cN_e|}{\longrightarrow} \R_{>0}\longrightarrow 1~~.
\een
Every $a\in \L$ can be written as $a=\sqrt{|\cN_e(a)|} a_0$ for some
uniquely determined $a_0\in \cL$. Thus:
\ben
\label{LcL}
\L=\R_{>0}\cL\simeq \R_{>0}\times \cL~~.
\een
The morphism of groups $\pi_0:\L\rightarrow \cL$ given by projection on the second factor: 
\ben
\label{pi0}
\pi_0(a)=a_0=\frac{1}{\sqrt{|\cN_e(a)|}}a
\een
will be called the {\em normalization morphism}. The adjoint
representation $\Ad:\L\rightarrow \Aut_\R(\End_\R(S))$ of the
Lipschitz group factors through this morphism:
\ben
\label{Adpi0}
\Ad(a)=\Ad(\pi_0(a))~~\forall a\in \L~~.
\een
The volume grading of $\L$ induces a $\Z_2$-grading
$\cL=\cL^0\sqcup \cL^1$, where $\cL^\kappa =\cL\cap \L^\kappa=\{a\in
\L^\kappa| |\cN_e(a)|=1\}$ for all $\kappa\in \Z_2$.

\begin{cor}
$\cL$ is homotopy equivalent with $\L$. 
\end{cor}

\proof Follows from \eqref{LcL}. \qed

\subsection{The canonical presentation of the reduced Lipschitz group}
\label{sec:canpres}

The following result describes the reduced Lipschitz groups of all
real pin representations, and hence their automorphism group in the
category $\ClRep$.

\begin{thm}
\label{thm:cLpres}
The following maps $\varphi:\Lambda(V,h)\rightarrow \cL$ are
well-defined and give isomorphisms of groups from the canonical spinor
group $\Lambda(V,h)$ of Section \ref{sec:enlargedspinor} to the
reduced Lipschitz group $\cL$ of the pin representation
$\gamma:\Cl(V,h)\rightarrow \End_\R(S)$:
\begin{enumerate}[1.]
\itemsep 0.0em
\item In the normal simple case, we have $\Lambda(V,h)=\Pin(V,h)$ and:
\be
\varphi(x)=\gamma(x)~~\forall x\in \Lambda(V,h)=\Pin(V,h)~~.
\ee
\item In the complex case, we have $\Lambda(V,h)=\Spin^o(V,h)$ and: 
\be
\varphi_{\nu,u}([x,\psi_\alpha(e^{\bi\theta}, \kappa)])=\gamma(x)e^{\bi\theta J_\nu}u^\kappa ~~\forall x\in \Spin(V,h)~~,~~\forall \theta\in \R~~,~~\forall \kappa\in \Z_2~~,
\ee
where $J_\nu=\gamma(\nu)$, $\nu$ is the Clifford volume element of
$(V,h)$ determined by an orientation of $V$ and $u\in A$ is as in
Proposition \ref{Alist}. Here, we have $(e^{\bi\theta},\kappa)\in
\O_2(\alpha)$ and
$\psi_\alpha:\O_2(\alpha)\stackrel{\sim}{\rightarrow}
\Pin_2(\alpha)$ is the isomorphism of Proposition
\ref{prop:Pin2O2}, where $\alpha\eqdef \alpha_{p,q}$.
\item In the quaternionic simple case, we have
  $\Lambda(V,h)=\Pin^q(V,h)$ and:
\be
\varphi_f([x,q])=\gamma(x) f(q)~~\forall x\in \Pin(V,h)~~,~~\forall q\in \U(\H)=\Sp(1)~~,
\ee 
where $f:\H\stackrel{\sim}{\rightarrow} \S$ is a unital isomorphism of normed $\R$-algebras (thus $f(\Im \H)=U(\gamma)$) as in Proposition \ref{Schur}.
\item In the normal non-simple case,  we have $\Lambda(V,h)=\Spin(V,h)$ and: 
\be
\varphi(x)=\gamma(x)~~\forall x\in \Spin(V,h)~~.
\ee
\item In the quaternionic non-simple case, we have $\Lambda(V,h)=\Spin^q(V,h)$ and:
\be
\varphi_f([x,q])=\gamma(x) f(q)~~\forall x\in \Spin(V,h)~~,~~\forall q\in \U(\H)=\Sp(1)~~,
\ee
where $f:\H\stackrel{\sim}{\rightarrow} \S$ is a unital isomorphism of normed $\R$-algebras (thus $f(\Im \H)=U(\gamma)$) as in Proposition \ref{Schur}.
\end{enumerate}
\end{thm}

\begin{proof} We consider each case in turn. 
\begin{enumerate}
\item Normal simple case. Theorem \ref{thm:PinSchur} gives
  $\L=\gamma(\G(V,h))$. By Proposition \ref{prop:cNe}, we have:
  $|\cN_e|(\gamma(x))=|\cN_e (\gamma(x))|=|N_e(x)|=|N|(x)$ for all
  $x\in \G(V,h)$. Since $\gamma$ is an isomorphism from $\G(V,h)$ to
  $\L$, this gives
  $\cL=\ker(|\cN_e|_{\L}|)=\gamma(\ker(|N_G|))=\gamma(\Pin(V,h))$ and
  the statement follows.
\item Complex case. Theorem \ref{thm:PinSchur} gives
  $\L^0=\gamma(\G^e(V,h))$ and $\L^1=\L^0u$. Since $u^{t_e}u=\id_\S$,
  Proposition \ref{prop:cNe} gives: $\cN_e(au)=\cN_e(a)$ for all
  $a=\gamma(x)\in \L^0$ (with $x\in \G^e(V,h)$) and Proposition
  \ref{prop:cNe} gives $|\cN_e|(au)=|\cN_e|(a)$, which implies
  $\cL^0=\gamma(\Spin^c(V,h))$ and $\cL^1=\gamma(\Spin^c(V,h))u$ since
  $\gamma$ is injective. This gives the desired statement.
\item Quaternionic simple case. Theorem \ref{thm:PinSchur} gives
  $\L=\gamma(\G(V,h))\U(\S)$ while Proposition \ref{prop:cNe} gives
  $|\cN_e|(\gamma(x))=|N_G|(x)$ for $x\in \G(V,h)$, where we used the
  fact that $\U(\S)\subset \O(\cS,\cB_e)$. Since $\gamma$ is
  injective, this gives $\cL=\gamma(\ker |N_G|)=\gamma(\Pin(V,h))$ and
  the statement follows.
\item Normal non-simple case. Theorem \ref{thm:PinSchur} gives
  $\L=\gamma(\G_+(V,h))$ while Proposition \ref{prop:cNe} gives
  $|\cN_e|(\gamma(x))=|N|(x)$ for $x\in \G_+(V,h)$. Since
  $\gamma|_{\G_+(V,h)}$ is injective, this gives $\cL=\gamma(\ker
  |N|_{\G_+(V,h)})=\gamma(\Spin(V,h))$ and the statement follows.
\item Quaternionic non-simple case. Theorem \ref{thm:PinSchur} gives
  $\L=\gamma(\G_+(V,h)\U(\S))$ while Proposition \ref{prop:cNe} gives
  $|\cN_e|(\gamma(x)s)=|N|(x)$ for $x\in \G_+(V,h)$ and $s\in \U(\S)$,
  where we used the fact that $\U(\S)\subset \O(S,\cB_e)$. Since
  $\gamma|_{\G_+(V,h)}$ is injective, this gives $\cL=\gamma(\ker
  |N|_{\G_+(V,h)})\U(\S)=\gamma(\Spin(V,h))\U(\S)$ and the statement
  follows.
\end{enumerate}
\qed
\end{proof}

\begin{definition} 
Any of the group isomorphisms given in the previous proposition is
called an {\em admissible isomorphism} from the enlarged spinor group
$\Lambda(V,h)$ to $\cL$.
\end{definition}

\begin{remark}
In the normal (simple or non-simple case), the unique admissible
isomorphism is given by the restriction of $\gamma$ and hence it is
canonically determined by $\gamma$. In the complex case, the
admissible isomorphisms are determined by $\gamma$, by a choice of
orientation of $V$ and by a choice of element $u\in A$ as in
Proposition \ref{Alist}. In the quaternionic (simple or non-simple
case), there exists an infinite set of admissible isomorphisms, each
of which is determined by $\gamma$ and by a choice of unital
isomorphism $f\in \Isom_\Alg(\H,\S)$ satisfying $f(\Im \H)=U(\gamma)$
(recall that the set of such $f$ forms an $\SO(3,\R)$-torsor).
\end{remark}

\subsection{The pairing induced by $\cB_e$ on $\End_\R(S)$}

The canonical pairing $\cB_e$ induces an $\R$-bilinear, symmetric and non-degenerate
pairing $(~,~)_e:\End_\R(S)\times \End_\R(S)\rightarrow \R$ defined through: 
\ben
\label{cEndPairing}
(T_1,T_2)_e\eqdef \frac{1}{\dim_\R S}\tr(T_1^{t_e}T_2)~~,~~\forall T_1,T_2\in \End_\R(S)~~.
\een
Symmetry of $(~,~)_e$ follows from involutivity of $~^{t_e}$ and from the property $\tr(T^{t_e})=\tr(T)$ for 
all $T\in \End_\R(S)$. 

\begin{prop}
\label{prop:AdOrt}
For any $a\in \cL$ and $T\in \End_\R(S)$, we have: 
\be
\Ad(a)(T)^{t_e}=\Ad(a)(T^{t_e})~~.
\ee
Moreover, the adjoint representation of $\cL$ preserves the pairing \eqref{cEndPairing}:
\be
(\Ad(a)(T_1),\Ad(a)(T_2))_e=(T_1,T_2)_e~~\forall a\in \cL~~,~~\forall T_1, T_2\in \End_\R(S)~~.
\ee
\end{prop}

\begin{proof}
For all $a\in \cL$, we have $\cN_e(a)=\pm \id_S$, hence $a^{t_e}=\pm
a^{-1}$ and $\Ad(a)(T)=aTa^{-1}=\pm a T a^{t_e}$. Thus $\Ad(a)(T)^{t_e}=\pm a
T^{t_e} a^{t_e}=a T^{t_e}a^{-1}=\Ad(a)(T^{t_e})$. This gives:
\begin{eqnarray}
& &(\Ad(a)(T_1),\Ad(a)(T_2))_e = \frac{1}{\dim_\R S}\tr(\Ad(a)(T_1^{t_e})\Ad(a)(T_2))=\nn\\
&=& \frac{1}{\dim_\R S}\tr(\Ad(a)(T_1^{t_e}T_2))= \frac{1}{\dim_\R S}\tr(T_1^{t_e}T_2)=(T_1,T_2)_e\, ,
\end{eqnarray}
where we used the cyclic property of the trace. \qed 
\end{proof}

\

\noindent Let $(~,~)_\S$ and $(~,~)_A$ denote the restrictions of
$(~,~)_e$ to the subspaces $\S$ and $A$ of $\End_\R(S)$. Since
$(~,~)_e$ is invariant under the adjoint representation of $\cL$,
these restricted pairings are invariant respectively under the Schur
and anticommutant representations of $\cL$. In the complex and
quaternionic cases, let $(~,~)_{\Im \S}$ denote the restriction of
$(~,~)$ to $\Im \S$.

\begin{prop}
The pairing $(~,~)_\S$ coincides with the canonical Euclidean scalar
product on $\S$ (the scalar product which induces the norm of
$\S$). Moreover, the Schur representation of the reduced Lipschitz
group $\cL$ preserves $(~,~)_\S$.
\end{prop}

\begin{proof}
The fact that $\Ad_\S$ preserves $(~,~)_\S$ follows from Proposition \ref{prop:AdOrt}.
Since $s^{t_e}=c(s)$ for all $s\in \S$ (see Proposition \ref{prop:canpairing}), we have: 
\ben
\label{pairingS}
(s_1,s_2)_\S=\frac{1}{\dim_\R S}\tr(c(s_1)s_2)~~\forall s_1,s_2\in \S~~.
\een
Thus: 
\begin{enumerate}[1.]
\itemsep 0.0em
\item In the normal simple or non-simple case, we have
  $s_i=\alpha_i\id_S$ with $\alpha_i\in \R$ and
  $(s_1,s_2)_\S=\alpha_1\alpha_2$, which is the Euclidean scalar
  product on $\S\simeq \R$. In this case, we have $\Ad_\S(a)=\id_\S$
  for all $a\in \cL$.
\item In the complex case, we have $s_i=\alpha_i\id_S +\beta_iJ$ with
  $\alpha_i,\beta_i\in \R$, which gives:
\be
c(s_1)s_2=(\alpha_1\alpha_2 +\beta_1\beta_2)\id_S+(\alpha_1\beta_2-\alpha_2\beta_1)J~~.
\ee 
Since $J^{t_e}=c(J)=-J$ , we have $\tr(J)=0$. Thus $(s_1,s_2)_\S=\alpha_1\alpha_2 +\beta_1\beta_2$, which is the 
canonical scalar product on $\S\simeq \C$ (that scalar product which induces the canonically normalized 
norm of the normed algebra $\S$). In this case, Theorem \ref{thm:PinSchur} gives: 
\be
\Ad_\S(a)=\twopartdef{\id_\S}{a\in \cL^0}{c}{a\in \cL^1}~~.
\ee
\item In the quaternionic case, let $J_i$ be an orthonormal and oriented basis of $\Im \S=U(\gamma)$. We have $s_i=\alpha_i\id_S
  +\sum_{k=1}^3\beta_i^kJ_k$ with $\alpha_i,\beta_i^k\in \R$, which
  gives:
\be
c(s_1)s_2=(\alpha_1\alpha_2 +\sum_{k=1}^3\beta_1^k\beta_2^k)\id_S+\sum_{k=1}^3 (\alpha_1
  \beta_2^k-\alpha_2\beta_1^k)J_k-\sum_{k,l,m=1}^3\epsilon_{klm} \beta_1^k\beta_2^l
  J_m~~.  
\ee Since $J_k^{t_e}=c(J_k)=-J_k$, we have $\tr(J_k)=0$. Thus
$(s_1,s_2)_\S=\alpha_1\alpha_2 +\sum_{k=1}^3\beta_1^k\beta_2^k$, which
is the canonical scalar product of the normed algebra $\S$. 
\end{enumerate}
\qed
\end{proof}

\begin{remark}
Relation \eqref{pairingS} implies that $(~,~)_\S$ satisfies the
following identities for all $s_1,s_2\in \S$:
\beqan
\label{pairingSprops}
(c(s_1),c(s_2))_\S=(s_1,s_2)_\S=(s_2,s_1)_\S~~,
\eeqan
which shows that $c:\S\rightarrow \S$ is $(~,~)_\S$-orthogonal (and hence also $(~,~)_\S$-symmetric, since $c^2=\id_\S$).
\end{remark} 

\begin{prop}
In the simple cases, the pairing $(~,~)_A$ on $A=\S u$ agrees up to sign with the pairing $(~,~)_\S$:
\begin{enumerate}[1.]
\itemsep 0.0em
\item In the normal and quaternionic simple cases, we have:
\be
(s_1 \omega,s_2\omega)_A=(-1)^p (s_1,s_2)_\S~~\forall s_1,s_2\in \S~~,
\ee
where $\beta_{p,q}=(-1)^p$ (see \eqref{omegate}).
\item In the complex case, we have:
\be
(s_1 u,s_2 u)_A=(s_1,s_2)_\S~~\forall s_1,s_2\in \S\simeq \C
\ee
\end{enumerate}
Moreover, the anticommutant representation of $\cL$ preserves the
pairing $(~,~)_A$.
\end{prop}

\begin{proof}
The last statement follows from the fact that the adjoint
representation of $\cL$ preserves $(~,~)_e$.  Statement 1. follows
immediately from Proposition \ref{Alist} using the fact that $\omega$
commutes with the elements of $\S$ while
$\omega^{t_e}\omega=\beta_{p,q}\id_S$ (see relation \eqref{omegate}),
where $\beta_{p,q}=(-1)^p$ for $p-q\equiv_8 0,2,4,6$. Statement 2
follows from Proposition \ref{Alist} using the fact that $\Ad_s(u)=c$
while $u^{t_e}u=\id_S$.  \qed
\end{proof}

\subsection{The vector representations of $\L$ and $\cL$}
The following result shows that the vector representation of $\L$
exhausts the full pseudo-orthogonal group in the simple cases; this is
a consequence of the fact that pin representations admit twisting
elements in the simple case (see Corollary \ref{TwistingElements}). In
the non-simple case, the vector representation of $\L=\L^0$ exhausts
the special pseudo-orthogonal group. We identify $\O(W,g)$ with
$\O(V,h)$ using the invertible isometry
$\gamma|_V:V\stackrel{\sim}{\rightarrow} W$, thus viewing the vector
representation of the Lipschitz group as a morphism from $\L$ to
$\O(V,h)$.

\begin{thm}
\label{thm:Lvector}
In the simple case, there exists a short exact sequence: 
\ben
\label{LextSimple}
1\longrightarrow \S^\times \longrightarrow \L \stackrel{\Ad_0}{\longrightarrow} \O(V,h)\longrightarrow 1~~,
\een
which restricts to a short exact sequence: 
\ben
\label{LextSimpleRes}
1\longrightarrow \S^\times \longrightarrow \L^0\stackrel{\Ad_0}{\longrightarrow} \SO(V,h)\longrightarrow 1~~.
\een
In the non-simple case, we have $\L=\L^0$ and there exists a short
exact sequence:
\ben
\label{LextNS}
1\longrightarrow \S^\times \hookrightarrow \L\stackrel{\Ad_0}{\longrightarrow} \SO(V,h)\longrightarrow 1~~.
\een
\end{thm}

\begin{proof}
The simple case follows from Theorem \ref{OrthogonalCover} and
Corollary \ref{TwistingElements}.  In the non-simple case, we have
$\L=\L^0$ by Theorem \ref{thm:PinSchur}. Hence the exact sequence
\eqref{L0ext} of Theorem \ref{OrthogonalCover} becomes \eqref{LextNS}.
\qed
\end{proof}

\noindent Since the adjoint representation of $\L$ factors through the
normalization morphism \eqref{pi0} while the restriction of $\cN_e$ to
$\S$ equals the square norm morphism $\rM:\S\rightarrow \R_{>0}$, we
obtain:

\begin{cor}
In the simple case, there exists a short exact sequence: 
\ben
1\longrightarrow \U(\S) \longrightarrow \cL \stackrel{\Ad_0}{\longrightarrow} \O(V,h)\longrightarrow 1~~,
\een
which restricts to a short exact sequence: 
\ben
1\longrightarrow \U(\S) \longrightarrow \cL^0\stackrel{\Ad_0}{\longrightarrow} \SO(V,h)\longrightarrow 1~~.
\een
In the non-simple case, we have $\cL=\cL^0$ and there exists a short
exact sequence:
\ben
1\longrightarrow \U(\S) \hookrightarrow \cL\stackrel{\Ad_0}{\longrightarrow} \SO(V,h)\longrightarrow 1~~.
\een
\end{cor}

\noindent We have a commutative diagram with exact rows and columns: 
\begin{equation}
\label{diagram:Ndiag}
\scalebox{1.0}{
\xymatrix{
&1 \ar[d]&1\ar[d]& &\\
1 \ar[r] & \U(\S)~\ar[d] \ar[r]~ &~\cL ~\ar[r]^{\!\!\!\!\Ad_0} \ar[d] &\O(V,h) \ar@{=}[d]~\ar[r] &1\\
1 \ar[r]~ & ~\S^\times\ar[r]\ar[d]^{\rM}~& ~\L ~\ar[r]^{\!\!\!\!\Ad_0}\ar[d]^{\cN_e}& \O(V,h)\ar[r]&1\\
& ~~\R^\times~\ar[d]\ar@{=}[r]&~~\R^\times\ar[d]& &\\
& 1&1 & &\\
  }}
\end{equation}

\begin{remark}
Recall the sign factor $\epsilon_d$ which
was defined in \eqref{epsilond} and notice that $\L\cap
\O(S,\cB_e)=\{a\in \L|\cN_e(a)=+\id_S\}\subset \cL$. Also recall that
$\cN_e(u)=+\id_S$ in the complex case. Relation \eqref{Ne} gives: 
\be
N_e|_{\Spin(V,h)}=N|_{\Spin(V,h)}~~,~~N_e|_{\Pin_-(V,h)}=\epsilon_d N|_{\Pin_-(V,h)}~~,
\ee
which implies:
\be
\ker(N_e:\Pin(V,h)\rightarrow \R^\times)=\Spin^+(V,h)\sqcup \Pin^{\epsilon_d}_-(V,h)~~.
\ee
When $pq\neq 0$, these
observations together with Proposition \ref{prop:cNe} show that the
connected components of the reduced Lipschitz group $\cL$ are as follows:
\begin{enumerate}[1.]
\itemsep 0.0em
\item In the normal simple case, $\cL\simeq \Pin(V,h)$ has four
  connected components, which are distinguished by the $\Z_2$ grading
  inherited from $\Cl(V,h)$ and by the value of the canonical
  Lipschitz norm $\cN_e$. We have $\cL\cap
  \O(S,\cB_e)=\gamma(\Spin^+(V,h)\sqcup
  \Pin^{\epsilon_d}_-(V,h))\simeq_{\Gp}\Spin^+(V,h)\sqcup
  \Pin^{\epsilon_d}_-(V,h) $, which has two connected components.
\item In the complex case, $\cL\simeq \Spin_\alpha^o(V,h)$ (where
  $\alpha\eqdef \alpha_{p,q}$) has four connected components, which
  are distinguished by the volume grading
  $\cL=\cL^0\sqcup \cL^1$ (where $\cL^0\simeq_{\Gp}\Spin^c(V,h)$ and
  $\cL^1=u\cL^0$) and by the value of $\cN_e$. We have $\cL\cap
  \O(S,\cB_e)= \gamma(\Spin^+(V,h)\cdot \U(1))\sqcup
  \gamma(\Spin^+(V,h)\cdot \U(1))u\simeq_{\Gp} \Spin^+(V,h)\cdot
  \O_2(\alpha)$, which has two connected components.
\item In the quaternionic simple case, $\cL\simeq \Pin^q(V,h)$ has
  four connected components, which are distinguished by the $\Z_2$
  grading inherited from $\Cl(V,h)$ and by the value of $\cN_e$. We
  have $\cL\cap \O(S,\cB_e)= \gamma(\Spin^+(V,h)\sqcup
  \Pin_-^{\epsilon_d}(V,h))\cdot \Sp(1)\simeq [\Spin^+(V,h)\sqcup
    \Pin_-^{\epsilon_d}(V,h)]\cdot \Sp(1)$, which has two connected
  components.
\item In the normal non-simple case, $\cL\simeq \Spin(V,h)$ has two
  connected components, which are distinguished by the value of
  $\cN_e$. We have $\cL\cap \O(S,\cB_e)= \gamma(\Spin^+(V,h))$, which
  is connected.
\item In the quaternionic non-simple case, $\cL\simeq \Spin^q(V,h)$
  has two connected components, which are distinguished by the value
  of $\cN_e$. We have $\cL\cap \O(S,\cB_e)=\gamma(\Spin^+(V,h))\cdot
  \Sp(1)\simeq \Spin^+(V,h)\cdot \Sp(1)$, which is connected.
\end{enumerate}
\end{remark} 

\subsection{The anticommutant representation of $\L$ in the simple case}

Assume that we are in the simple case. Recall that
$A\subset \End_\R(S)$ denotes the anticommutant subspace of
$\gamma:\Cl(V,h)\rightarrow \End_\R(S)$, which is a left $\S$-module
of rank one (see Proposition \eqref{Alist}). Also recall the
anticommutant representation $\Ad_A:\L\rightarrow \Aut_\S^\tw(A,\fp)$
of the Lipschitz group (introduced in Subsection
\ref{subsec:anticommutantrep}), where $\Aut_\S^\tw(A,\fp)$ denotes the
group of those twisted automorphisms of the $\S$-module $A$ which are
twisted-orthogonal with respect to the Schur pairing $\fp$ (see
Appendix \ref{app:twisted}).

\begin{prop}
\label{prop:anticommutant}
The following statements hold in the simple case:
\begin{enumerate}[1.]
\itemsep 0.0em
\item In the normal simple case, we have
  $\Aut_\S^\tw(A)=\Aut_\R(A)=\R^\times\id_A\simeq \R^\times $ and
  $\Aut_\S^\tw(A,\fp)=\{-\id_A,\id_A\}\simeq \Z_2$.
\item In the complex case, we have $\Aut_\S^\tw(A)\simeq
  \C^\times\rtimes_\Phi \Z_2\simeq \GL(2,\R)$ and
  $\Aut_\S^\tw(A,\fp)\simeq \U(1)\rtimes_{\Phi}\Z_2\simeq \O(2,\R)$,
  where $\Phi:\Z_2\rightarrow \Aut_{\Gp}(\C^\times)$ is the group
  morphism given by $\Phi({\hat 0})(z)=z$ and $\Phi({\hat
    1})(z)=\bar{z}$ ($z\in \C^\times$).
\item In the quaternionic simple case, we have $\Aut_\S^\tw(A)\simeq
  (\H^\times)^\op \rtimes_\Phi \SO(3,\R)$, where $\Phi:
  \SO(3,\R)\rightarrow \Aut_{\Gp}(\H^\times)$ is the group morphism
  given by $\Phi([q])=\Ad(q)|_{\H^\times}$ for all $[q]\in
  \U(\H)/\{-1,1\}=\SO(3,\R)$. We also have $\Aut_\S^\tw(A,\fp)\simeq
  \{-\id_A,\id_A\} \times \Aut_\Alg(\S) \simeq \Z_2\times \SO(3,\R)$.
\end{enumerate}
\end{prop}

\begin{proof}

\

\begin{enumerate}[1.]
\itemsep 0.0em
\item Normal simple case. In this case, we have $\S\simeq \R$ and
  $A=\R\omega$. Since $\Aut_\Alg(\R)=\{\id_\R\}$, Proposition
  \ref{prop:RankOneTw} of Appendix \ref{app:twisted} gives
  $\Aut_\S^\tw (A)=\Aut_\S(A)=\R^\times \id_A\simeq \R^\times\simeq
  \GL(1,\R)$. Proposition \ref{prop:SchurPairing} shows that the Schur
  pairing coincides up to sign with the usual scalar product on
  $\R$. Since $\Aut_\S^\tw(A)=\Aut_\S(A)$, the $\fp$-orthogonality
  condition \eqref{OrtRho} gives $\Aut_\S^\tw(A,\fp)\simeq
  \O(1,\R)\simeq \Z_2$.
\item In the complex case, we have $\Aut_\Alg(\S)=\{\id_\S , c
  \}\simeq \Z_2$ and the element $u=D$ of Proposition \ref{Alist} is a
  basis of $A$ over $\S\simeq \C$. Hence the splitting morphism
  $G_u:\Aut_\Alg(\S)\rightarrow \Aut_\S^\tw(A)$ in the proof of
  Proposition \ref{prop:RankOneTw} takes $c$ into the twisted morphism
  $\varphi \in \Aut_\S^\tw(A)$ given by $\varphi (sD)=c(s)D$ for all
  $s\in \S$. Since $c$ corresponds to complex conjugation, Proposition
  \ref{prop:RankOneTw} gives $\Aut_\S^\tw(A)\simeq
  \C^\times\rtimes_\Phi\Z_2\simeq \O(2,\R)$ if we identify $\S$ with
  $\C$ as in Proposition \ref{Schur}. With this identification,
  Proposition \ref{prop:SchurPairing} shows that the Schur pairing
  coincides (up to sign) with the Euclidean scalar product on
  $\C\equiv \R^2$, which is $\R$-valued. Hence the Schur pairing
  satisfies $c(\fp(x_1,x_2))=\fp(x_1,x_2)$ and condition
  \eqref{OrtRho} gives $\Aut_\S^\tw(A,\fp)\simeq\U(1)\rtimes_\Phi
  \Z_2\simeq \O(2,\R)$.
\item In the quaternionic simple case, we have $\Aut_\Alg(\S)\simeq
  \U(\S)/\{-\id_S,\id_S\}\simeq \SO(3,\R)$ and the element $u=\omega$
  generates $A$ over $\S$. The splitting morphism
  $G_u:\Aut_\Alg(\S)\rightarrow \Aut_\S^\tw(A)$ in the proof of
  Proposition \ref{prop:RankOneTw} takes $[q]\in
  \U(\S)/\{-\id_S,\id_S\}$ into the twisted morphism $\varphi_q\in
  \Aut_\S^\tw(A)$ given by $\varphi_q(s \omega)=\Ad(q)(s)\omega$ for
  all $s\in \S$. Identifying $\S$ with $\H$ as in Proposition
  \ref{Schur}, Proposition \ref{prop:RankOneTw} gives
  $\Aut_\S^\tw(A)\simeq (\H^\times)^\op \rtimes_{\Res} \Aut_\Alg(\H)\simeq
  (\H^\times)^\op \rtimes_\Phi \SO(3,\R)$. Any twisted automorphism
  $(\varphi_0,\varphi)$ of $A=\S \omega$ satisfies
  $\varphi(\omega)=\sigma \omega$ for some $\sigma\in \S$ and has
  the form:
\beqa
&& \varphi_0=\Ad_s(s_0)\\
&&\varphi(s\omega)=\varphi_0(s) \sigma \omega
\eeqa
for some $s_0\in \U(\S)$. Since $s$ and $\omega$ commute, it is easy
to see that $(\varphi_0,\varphi)$ satisfies condition \eqref{OrtRho}
with the Schur pairing given in Proposition \ref{prop:SchurPairing}
iff:
\ben
\label{sqrel}
(\varphi_0(s)\sigma)^2=\varphi_0(s^2)~~\forall~~s\in \S~~.
\een
For $s=\id_S$, this gives $\sigma^2=\id_S$, which amounts to
$\sigma\in \{-\id_S,\id_S\}$ (because $\S$ is a division algebra);
both of these solutions satisfy \eqref{sqrel}. Thus
$\varphi(s\omega)=\pm \Ad(s_0)(s)\omega$ and $\Aut_\S^\tw(A,\fp)\simeq
\{-\id_S,\id_S\} \times \Aut_\Alg(\S)$.
\end{enumerate}
\qed
\end{proof}

\begin{thm}
\label{thm:anticommutant}
The following statements hold in the simple case:
\begin{enumerate}[1.]
\item In the normal simple case, the anticommutant representation
  gives a short exact sequence:
\ben
\label{RealSimpleARep}
1\longrightarrow \L^0\hookrightarrow \L\stackrel{\Ad_A}{\longrightarrow}\{-\id_A,\id_A\}\simeq \Z_2\longrightarrow 1~~,
\een
where $\L^0=\gamma(\G_+(V,h))\simeq \G_+(V,h)$  and $\L^1=\gamma(\G_-(V,h))\simeq \G_-(V,h)$. We have:
\ben
\label{AdARealSimple}
\Ad_A(a)=\twopartdef{+\id_A}{a\in \L^0=\gamma(\G_+(V,h))}{-\id_A}{a\in \L^1=\gamma(\G_-(V,h))}~~.
\een
\item
In the complex case, the anticommutant representation of $\L$ gives a short exact sequence: 
\ben
\label{ComplexARep}
1\longrightarrow \gamma(\G_+(V,h))\hookrightarrow \L\stackrel{\Ad_A}{\longrightarrow}\Aut_\S^\tw(A,\fp)\simeq \O(2,\R)\longrightarrow 1~~
\een
which restricts to a short exact sequence:
\ben
\label{ComplexARep0}
1\longrightarrow \gamma(\G_+(V,h))\hookrightarrow \L^0=\gamma(\G^e(V,h))\stackrel{\Ad_A}{\longrightarrow}\Aut_\S(A,\fp)\simeq \SO(2,\R)\longrightarrow 1~~
\een
and for all $a\in \G_+(V,h)$ and all $\theta \in \R$ we have: 
\ben
\label{AdAComplex}
\Ad_A(\gamma(a e^{\theta \nu}))(u)=~~\Ad_A(\gamma(a e^{\theta \nu})u)(u)=e^{2\theta J}u~~.
\een
\item In the quaternionic simple case, we have $\L=\gamma(\G(V,h))\U(\S)$ with $\L^0=\U(\S)
  \gamma(\G_+(V,h))\U(\S)$ and $\L^1=\gamma(\G_-(V,h))\U(\S)$. The
  anticommutant representation of $\L$ gives a short exact sequence:
\ben
\label{SimpleQuatARep}
1\longrightarrow \gamma(\G_+(V,h))\hookrightarrow \L\stackrel{\Ad_A}{\longrightarrow}\Aut_\S^\tw(A,\fp)\simeq \Z_2\times \SO(3,\R)\longrightarrow 1~~
\een
which restricts to a short exact sequence: 
\ben
\label{AdAQuat0}
1\longrightarrow \gamma(\G_+(V,h))\hookrightarrow \L^0 \longrightarrow \SO(3,\R)\longrightarrow 1~~
\een
and for all $s\in \S$ we have:
\ben
\label{AdAQuat}
\Ad_A(a)(s\omega)=\twopartdef{+\Ad_s(a)(s)\omega}{a\in \L^0}{-\Ad_s(a)(s)\omega}{a\in \L^1}~~.
\een
\end{enumerate}
\end{thm}

\

\begin{proof}

\

\begin{enumerate}[1.]
\itemsep 0.0em
\item The normal simple case. We have $A=\R\omega$ and
  $\Aut_\S^\tw(A,\fp)=\{\id_A,-\id_A\}$. Relation
  \eqref{AdARealSimple} follows from the fact that $a\in \L$ commutes
  with $\omega$ iff $a\in \L^0$ and anticommutes with $\omega$ iff
  $a\in \L^1$. Thus $\ker\Ad_A=\L^0$. The map $\gamma$ induces an
  isomorphism from $\G(V,h)$ to $\L$ and (since $d$ is even) we have
  $\L^0=\gamma(\G_+(V,h))$ and $\L^1=\gamma(\G_-(V,h))$. The set
  $\L^1=\gamma(\G_-(V,h))$ is non-empty since $V\neq 0$; this shows
  that $\Ad_A$ is surjective.
\item The complex case. In this case, we have $A=\S u$ for $u$ as in
  Proposition \ref{Alist} and $\Aut_\S^\tw(A,\fp)\simeq \O(2,\R)$.
  For $a\in \L^0=\gamma(\G^e(V,h))$, we have $\Ad_s(a)=\id_\S$, hence
  $\Ad(a)$ is $\S$-linear and $\Ad_A(a)$ corresponds to an element of
  $\SO(2,\R)$. For $a\in \L^1=u\L^0$, we have $\Ad_s(a)=c$ (see
  Theorem \ref{thm:PinSchur}), hence $\Ad(a)$ is $\S$-antilinear and
  $\Ad_A(a)$ corresponds to an element of $\O_-(2,\R)$. Thus
  $\Ad_A(a)=\id_A$ iff $a=\gamma(x)\in \L^0$ for some $x\in \G^e(V,h)$
  which satisfies $\Ad(u)(\gamma(x))=\gamma(x)$. Since $\gamma$ is
  injective and $\Ad(u)\circ \gamma=\gamma\circ \pi$, this amounts to
  $\pi(x)=x$ i.e. $x\in \G_+(V,h)$. This shows that $\ker
  \Ad_A=\gamma(\G_+(V,h))$. We have $Z(V,h)=\R\oplus
  \R\nu\simeq_{\Alg} \C$ and $\nu\in Z(V,h)$. For any $x,y\in \R$, we
  have $e^{x+y \nu} \in Z(V,h)^\times\subset \G^e(V,h)$ (recall that
  $\nu^2=-1$ in the complex case).  Since $\gamma(\nu)=\omega=J$, we
  have $\gamma(Z(V,h)^\times)=\S\subset \L$ and $\gamma(e^{\theta
    \nu})=e^{\theta J}\in \gamma(\G^e(V,h))=\L^0$. This implies
  \eqref{AdAComplex} upon using the fact that $u$ and $J$
  anticommute. In particular, $\Ad_A(\U(\S))$ corresponds to
  $\SO(2,\R)\simeq \U(1)$ and $\Ad(\U(\S)u)$ corresponds to
  $\O_-(2,\R)$. Thus $\Ad_A$ is surjective, the sequence
  \eqref{ComplexARep} holds and it restricts to \eqref{ComplexARep0}.
\item The quaternionic simple case. In this case, we have $A=\S\omega$. Clearly $a\in \ker\Ad_A$ implies
  $\Ad_s(a)=\id_\S$ and hence $a\in \L_\gamma=\gamma(\G(V,h))$ by Proposition
  \ref{KAds}. Since $d$ is even, we have $\Ad(a)(\omega)=\epsilon
  \omega$ for $a\in \gamma(\G_\epsilon(V,h))$ (where $\epsilon\in \{-1,1\})$ and we conclude that
  $\ker\Ad_A=\gamma(\G_+(V,h))$. Since $\L=\U(\S) \gamma(\G(V,h))$,
  any $a\in \L$ can be written as $a=s_a a_0$ for some $s_a\in \U(\S)$
  and some $a_0\in \gamma(\G(V,h))$. Since $s_a$ commutes with $a_0$ and
  $\omega$, we have $\Ad(a)(\omega)=\Ad(a_0)(\omega)$ and hence:
\ben
\Ad_A(a)(\omega)=\epsilon \omega ~~\mathrm{for}~~a_0\in \gamma(\G_\epsilon(V,h))~~.
\een
This gives:
\be
\Ad_A(a)(s\omega)=\Ad_s(a)(s)\Ad_A(a)(\omega)=\epsilon\Ad_s(a)(s)\omega ~~\mathrm{for}~~a_0\in \gamma(\G_\epsilon(V,h))~~.
\ee
Thus \eqref{AdAQuat} holds. Since $\Ad_s(a_0)=\id_\S$, we have $\Ad_s(a)=\Ad_s(s_0)\Ad_s(a_0)=\Ad_s(s_0)$, which together 
by \eqref{AdAQuat} implies that the sequences 
\eqref{SimpleQuatARep} and \eqref{AdAQuat0} are exact. 
\end{enumerate}
\qed
\end{proof}

\subsection{Relation between the elementary representations of $\cL$ and those of $\Lambda(V,h)$}

\begin{definition}
The {\em characteristic representation} $\mu_\cL$ of the reduced Lipschitz group $\cL$ is defined
as follows:
\begin{enumerate}[1.]
\itemsep 0.0em
\item In the normal simple or non-simple case, it is the Schur
  representation of $\cL$, i.e. the trivial one-dimensional
  representation $\mu_\cL\eqdef \Ad_\S:\cL\rightarrow 1$.
\item In the complex case, it is the anticommutant representation
  $\mu_\cL\eqdef \Ad_A:\cL\rightarrow \O(A,(~,~)_A)\simeq \O(2,\R)$.
\item In the quaternionic simple or non-simple case, $\mu_\cL:\cL\rightarrow
  \SO(\Im \S,(~,~)_\S|_{\Im \S})\simeq \SO(3,\R)$ is the Schur
  representation restricted to $\Im \S$:
\be
\mu_\cL(a)\eqdef \Ad_s(a)|_{\Im \S}~~\forall a\in \L~~.
\ee
\end{enumerate}
\end{definition}

\begin{definition}
The {\em basic representation} of $\cL$ is the
representation $\rho_\cL=\Ad_0\times \mu_\cL$.
\end{definition}

\noindent Notice that $\rho_\cL=\Ad_0$ in the normal (simple and
non-simple) cases. 

\

\noindent To state the next result precisely, we need certain identifications which we discuss in turn. 
\begin{enumerate}
\itemsep 0.0em
\item Recall that $\gamma|_V:V\stackrel{\sim}{\rightarrow}
  W=\gamma(V)$ is an invertible isometry between the quadratic spaces
  $(V,h)$ and $(W,g)$. This induces an isomorphism of groups
  $\Ad(\gamma|_V):\O(V,h)\stackrel{\sim}{\rightarrow} \O(W,g)$ given
  by:
\be
\Ad(\gamma|_V)(R)\eqdef (\gamma|_V)\circ R\circ (\gamma|_V)^{-1}\in \O(W,g)~~\forall R\in \O(V,g)
\ee
and we have: 
\ben
\label{AdAdCl}
\Ad_0\circ \gamma|_{\G^e(V,h)}=\Ad(\gamma|_V)\circ \Ad_0^e~~,
\een
where $\Ad_0^e:\G^e(V,h)\rightarrow \O(V,h)$ is the untwisted vector
representation of the extended Clifford group while
$\Ad_0:\L\rightarrow \O(W,g)$ is the vector representation of the
Lipschitz group.
\item In the complex case, any choice of orientation of $V$ and of an
  element $u\in A$ as in Proposition \ref{Alist} gives an $\R$-linear
  isomorphism $g_{\nu,u}:\R^2\stackrel{\sim}{\rightarrow} A$ given by:
\be
g_{\nu,u}(x,y)=(x-\alpha_{p,q}yJ_\nu)u\in A~~\forall x, y\in \R~~,
\ee
where $J_\nu=\gamma(\nu)$ and $\nu$ is the Clifford volume element of $V$
with respect to the given orientation.  This isomorphism transports
both $(~,~)_A$ and the Schur pairing $\fp$ to scalar products on $\R^2$
which coincide up to sign with the canonical scalar product. It
follows that the unital isomorphism of $\R$-algebras
$\Ad(g_{\nu,u}):\End_\R(\R^2)\stackrel{\sim}{\rightarrow}\End_\R(A)$
restricts to an isomorphism of groups
$\Ad(g_{\nu,u}):\O(2,\R)\stackrel{\sim}{\rightarrow} \O(A,(~,~)_A)$.
\item In the quaternionic (simple and non-simple) cases, any unital
  isomorphism of $\R$-algebras $f:\H\stackrel{\sim}{\rightarrow}\S$
  induces a unital isomorphism of $\R$-algebras
  $\Ad(f):\End_\R(\H)\rightarrow \End_\R(\S)$ given by:
\be
\Ad(f)(g)\eqdef f\circ g\circ f^{-1}~~,~~\forall g\in \End_\R(\H)~~.
\ee
Consider the morphism of groups $\Ad_{\S}:\U(\S)\rightarrow
\Aut_\Alg(\S)$ defined through:
\be
\Ad_\S(s)(s')\eqdef ss's^{-1}~~\forall s\in \U(\S)~~\forall s'\in \S~~.
\ee
Then: 
\ben
\label{AdSAdH}
\Ad_{\S}\circ f|_{\U(\H)}=\Ad(f)\circ \Ad_\H~~.
\een
When viewed as a linear representation of $\U(\S)$ over $\R$, $\Ad_\S$
decomposes as a direct sum of the trivial representation supported on
the subspace $\R\id_\S$ and the representation
$\Ad_\S^\bullet:\U(\S)\rightarrow \SO(\Im\S,(~,~)_\S|_{\Im \S})$ given
by $\Ad_\S^\bullet(s)\eqdef \Ad_\S(s)|_{\Im \S}$ for all $s\in
\U(\S)$. We have $f(\Im \H)=\Im \S$ and $f|_{\Im \H}$ is a linear
isomorphism from $\Im \H$ to $\Im \S$, which induces an isomorphism of
groups
$\Ad(f|_{\Im\H}):\SO(\Im\H)\stackrel{\sim}{\rightarrow}\SO(\Im\S,(~,~)_\S|_{\Im
  \S})$.  Relation \eqref{AdSAdH} implies:
\ben
\label{AdSAdbullet}
\Ad_\S^\bullet\circ (f|_{\U(\H)})=\Ad(f|_{\Im \H})\circ \Ad_\bullet~~,
\een
where $\Ad_\bullet:\U(\H)\rightarrow \SO(\Im\H)=\SO(3,\R)$ is the
adjoint representation of $\U(\H)=\Sp(1)$ (see Section
\ref{sec:enlargedspinor}).
\end{enumerate}

\noindent Recall that $\lambda,\mu$ and $\rho$ are the vector, characteristic
and basic representations of the extended spinor group
$\Lambda(V,h)$. The following result shows that the elementary
representations $\Ad_0,\mu_\cL$ and $\rho_\cL$ of the reduced
Lipschitz group $\cL$ of a pin representation $\gamma$ agree with the
elementary representations $\lambda, \mu$ and $\rho$ of the enlarged
spinorial group $\Lambda(V,h)$ studied in Section
\ref{sec:enlargedspinor}.

\begin{thm}
\label{thm:basicreps}
Let $\varphi:\Lambda(V,h)\stackrel{\sim}{\rightarrow} \cL$ be any of
the admissible isomorphisms given in Theorem \ref{thm:cLpres}. Then:
\begin{enumerate}
\item In the normal (simple and non-simple) cases, we have
  $\varphi=\gamma|_{\Lambda(V,h)}$ and the following relations hold:
\beqan
\label{basicrepsnormal}
\Ad(\gamma|_V)\circ \lambda &=& \Ad_0 \circ \varphi~~,\nn\\ 
\mu &=& \mu_\cL\circ \varphi~~,\\
\Ad(\gamma|_V)\circ \rho &=& \rho_\cL \circ \varphi~~.\nn
\eeqan
\item In the complex case, $\varphi=\varphi_{\nu,u}$ is determined by
  an orientation of $(V,h)$ and by a choice of element $u\in A$ as in
  Proposition \ref{Alist} and the following relations hold:
\beqan
\label{basicrepscomplex}
\Ad(\gamma|_V)\circ \lambda &=& \Ad_0\circ \varphi_{\nu,u}~~,\nn\\ 
\Ad(g_{\nu,u}) \circ \mu &=& \mu_\cL\circ \varphi_{\nu,u}~~,\\
\left[\Ad(\gamma|_V)\times \Ad(g_{\nu,u})\right]\circ \rho &=& \rho_\cL \circ \varphi_{\nu,u}~~.\nn
\eeqan
\item In the quaternionic (simple and non-simple) cases,
  $\varphi=\varphi_f$ is determined by $\gamma$ and by a unital
  isomorphism of normed $\R$-algebras
  $f:\H\stackrel{\sim}{\rightarrow}\S$ and the following relations
  hold:
\beqan
\label{basicrepsquaternionic}
\Ad(\gamma|_V)\circ \lambda &=&\Ad_0 \circ \varphi_f~~,\nn\\ 
\Ad(f|_{\Im\H}) \circ \mu &=& \mu_\cL\circ \varphi_f~~,\\
\left[\Ad(\gamma|_V)\times \Ad(f|_{\Im\H})\right]\circ \rho &=& \rho_\cL \circ \varphi_f~~.\nn
\eeqan
\end{enumerate}
\end{thm}

\noindent This result gives the intrinsic meaning of the elementary
representations of the enlarged spinor groups considered in Section
\ref{sec:enlargedspinor}. It also shows that one can use the model
$\Lambda(V,h)$ of the reduced Lipschitz group $\cL$ (which is homotopy
equivalent with the full Lipschitz group $\L$) when developing the
theory of Lipschitz structures for the Lipschitz groups of pin
representations.

\begin{proof}

\

\begin{enumerate}
\item Normal simple case. In this case, we have
  $\varphi=\gamma|_{\Pin(V,h)}:\Pin(V,h)\stackrel{\sim}{\rightarrow}\cL$. Since
  $\Pin(V,h)$ is a subgroup of $\G^e(V,h)$, relation \eqref{AdAdCl}
  gives $\Ad_0\circ
  \varphi=\Ad(\gamma|_V)(\Ad_0^\Cl|_{\Pin(V,h)})=\Ad(\gamma|_V)\circ
  \lambda$, which is the first equation in
  \eqref{basicrepsnormal}. The second relation holds since $\mu$ and
  $\mu_\cL$ are the trivial representations, while the last relation
  in \eqref{basicrepsnormal} holds because $\rho=\lambda$ and
  $\rho_\cL=\Ad_0|_\cL$.
\item Normal non-simple case. In this case, we have
  $\varphi=\gamma|_{\Spin(V,h)}:\Spin(V,h)\stackrel{\sim}{\rightarrow}\cL$,
  with $\mu_\cL$ and $\mu$ trivial and
  $\rho=\lambda=\Ad_0^\Cl|_{\Spin(V,h)}$,
  $\rho_\cL=\Ad_0|_\cL$. Relations \eqref{basicrepsnormal} hold by an
  argument similar to that for the normal simple case.
\item Complex case. Let $\alpha:=\alpha_{p,q}$ and
  $\psi:=\psi_{\alpha}$ be the isomorphism of Proposition
  \ref{prop:Pin2O2}. We have $\Lambda(V,h)=\Spin^o(V,h)$ and
  $\varphi([a,\psi(e^{\bi\theta},\kappa)])=\gamma(a)e^{\theta
    J_\nu}u^\kappa$ for all $a\in \Spin(V,h)$, with $(e^{\bi
    \theta},\kappa)\in \O_2(\alpha)$, where $\theta\in \R$, $\kappa\in
  \Z_2$. Thus:
\be
(\Ad_0\circ \varphi)([a,\psi(e^{\bi\theta},\kappa)])=(\Ad_0\circ \gamma)(a)\Ad_0(e^{\theta J_\nu})\Ad_0(u^\kappa)=(-1)^\kappa (\Ad_0\circ \gamma)(a)~~,
\ee
where we used the fact that $\Ad_0(e^{\theta J_\nu})=\id_W$ (since
$e^{\theta J_\nu}\in \S^\times$) and $\Ad_0(u)=-\id_W$. This can
also be written as:
\be
(\Ad_0\circ \varphi)([a,\psi(e^{\bi\theta},\kappa)])=\eta(e^{\bi\theta},\kappa)(\Ad_0\circ \gamma)(a)=\eta(e^{\bi \theta},\kappa) \Ad(\gamma|_V)\circ \Ad_0^\Cl(a)~~,
\ee
where $\eta:=\eta_{\alpha}:\O_2(\alpha)\rightarrow \mG_2$
is the abstract determinant introduced in Subsection
\ref{sec:spino} and in the last equality we used relations
\eqref{AdAdCl}. Since $\eta=\det\circ \Ad_0^{(2)}\circ \psi$ (see
Proposition \ref{prop:Pin2O2}), the relation above becomes:
\beqa
(\Ad_0\circ \varphi)([a,\psi(e^{\bi \theta},\kappa)])&=&\det(\Ad_0^{(2)}(\psi(e^{\bi \theta},\kappa)))\Ad(\gamma|_V)\circ \Ad_0^\Cl(a)=\nn\\
&=&\Ad(\gamma|_V)\circ \lambda([a,\psi(e^{\bi\theta},\kappa)])~~,
\eeqa
where in the last line we used Definition
\ref{def:SpinoBasicReps}. This shows that the first relation of
\eqref{basicrepscomplex} holds. Since $\mu_\cL=\Ad_A$, Theorem
\ref{thm:anticommutant} gives:
\be
(\mu_\cL\circ\varphi)([a,\psi(e^{\bi \theta},\kappa)])(u)=\Ad_A(\gamma(a)e^{\theta J_\nu}u^\kappa)(u)=e^{2\theta J_\nu}u~~,
\ee
Setting $g:=g_{\nu,u}$, we have $g(1)=u$ and $g(e^{-2\bi \alpha \theta})=e^{2\theta J}u$ and the relation above gives:
\beqa
(\mu_\cL\circ \varphi)([a,\psi(e^{\bi \theta},\kappa)])&=&\Ad(g)(R(-2\alpha\theta)C_0^\kappa)=\nn\\
&=& (\Ad(g)\circ \Phi_0^{(-\alpha)}\circ \sigma_{\alpha})(e^{\bi \theta},\kappa)=(\Ad(g)\circ \Ad_0^{(2)})(\psi(e^{\bi \theta},\kappa))~~,
\eeqa
where we identified $\R^2=\C$ and used Proposition \ref{prop:Pin2O2}
in the last line. Here, $\sigma_{\alpha}$ is the squaring
representation of $\O_2(\alpha)$ discussed in Subsection
\ref{sec:spino}. This shows that the second relation in
\eqref{basicrepscomplex} holds. The last relation in \eqref{basicrepscomplex} now
follows because $\rho_\cL=\lambda_\cL\times \mu_\cL$ and
$\rho=\lambda\times \mu$.
 \item Quaternionic simple case. In this case, we have
   $\Lambda(V,h)=\Pin(V,h)\cdot \Sp(1)$ and
   $\varphi([a,q])=\gamma(a)f(q)\in \cL$ for all $a\in \Pin(V,h)$ and
   $q\in \Sp(1)=\U(\H)$, where $f:\H\stackrel{\sim}{\rightarrow}\S$ is
   any unital isomorphism of normed $\R$-algebras. Thus:
\be
(\Ad_0\circ \varphi)([a,q])=\Ad_0(\gamma(a))\Ad_0(f(q))=\Ad_0(\gamma(a))~~,
\ee
where we used the fact that $\Ad_0(s)=\id_W$ for all $s\in \U(\S)$
(since $W=\gamma(V)\subset C$).  Using relation \eqref{AdAdCl}, the
equation above gives $\Ad_0\circ \varphi=\Ad(\gamma|_V)\circ \lambda$
(where $\lambda$ is the vector representation of $\Pin^q(V,h)$
introduced in Section \ref{sec:enlargedspinor}), showing that the
first relation in \eqref{basicrepsquaternionic} holds. For the Schur
representation of $\cL$, we have:
\be
(\Ad_s\circ \varphi)([a,q])=\Ad_s(\gamma(a))\Ad(f(q))|_\S=\Ad_\S(f(q))~~,
\ee
where we used the fact that $\Ad_s(\gamma(a))=\id_\S$. Restricting
this to the subspace $\Im \S$ and using relation \eqref{AdSAdbullet}
gives the second relation in \eqref{basicrepsquaternionic}. The third
relation also holds, because $\rho_\cL=\lambda_\cL\times \mu_\cL$ and
$\rho=\lambda\times \mu$.
\item Quaternionic non-simple case. The argument is almost identical
  to that for the quaternionic simple case.
\end{enumerate}
\qed
\end{proof}

\section{Real pinor bundles and real Lipschitz structures}
\label{sec:realpinors}

In this section, we discuss real Lipschitz structures and bundles of
Clifford modules over a pseudo-Riemannian manifold $(M,g)$,
establishing an equivalence between the corresponding groupoids. This
shows that the classification of the former agrees with that of the
latter. For the case of bundles of irreducible Clifford modules, the
relevant Lipschitz structures are called {\em elementary} and can be
described using the enlarged spinor groups of Section
\ref{sec:enlargedspinor}. Combining this with the results of Section
\ref{sec:structures} will allow us, in the next section, to extract
the topological obstructions to existence of elementary Lipschitz
structures and hence to existence of bundles of irreducible Clifford
modules in each dimension and signature.

Let $(M,g)$ be a connected second countable smooth pseudo-Riemannian
manifold of dimension greater than zero and let $g^\ast$ denote the
metric induced on $T^\ast M$. For reasons having to do with various
applications (to be discussed in other papers), we choose to work with
the Clifford bundle of the pseudo-Euclidean vector bundle $(T^\ast M,
g^\ast)$, rather than with the Clifford bundle of $(TM,g)$ (as is more
customary).  Since the pseudo-Euclidean vector bundles $(TM,g)$ and
$(T^\ast M, g^\ast)$ are isometric through the musical isomorphism,
this is of course equivalent with the more traditional approach.

\subsection{Real pinor bundles}

\begin{definition}
A {\em real pinor bundle} is a smooth bundle $S\neq 0$ of finite-dimensional
modules over the Clifford bundle $\Cl(T^\ast M,g^\ast)$, i.e. a pair
$(S,\gamma)$ where $S$ is a real vector bundle over $M$ and
$\gamma:\Cl(T^\ast M,g^\ast)\rightarrow \End(S)$ is a smooth morphism
of vector bundles (called the {\em structure morphism}) such that the
fiber map $\gamma_m:\Cl(T_m^\ast,g_m^\ast)\rightarrow \End_\R(S_m)$ is
a unital morphism of associative $\R$-algebras for each $m \in M$.
\end{definition}

\noindent Hence any fiber $S_m$ is a real representation of the
Clifford algebra $\Cl(T_m^\ast M, g_m^\ast)\simeq \Cl(T_mM,g_m)$.  For
any based smooth morphism $f\in \Hom(S, S')$ of real vector bundles,
we let $L_f:\End(S)\rightarrow \Hom(S,S')$ and
$R_f:\End(S')\rightarrow \Hom(S,S')$ denote the vector bundle
morphisms defined through:
\begin{equation*}
L_f(\varphi)\eqdef f\circ \varphi\, , \quad R_f(\varphi')\eqdef \varphi'\circ f\, , \quad \varphi\in \End_\R(S)\, , \quad \varphi'\in \End_\R(S^{\prime})\, .
\end{equation*}

\begin{definition}
A {\em based morphism of real pinor bundles} $f:(S,\gamma)\rightarrow
 (S',\gamma')$ is a smooth based morphism $f:S\rightarrow S'$ of real
 vector bundles such that: 
\be
L_f\circ \gamma=R_f\circ \gamma'~~,
\ee
i.e. such that the fiber $f_m:S_m\rightarrow S'_m$ at any point $m\in
M$ is a {\em based} morphism of Clifford representations from
$\gamma_m:\Cl(T_m^\ast M,g_m^\ast)\rightarrow \End_\R(S_m)$ to
$\gamma'_m:\Cl(T_m^\ast M,g_m^\ast)\rightarrow \End_\R(S'_m)$.
\end{definition}

\noindent Since $M$ is connected, all quadratic spaces $(T_m^\ast,
g_m^\ast)$ are mutually isomorphic and hence isomorphic to some model
quadratic space $(V,h)$. Similarly, all fibers of $S$ are isomorphic
as $\R$-vector spaces and hence isomorphic with some model vector
space $S_0$. Using a common trivializing cover of $TM$ and $S$, this
implies that the real Clifford representations $\gamma_m:T_m
M\rightarrow \End_\R(S_{m})$ ($m\in M$) are mutually isomorphic in the
category $\ClRep$ and hence isomorphic with some model representation
$\gamma_0:\Cl(V,h)\rightarrow \End_\R(S_0)$. The isomorphism class of
$\gamma_0$ in $\ClRep$ is invariant under isomorphism of real pinor
bundles.

\begin{definition}
The {\em type} $\eta$ of a real pinor bundle $(S,\gamma)$ is the isomorphism
class of its fiberwise Clifford representation $\gamma_m:(T_m^\ast M,
g_m^\ast)\rightarrow \Aut_\R(S_m)$ in the category $\ClRep$.
\end{definition}

\begin{definition}
A real pinor bundle $(S,\gamma)$ is called {\em weakly faithful} if
$\gamma|_{T^\ast M}$ is a monomorphism of vector bundles from $T^\ast
M$ to $\End(S)$, i.e. if the map $\gamma_m|_{T_m^\ast M}:T_m^\ast
M\rightarrow \End_\R(S_m)$ is injective (and thus a weakly faithful
Clifford representation) for each $m\in M$.
\end{definition}

\noindent Let $\ClB(M,g)$ denote the category of real pinor bundles
over $(M,g)$ and based pinor bundle morphisms and $\ClB_w(M,g)$
denote the full sub-category whose objects are the weakly faithful
real pinor bundles. Clearly $(S,\gamma)$ is weakly faithful iff
its type is. If $\eta:\Cl(V,h)\rightarrow \End_\R(S)$
is a weakly faithful Clifford representation, we let
$\ClB_w^\eta(M,g)$ denote the full sub-category of $\ClB_w(M,g)$
consisting of all real pinor bundles of type equal to the isomorphism
class of $\eta$ and $\ClB_w^\eta(M,g)^\times$ denote the corresponding
unit groupoid.

\subsection{The pseudo-orthogonal coframe bundle}

Let $(p,q)$ denote the signature of $(M,g)$ and $d=p+q$ denote the
dimension of $M$. Let $(V,h)$ be a quadratic space isomorphic with
any (and hence all) fibers of the pseudo-Euclidean bundle $(T^\ast M, g^\ast)$.

\begin{definition}
The {\em pseudo-orthogonal coframe bundle $P_{\O(V,h)}(M,g)$ of $(M,g)$
  relative to $(V,h)$} is the principal bundle with structure group
$\O(V,h)=\Aut_{\Quad^\times}(V,h)$, total space:
\be
P_{\O(V,h)}(M,g)\eqdef \sqcup_{m\in M}\Hom_{\Quad^\times}((V,h),(T_m^\ast M,g_m^\ast))
\ee 
and right $\O(V,h)$-action given by right composition of $r \in
P_{\O(V,h)}(M,g)_m$ with elements
$R\in \O(V,h)$: 
\be
r R\eqdef r \circ R~~.
\ee
\end{definition}

\

\noindent Notice that the group $\O(T_m^\ast M,g_m^\ast)$ acts from
the left on each fiber $P_{\O(V,h)}(M,g)_m$ by left composition:
\be
R_m r\eqdef R_m \circ r~~~~\forall r\in \Hom_{\Quad^\times}((V,h),(T_m^\ast M,g_m^\ast))~~\mathrm{and}~~R_m\in \O(T_m^\ast M,g_m^\ast)\, .
\ee
The pseudo-orthogonal coframe bundle of $(M,g)$ relative to
$\R^{p,q}$ is denoted $P_{p,q}(M,g)$ and is called the {\em canonical
  pseudo-orthogonal coframe bundle} of $(M,g)$. Its fiber at $m\in M$ is the
set of all invertible isometries $f:\R^{p,q}\rightarrow (T_m^\ast
M, g_m^\ast)$, which can be identified with the set of all pseudo-orthogonal
frames of $(T_m^\ast M, g_m^\ast)$ through the map:
\be
f\rightarrow (f(\epsilon_1),\ldots, f(\epsilon_d))~~,
\ee
where $(\epsilon_1,\ldots,\epsilon_d)$ is the canonical basis of
$\R^{p+q}$. The pseudo-orthogonal coframe bundle of $(M,g)$ relative
to any model $(V,h)$ of the fiber of $(T^\ast M,g^\ast)$ is isomorphic
with the canonical pseudo-orthogonal coframe bundle.

\begin{remark}
Let $(V_0,h_0)$ be an isometric model of the tangent spaces $(T_mM,
g_m)$ of $(M,g)$.  The bundle $P_{\O(V_0,h_0)}(M,g)$ of
pseudo-orthogonal frames of $(M,g)$ relative to $(V_0,h_0)$ is the
principal $\O(V_0,h_0)$-bundle with fibers
$P_{\O(V_0,h_0)}(M,g)_m\eqdef \Hom_{\Quad^\times}((V_0,h_0),(T_mM,
g_m))$.  Taking $V_0=V^\ast=\Hom_\R(V,\R)$ and canonically identifying
$V_0^\ast=(V^\ast)^\ast$ with $V$, the musical ismorphism gives an
invertible isometry $\zeta: (V,h)\rightarrow (V_0,h_0)$ defined
through:
\be
\zeta(x)(y)\eqdef h(x,y)~~\forall x, y\in V_0~~.
\ee
This induces an isomorphism of groups
$\Ad(\zeta):\O(V,h)\stackrel{\sim}{\rightarrow} \O(V_0,h_0)$ which
identifies the tautological representation of $\O(V_0,h_0)$ with the
dual (a.k.a. contragradient) representation of $\O(V,h)$:
\be
\Ad(\zeta)(R) (\eta)=\eta\circ R^{-1}~~\forall \eta\in V_0=V^\ast~~\forall R\in \O(V,h)~~.
\ee 
This implies that $P_{\O(V_0,h_0)}(M,g)$ is naturally isomorphic with  $P_{\O(V,h)}(M,g)$ and that $T M$ is
associated to $P_{\O(V,h)}(M,h)$ through the contragradient
representation of $\O(V,h)$.
\end{remark}

\subsection{Real Lipschitz structures}

Let $\varphi:H\rightarrow G$ be a morphism of groups and $P$ be a
principal $G$-bundle over $M$. We say that $P$ admits a {\em
  $\varphi$-reduction to $H$} if there exists a principal $H$-bundle
$Q$ over $M$ and a $\varphi$-equivariant bundle map $\tau:Q\rightarrow
P$, where $\varphi$-equivariance means that we have\footnote{The
  element $q\in Q$ should not be confused with the number appearing in
  the signature $(p,q)$ of $g$.}  $\tau(qh)=\tau(q)\varphi(h)$ for all
$q\in Q$ and $h\in H$. In this case, the pair $(Q,\tau)$ is called a
{\em $\varphi$-reduction of $P$ to $H$}. Notice that we do not require
that $\varphi$ be injective or surjective. Given two $\varphi$-reductions $(Q,\tau)$
and $(Q',\tau')$ of $P$ to $H$, an {\em isomorphism of
  $\varphi$-reductions} from $(Q,\tau)$ to $(Q',\tau')$ is a based
isomorphism of principal $H$-bundles $f:Q\rightarrow Q'$ such that
$\tau'\circ f=\tau$. Notice that $\varphi$-reductions of $P$ to
$H$ form a groupoid whose morphisms are given by isomorphisms of
reductions.

\begin{remark}
\label{rem:twostepreduction}
Suppose that $\varphi$ is not surjective and let $I(\varphi)\eqdef
\varphi(H)$, which is a normal subgroup of $G$.  Let
$j:I(\varphi)\rightarrow G$ be the inclusion morphism and
$\varphi_0:H\rightarrow I(\varphi)$ be the corestriction of $\varphi$.
Then $P$ admits a $\varphi$-reduction to $H$ iff $P$ admits a
$j$-reduction $(P',\tau')$ to $I(\varphi)$ and $(P',\tau')$ admits a
$\varphi_0$-reduction to $H$. Of course, a $j$-reduction of $P$ to $I(\varphi)$ 
is the same as an ordinary reduction of structure group. 
\end{remark}

\noindent The principal bundle $P$ admits a $\varphi$-reduction to $H$ iff
$[P]\in H^1(M,G)$ lies in the image of the induced map
$\varphi_\ast:H^1(M,H)\rightarrow H^1(M,G)$. In this case, we have
$\varphi_\ast([Q])=[P]$ for any $\varphi$-reduction $(Q,\tau)$. 
Let $\eta:\Cl(V,h)\rightarrow \End_\R(S_0)$ be a weakly faithful
Clifford representation, $\L(\eta)\eqdef \Aut_{\ClRep}(\eta)$ be its
Lipschitz group and $\Ad_0:\L(\eta)\rightarrow \O(V,h)$ be the vector
representation of $\L(\eta)$.

\begin{definition}
Let $P$ be a principal $\O(V,h)$-bundle over $M$. A {\em real
  Lipschitz structure on $P$ relative to $\eta\in \ClRep$} is an
$\Ad_0$-reduction $(Q,\tau)$ of $P$ to $\L(\eta)$. A {\em real Lipschitz
  structure on $(M,g)$ relative to $\eta$} is a real Lipschitz
structure on $P_{\O(V,h)}(M,g)$ relative to $\eta$.
\end{definition}

\noindent $\Ad_0$-equivariance for a real Lipschitz structure $(Q,
\tau)$ on $(M,g)$ means that the following relation is satisfied by the map
$\tau_m:Q_m\rightarrow
P_{\O(V,h)}(M,g)_m=\Hom_{\Quad^\times}((V,h),(T_m^\ast M, g_m^\ast))$
for all $m\in M$, $q\in Q_m$ and $a \in \L(\eta)$:
\ben
\label{taueq}
\tau_m(q a)=\tau_m(q)\circ \Ad_0(a)~~.
\een

\begin{definition}
Let $(Q,\tau)$ and $(Q',\tau')$ be two real Lipschitz structures on
$(M,g)$ relative to $\eta$. An {\em isomorphism of Lipschitz
  structures} from $(Q,\tau)$ to $(Q',\tau')$ is an isomorphism of
$\Ad_0$-reductions of $P_{\O(V,h)}(M,g)$ to $\L(\eta)$.
\end{definition}

\noindent Let $L_\eta(M,g)$ be the groupoid of real Lipschitz
structures of $(M,g)$ relative to $\eta$. 

\subsection{Relation between weakly faithful real pinor bundles and real Lipschitz structures}

Let $\eta:\Cl(V,h)\rightarrow \End_\R(S_0)$ be a weakly faithful real
Clifford representation, where $(V,h)$ is a model of the fiber of
$(T^\ast M, g^\ast)$ and let $\L\eqdef \L(\eta)$ be the Lipschitz
group of $\eta$. Let $[\eta]$ be the isomorphism class of $\eta$ in
the category $\ClRep$. Recall the surjective functor $F:\ClRep\rightarrow \Quad$ of
Subsection \ref{subsec:ClRep}.  This sends every object
$\gamma':\Cl(V',h')\rightarrow \End_\R(S'_0)$ of $\ClRep$ to the
quadratic space $(V',h')$. Given two objects
$\gamma_i:\Cl(V_i,h_i)\rightarrow \End_\R(S_{0i})$ ($i=1,2$) of
$\ClRep$, the functor $F$ sends any morphism $(f_0,f)\in
\Hom_{\ClRep}(\gamma_1,\gamma_2)$ (where $f\in \Hom_\R(S_{01},
S_{02})$ and $f_0\in \Hom_{\Quad}((V_1,h_1),(V_2,h_2))$) to the
morphism of quadratic spaces $F(f_0,f)=f_0$. When $\gamma_1$ and
$\gamma_2$ are weakly faithful and $f$ is bijective, Proposition
\ref{prop:AdRelation} shows that $f_0$ is also bijective and that it
is uniquely determined by $f$ (namely, we have
$\Ad(f)(\gamma_1(V_1))=\gamma_2(V_2)$ and
$f_0=(\gamma_2|_{V_2})^{-1}\circ \Ad(f)|_{\gamma_1(V_1)}\circ
(\gamma_1|_{V_1})$). Moreover, morphisms $(f_0,f)\in
\Hom_{\ClRep^\times_w}(\gamma_1,\gamma_2)$ of the groupoid
$\ClRep^\times_w$ can be identified with those linear isomorphisms
$f:S_{01}\rightarrow S_{02}$ which satisfy
$\Ad(f)(\gamma_1(V_1))=\gamma_2(V_2)$ (see \eqref{isomwf}).  We will
use this identification throughout this subsection.

\begin{pdef}
\label{propdef:functors}
There exists a correspondence between real pinor bundles and real Lipschitz structures given as follows.
\begin{enumerate}[A.]
\item Let $(S,\gamma)$ be a weakly faithful real pinor bundle of type
  $\eta$ on $(M,g)$. Let $Q \eqdef Q_\eta(S,\gamma)$ be the principal
  bundle with structure group $\L=\L(\eta)=\Aut_{\ClRep_w}(\eta)$, total space:
\begin{equation*}
Q\eqdef \sqcup_{m\in M}\Hom_{\ClRep^\times}(\eta,\gamma_m)\, ,
\end{equation*}
projection given by $\pi(q)=m$ for $q \in Q_m=
\Hom_{\ClRep^\times}(\eta,\gamma_m)$ and right $\L$-action given by $q
a\eqdef q\circ a$ for $a\in \L$. Let
$\tau\eqdef \tau_\eta(S,\gamma)\colon Q_\eta(S,\gamma)\rightarrow
P_{\O(V,h)}(M,g)$ be the map defined through:
\ben
\label{taudef}
\tau_m(q)=F(q)=q_0\in \Hom_{\Quad^\times}((V,h),(T_m^\ast M, g_m^\ast))=P_{\O(V,h)}(M,g)_m~~.
\een
Then $(Q,\tau)$ is a Lipschitz structure
on $(M,g)$ relative to $\eta$, called the {\em Lipschitz structure
  defined by $(S,\gamma)$}. A based isomorphism of weakly faithful
real pinor bundles $f:(S,\gamma)\rightarrow
(S',\gamma')$ of type $\eta$ induces an isomorphism
$Q_\eta(f):(Q_\eta(S,\gamma), \tau_\eta(S,\gamma))\rightarrow
(Q_\eta(S',\gamma'),\tau_\eta(S',\gamma'))$ of Lipschitz structures
relative to $\eta$, which is defined through (notice that $f_m\in
\Hom_{\ClRep^\times}(\gamma_m,\gamma'_m)$ and $(f_m)_0=F(f_m)=\id_{T_m^\ast M}$):
\ben
\label{Pfdef}
Q_\eta(f)(q)\eqdef (\id_{T_m^\ast M}, f_m)\circ q~~,~~\forall q\in Q_\eta(S,\gamma)_m=\Hom_{\ClRep^\times}(\eta,\gamma_m)~~.
\een
\item Let $(Q,\tau)$ be a Lipschitz structure on $(M,g)$ relative to
  $\eta$. Then the vector bundle $S:=S_\eta(Q,\tau)\eqdef
  Q\times_{\rho_\eta} S_0$ associated to $Q$ through the tautological
  representation $\rho_\eta:\L\rightarrow \Aut_\R(S_0)$ of $\L$
  becomes a weakly faithful real pinor bundle of type $\eta$ when
  endowed with the structure morphism
  $\gamma:=\gamma(Q,\tau):\Cl(T^\ast M, g^\ast)\rightarrow \End_{\R}(S)$
  given by:
\ben
\label{gammadef}
\gamma_m(x)([q,s])=[q,\eta(\Cl(\tau_m(q)^{-1})(x))(s)] ~~,~~\forall x\in \Cl(T_m^\ast M,g_m^\ast)\, ,
\een
for all $q\in Q_m$ and $s\in S_0$. The pair
$S_\eta(Q,\tau)=(S,\gamma)$ thus constructed is called the {\em real
  pinor bundle defined by the Lipschitz structure $(Q,\tau)$} (in
particular, it is weakly faithful). An isomorphism of Lipschitz
structures $f:(Q,\tau)\rightarrow (Q',\tau')$ relative to $\eta$
induces a based isomorphism of real pinor bundles
$S_\eta(f)=(S_\eta(Q,\tau), \gamma_\eta(Q,\tau))\rightarrow
(S_\eta(Q',\tau'),\gamma_\eta(Q',\tau'))$ given by:
\ben
\label{Sfdef}
S_\eta(f)_m([q,s])=[f_m(q),s]~~,~~\forall q\in Q_m~~\forall s\in S_0~~.
\een
\end{enumerate}
Furthermore, the correspondences defined above give functors
$Q_\eta:\ClB_w^\eta(M,g)^\times\rightarrow \L_\eta(M,g)$ and
$S_\eta:\L_\eta(M,g)\rightarrow \ClB_w^\eta(M,g)^\times$.
\end{pdef}

\noindent At point B. of the proposition, notice that $\tau_m(q)\in
\Hom_{\Quad^\times}((V,h),(T_m^\ast,g_m^\ast))$, is
an invertible isometry which induces the unital isomorphism of
$\R$-algebras
$\Cl(\tau_m(q)):\Cl(V,h)\stackrel{\sim}{\rightarrow}\Cl(T_m^\ast,
g_m^\ast)$, whose inverse appears in relation \eqref{gammadef}. Thus
$\eta\circ
[\Cl(\tau_m(q))]^{-1}=\eta\circ \Cl(\tau_m(q)^{-1}):\Cl(T_m^\ast,g_m^\ast)\rightarrow \End_\R(S_0)$ is
a representation of the Clifford algebra $\Cl(T_m^\ast,g_m^\ast)$ in
the vector space $S_0$. This representation is isomorphic with $\eta$ in the category $\ClRep$. 

\begin{remark}
Notice that the definition of $Q_\eta(S,\gamma)$ is similar to that of
the pseudo-orthogonal coframe bundle $P_{\O(V,h)}(M,g)$, where the
groupoid of quadratic spaces is replaced by the groupoid
$\ClRep_w^\times$ while the quadratic spaces $(V,h)$ and $(T_m^\ast M,
g_m^\ast)$ are replaced by the Clifford representations
$\eta:\Cl(V,h)\rightarrow \End_\R(S_0)$ and $\gamma_m:\Cl(T^\ast_m M,
g_m^\ast)\rightarrow \End_\R(S_m)$ (the former being a model of the
latter in the category $\ClRep_w$). In particular, the Lipschitz group
$\L=\L(\eta)$ is a model for the groups $\Aut_{\ClRep_w}(\gamma_m)$ of
automorphisms of the fibers of $S$, when the latter is considered as a
bundle of Clifford representations.  It is crucial that these
automorphisms are considered in the category $\ClRep$ (thus they need
not be based) and not in the ordinary category of representations of
unital algebras. This is because $\eta$ and $\gamma_m$ can only be
identified if one picks an invertible isometry between $(V,h)$ and
$(T_m^\ast M, g_m^\ast)$; this forces one to use the structure group
$\L=\Hom_{\ClRep^{\times}_{w}}(\eta)$ in order to make $Q$ into a principal bundle.
\end{remark}

\begin{proof}
The fact that \eqref{taudef} is $\Ad_0$-equivariant follows from the
relation $(q\circ a)_0=q_0a_0=q_0\circ \Ad_0(a)$ (which holds because
$F$ is a functor which restricts to $\Ad_0$ on
$\L=\Hom_{\ClRep^\times_{w}}(\eta)$, in particular we have $a_0=\Ad_0(a)$). This implies that \eqref{taueq}
holds and hence $(Q,\tau)$ is a Lipschitz structure of type
$\eta$. To show that \eqref{gammadef} is well-defined, notice that
relation \eqref{f0f} gives $\Ad(a)\circ \eta=\eta\circ \Cl(\Ad_0(a))$
for any $a\in \L$, which implies (using $[\Cl(\tau_m(q))]^{-1}=\Cl(\tau_m(q)^{-1})$ and $\Ad_0(a^{-1})=\Ad_0(a)^{-1}$):
\be
\Ad(a^{-1})\circ \eta\circ \Cl(\tau_m(q)^{-1})=\eta\circ \Cl(\Ad_0(a^{-1})\circ \tau_m(q)^{-1})=\eta\circ [\Cl(\tau_m(q)\circ \Ad_0(a))]^{-1}~~.
\ee
Using relation \eqref{taueq}, this gives: 
\ben
\label{rel1}
\Ad(a^{-1})\circ \eta\circ \Cl(\tau_m(q)^{-1})=\eta\circ \Cl(\tau_m(qa)^{-1})~~\forall a\in \L~~.
\een 
Thus:
\begin{eqnarray}
& & [qa^{-1},\eta(\Cl(\tau_m(qa^{-1})^{-1})(x))(as)] = [q, a^{-1} \eta(\Cl(\tau_m(qa^{-1})^{-1})(x))(as)]=\nonumber\\ 
& & [q,(\Ad(a^{-1})\circ \eta\circ \Cl(\tau_m(qa^{-1})^{-1}))(x)(s)] =[q,\eta(\Cl(\tau_m(q)^{-1})(x))(s)]\, ,
\end{eqnarray}
where in the last equality we used \eqref{rel1}. This shows that \eqref{gammadef} is well-defined. The fact that $\gamma_m$ defined in \eqref{gammadef} is a Clifford representation is obvious, as are the remaining statements. 
\qed
\end{proof}

\begin{thm}
\label{thm:BundleLipschitz}
There exist invertible natural transformations:
\be
\cF_\eta:S_\eta\circ Q_\eta\stackrel{\sim}{\rightarrow} \id_{\ClB_w^\eta(M,g)^\times}~~\mathrm{and}~~\cG_\eta:\id_{L_\eta(M,g)} \stackrel{\sim}{\rightarrow} Q_\eta\circ S_\eta~~.
\ee
Hence the functors $Q_\eta$ and $S_\eta$ are mutually
quasi-inverse equivalences between the groupoids
$\ClB_w^\eta(M,g)^\times$ and $L_\eta(M,g)$.
\end{thm}

\begin{remark}
Theorem \ref{thm:BundleLipschitz} shows, in particular, that the
classification of weakly faithful real pinor bundles of type $[\eta]$
over $(M,g)$ (up to based isomorphism of real pinor bundles) is
equivalent with that of real Lipschitz structures relative to $[\eta]$
(up to isomorphism of Lipschitz structures). The first classification
problem asks one to determine the set of isomorphism classes of
objects in the category $\ClB_w^\eta(M,g)$, while the latter asks for
the set of isomorphism classes of objects in the category
$L_\eta(M,g)$. The theorem implies that there exists a
canonically-defined bijection between these two sets.
\end{remark}

\begin{proof}
Fixing $\eta$, we denote $\cF_\eta$ by $\cF$ and $\cG_\eta$ by $\cG$ for ease of notation. 

I. {\bf Construction of $\cF$.} Let $(S,\gamma)$ be an object of
$\ClB_w^\eta(M,g)$, $(Q,\tau)\eqdef Q_\eta(S,\gamma)$ and
$(S',\gamma')\eqdef S_\eta(Q,\tau)$. For any $m\in M$, we have
$Q_m=\Hom_{\ClRep^\times}(\eta,\gamma_m)$ and
$S'_m=Q_m\times_{\rho_\eta}S_0=Q_m\times S_0/\sim$, where $\sim$ is
the following equivalence relation on pairs $(q,s)\in Q_m\times S_0$:
\be
(q,s)\simeq (q',s')~\mathrm{if}~~\exists a\in \L=\Aut_{\ClRep}(\eta): q'=q\circ a~~\mathrm{and}~~s'=a^{-1}(s)~~.
\ee
The smooth map $\bar{\cF}:Q\times S_0 \rightarrow S$ given by: 
\be
\bar{\cF}_m(q,s)=q(s)\in S_m~~\forall (q,s)\in Q_m\times S_0
\ee
satisfies $\bar{\cF}(qa,a^{-1}(s))=(q\circ a)(a^{-1}(s))=q(s)$ and hence descends to
a based morphism of vector bundles $\cF:=\cF_{S,\gamma}:S'\rightarrow S$. Since $q$ is an
$\R$-linear bijection from $S_0$ to $S_m$, the condition $\bar{\cF}(q,s)=0$
is equivalent with $s=0$, which implies that $\cF_m$ is injective
for all $m\in M$. Since $S$ has type $\eta$, we have $\rk
S'=\dim_\R S_0=\rk S$ and hence $\dim_\R S'_m=\dim_\R S_m$. Thus $\cF_m$
is bijective for all $m$ and hence $\cF$ is a based isomorphism of
vector bundles. The bundle $S'$ is endowed with the Clifford module 
structure given by:
\be
\gamma'_m(x)([q,s])=[q,\eta(\Cl(\tau(q)^{-1})(x))(s)] ~~,~~\forall x\in \Cl(T_m^\ast M,g_m^\ast)~~,~~(q\in Q_m, s\in S_0)~~,
\ee
where $\tau(q)=q_0:\in \Hom_{\Quad^\times}((V,h), (T_m^\ast
M,g_m^\ast))$ for any $q\in Q_m$. Since $q^{-1}\in
\Hom_{\ClRep^\times}(\gamma_m,\eta)$ is an isomorphism of Clifford
representations, we have $\eta\circ \Cl(\tau(q)^{-1})=\eta\circ
\Cl(q_0^{-1})=\Ad(q^{-1})\circ \gamma_m$ (see relation
\eqref{f0f}). Thus $\eta(\Cl(\tau(q)^{-1})(x))$ $=q^{-1}\circ
\gamma_m(x)\circ q\in \End_\R(S_0)$ and:
\be
(\cF_m\circ \gamma'_m(x))([q,s])=\cF_m([q,q^{-1}(\gamma_m(x)(q(s)))])=\gamma_m(x)(q(s))=(\gamma_m(x)\circ \cF_m)([q,s])~~,
\ee 
which shows that $\cF_m\circ \gamma'_m(x)=\gamma_m(x)\circ \cF_m$ for
all $m\in M$. Hence $\cF=\cF_{S,\gamma}$ is a based isomorphism of
pinor bundles, i.e. $\cF_{S,\gamma}\in
\Hom_{\ClB(M,g)^\times}((S',\gamma'),(S,\gamma))=\Hom_{\ClB(M,g)^\times}((S_\eta\circ
Q_\eta)(S,\gamma),(S,\gamma))$. To show that $\cF_{S,\gamma}$ is a
natural transformation, let $(S_1,\gamma_1)$ and $(S_2,\gamma_2)$ be
two isomorphic weakly faithful pinor bundles of type $\eta$ and
$f\in \Hom_{\ClB(M,g)^\times}((S_1,\gamma_1),(S_2,\gamma_2))$ be a
based isomorphism of pinor bundles. Then $f_m\in
\Hom_{\ClRep^\times}(\gamma_{1,m},\gamma_{2,m})$ is an isomorphism of
Clifford representations for every $m\in M$. Let $(Q_i,\tau_i)\eqdef
Q_\eta(S_i,\gamma_i)$ and $(S'_i,\gamma'_i)\eqdef
S_\eta(Q_i,\tau_i)$. Let $\cF_i\eqdef
\cF_{S_i,\gamma_i}:(S'_i,\gamma'_i)\stackrel{\sim}{\rightarrow}(S_i,\gamma_i)$
be the based isomorphisms of pinor bundles constructed as
above. Finally, let $g\eqdef
Q_\eta(f):(Q_1,\tau_1)\stackrel{\sim}{\rightarrow} (Q_2,\tau_2)$ be
the isomorphism of Lipschitz structures induced by $f$ and $f'\eqdef
S_\eta(g)=S_\eta(Q_\eta(f)):(S'_1,\gamma'_1)\stackrel{\sim}{\rightarrow}(S'_2,\gamma'_2)$
be the isomorphism of pinor bundles induced by $g$. For any
$[q_1,s]\in S'_{1,m}=Q_{1,m}\times_{\rho_\eta}
S_0=\Hom_{\ClRep^\times}(\eta,\gamma_{1,m})\times_{\rho_\eta}S_0$, we
have $f'_m([q_1,s])=[g(q_1),s]=[f_m\circ q_1,s]$ (see Proposition
\ref{propdef:functors}). Thus:
\be
(\cF_2\circ f')_m([q_1,s])=(f_m\circ q_1)(s)=f_m(q_1(s))=(f\circ \cF_1)_m([q_1,s])~~,
\ee
which shows that $\cF_2\circ f'=f\circ \cF_1$, i.e. $\cF_2\circ (S_\eta\circ Q_\eta)(f)=f\circ
\cF_1$. Hence the diagram
in Figure \ref{diagram:Natural1} commutes. Thus $\cF:S_\eta\circ
Q_\eta\stackrel{\sim}{\rightarrow} \id_{\ClB_w^\eta(M,g)^\times}$ is an
invertible natural transformation.
\begin{equation}
\label{diagram:Natural1}
\scalebox{1.0}{
	\xymatrix{
		(S^{\prime}_{1},\gamma^{\prime}_{1})~\ar[d]_{(S_{\eta}\circ Q_{\eta})(f)} \ar[r]^{\cF_{1}} ~ & ~ (S_{1},\gamma_{1})\ar[d]^{f}\\
		(S^{\prime}_{2},\gamma^{\prime}_{2})\ar[r]^{\cF_{2}} ~ & ~ (S_{2},\gamma_{2}) \\
	}}
	\end{equation}

\

II. {\bf Construction of $\cG$.} Let $(Q,\tau)$ be an object of $L_\eta(M,g)$ and
$(S,\gamma)=S_\eta(Q,\tau)$. Let $(Q',\tau')\eqdef Q_\eta(S,\gamma)$. We have
$Q'_m=\Hom_{\ClRep^\times}(\eta,\gamma_m)$ and $\tau'_m(q')=q'_0$ for
any $q'\in Q'_m$. On the other hand, we have
$S_m=Q_m\times_{\rho_\eta}S_0$ and $\gamma_m$ is given by
\eqref{gammadef}. For any $q\in Q_m$, let $\cG_m(q):S_0\rightarrow
S_m$ be the linear map defined through:
\be
\cG_m(q)(s)=[q,s]~~(s\in S_0)~~.
\ee 
Then $\cG_m(q)$ is a linear isomorphism from $S_0$ to $S_m$ and for all $a\in \L$ and $s\in S_0$ 
we have $\cG_m(qa)(s)=[qa,s]=[q,a(s)]=\cG_m(q)(a(s))$, i.e.:
\ben
\label{Ga}
\cG_m(qa)=\cG_m(q)\circ \rho_\eta(a)=\cG_m(q)\circ a~~(a\in \L)~~.
\een
For any $x\in \Cl(T_m^\ast M,g_m^\ast)$ and any $s\in S_0$, relation \eqref{gammadef} gives:
\begin{eqnarray}
(\gamma_m(x)\circ \cG_m(q))(s)=\gamma_m(x)([q,s])=[q,\eta(\Cl(\tau(q)^{-1})(x))(s)]\nonumber\\=(\cG_m(q)\circ \eta(\Cl(\tau(q)^{-1})(x)))(s)\, ,
\end{eqnarray}
i.e. $\gamma_m(x)\circ \cG_m(q)=\cG_m(q)\circ \eta(\Cl(\tau(q)^{-1})(x))$. Recall that $\tau(q)\in P_{\O(V,h)}(M,g)_m=\Hom_{\Quad^\times}((V,h),(T^\ast_m M,g_m^\ast))$, 
so $\Cl(\tau(q))$ is a unital isomorphism of $\R$-algebras from $\Cl(V,h)$ to $\Cl(T_m^\ast M, g_m^\ast$). 
Replacing $x$ with $\Cl(\tau(q))(y)$ where $y\in \Cl(V,h)$, the relation above gives:  
\be
\Ad(\cG_m(q))\circ \eta=\gamma_m\circ \Cl(\tau(q))~~.
\ee
Using relation \eqref{f0f}, this shows that $\cG_m(q)\in Q'_m=\Hom_{\ClRep^\times}(\eta,\gamma_m)$ and that
we have: 
\ben
\label{tauequiv}
\tau'(\cG_m(q))= (\cG_m(q))_0=\tau(q)\in \Hom_{\Quad^\times}((V,h),(T_m^\ast
M,g_m^\ast))~~.
\een 
Relation \eqref{Ga} reads: 
\be
\cG_m(qa)=\cG_m(q)a~~,
\ee
where in the right hand side we use the right $\L$-action on the bundle
$Q'$. Thus $(\cG_m)_{m\in M}$ define a morphism of principal
$\L$-bundles $\cG=\cG_{Q,\tau}:Q\rightarrow Q'$ (which is automatically an
isomorphism). Relation \eqref{tauequiv} gives $\tau'\circ \cG=\tau$,
showing that we have $\cG\in \Hom_{L_\eta(M,g)}(Q,Q')$.

To show that $\cG$ gives a natural transformation from
$\id_{L_\eta(M,g)}$ to $Q_\eta\circ S_\eta$, consider two isomorphic
objects $(Q_i,\tau_i)$ ($i=1,2$) of
$L_\eta(M,g)$ and let $(S_i,\gamma_i)\eqdef S_\eta(Q_i,\tau_i)$ and
$(Q'_i,\tau'_i)\eqdef Q_\eta(S_i,\gamma_i)$. Let $\cG_i\eqdef
\cG_{Q_i, \tau_i}:Q_i\stackrel{\sim}{\rightarrow} Q'_i$ be the isomorphisms
of Lipschitz structures defined as above. For any isomorphism $f\in
\Hom_{L_\eta(M,g)}((Q_1,\tau_1),(Q_2,\tau_2))$ and any elements $q_1\in Q_{1,m}$ and $s\in
S_0$, we have:
\begin{eqnarray}
&&(\cG_2\circ f)(q_1)(s)=\cG_2(f(q_1))(s)=[f(q_1),s]=S_\eta(f)([q_1,s])=S_\eta(f)(\cG_1(q_1)(s))\nn\\&&=(S_\eta(f)\circ \cG_1(q_1))(s)=(Q_\eta\circ S_\eta)(f)(\cG_1(q_1))(s)\, ,
\end{eqnarray}
which implies $\cG_2\circ f=(Q_\eta\circ S_\eta)(f)\circ \cG_1$ and hence
the diagram in Figure \ref{diagram:Natural2} commutes. Thus
$\cG:\id_{L_\eta(M,g)}\stackrel{\sim}{\rightarrow} Q_\eta\circ S_\eta$
is an isomorphism of functors.  

\begin{equation}
\label{diagram:Natural2}
\scalebox{1.0}{
	\xymatrix{
		(Q_1,\tau_{1}) ~\ar[d]_{f} \ar[r]^{\cG_{1}} ~ & ~ (Q'_1,\tau^{\prime}_{1})\ar[d]^{(Q_{\eta}\circ S_\eta)(f)}\\
		(Q_2,\tau_{2}) \ar[r]^{\cG_{2}} ~ & ~ (Q'_2,\tau^{\prime}_{2}) \\
	}}
	\end{equation}
\qed
\end{proof}

\section{Some enlarged spinorial structures in general signature} 
\label{sec:structures}

In this section, we discuss certain enlarged spinorial structures
associated to the groups appearing in the list of canonical spinor
groups of Section \ref{sec:enlargedspinor}. In particular, we discuss
the topological obstructions to existence of such structures (some of
which are known or are extensions of known results to arbitrary
signature and some of which are new) as well as the behavior of $\Pin$
and $\Pin^q$ structures under sign reversal of the metric in even
dimensions.

\subsection{Modified Stiefel-Whitney classes of a pseudo-Riemannian manifold}

Let $M$ be a connected pseudo-Riemannian manifold (which need not
be orientable) and $(V,h)$ be a quadratic vector space of signature $(p,q)$ and 
dimension $d=p+q$. Let $P$ be a principal $\O(V,h)$-bundle. Recall that
$\O(p)\times \O(q)$ is a maximal compact form of $\O(p,q)$ and the 
inclusion morphism $j:\O(p)\times \O(q)\rightarrow \O(p,q)$ is a
deformation retract. On the other hand, any pseudo-orthonormal basis
of $(V,h)$ determines an isomorphism of groups
$\xi:\O(p,q)\stackrel{\sim}{\rightarrow}\O(V,h)$ and the isomorphisms
determined by two such bases differ by conjugation. The morphism of
groups $j_\xi\eqdef \xi\circ j:\O(p)\times \O(q)\rightarrow \O(V,h)$ is a
deformation retract which induces a bijection
$(j_\xi)_\ast:H^1(M,\O(p))\times H^1(M,\O(q))\stackrel{\sim}{\rightarrow}
H^1(M,\O(V,h))$.  It follows that every principal $\O(V,h)$-bundle $P$
over $M$ is isomorphic with a fiber product $P_+\times_M P_-$, where
$P_+$ is a principal $\O(p)$-bundle and $P_-$ is a principal
$\O(q)$-bundle and where $P_+$ and $P_-$ are determined by $P$ up to
isomorphism. The same conclusion also follows from the homotopy
equivalence of classifying spaces $\BO(V,h)\simeq \BO(p)\times
\BO(q)$. The {\em modified Stiefel-Whitney classes of $P$} are
defined \cite{Karoubi} as the Stiefel-Whitney classes of $P_\pm$:
\be
\w_k^\pm(P)\eqdef \w_k(P_\pm)\in H^k(M,\Z_2)~~.
\ee
Notice that $\w_1(P)=\w_1^+(P)+\w_1^-(P)$.

Assume now that $P=P_{\O(V_0,h_0)}(M,g)\simeq P_{\O(V,h)}(M,g)$ is the pseudo-orthogonal frame bundle
of a connected pseudo-Riemannian manifold $(M,g)$, where $(V_0,h_0)$ is an
isometric model of the tangent spaces $(T_mM, g_m)$ of $M$. Then:
\be
\w_k^\pm(M,g)\eqdef \w_k^\pm(P_{\O(V_0,h_0)}(M,g))=\w_k^\pm(P_{\O(V,h)}(M,g))~~.
\ee
are called the {\em modified Stiefel-Whitney classes of $(M,g)$}.
Since $T M$ is associated to $P_{\O(V_0,h_0)}(M,g)$ through the
tautological representation of $\O(V_0,h_0)$, we have a corresponding
Whitney sum decomposition $T M=T_+ M\oplus T_- M$, where $T_\pm M$ are
vector sub-bundles of $T M$ which are associated to $\O(p)$ and
$\O(q)$ through the tautological representations of those groups;
these sub-bundles are determined only up to isomorphism \footnote{They
  can be constructed as the sums of positive and negative eigenbundles of the
  $g_0$-symmetric endomorphism $A$ of $T M$ which represents $g$ with
  respect to any Riemannian metric $g_0$ on $M$; changing $g_0$
  leads to isomorphic bundles.}. We have $\w_k^\pm(M,g)=\w_k(T^\pm M)$
and $\w_1(M)=\w_1^+(M,g)+\w_1^-(M,g)$, where $\w_k(M)$ are the usual
Stiefel-Whitney classes of $TM$.

\subsection{Spin structures in general signature} 

Recall that a spin structure on $(M,g)$ is an $\Ad_0^\Cl$-reduction of
$P_{\O(V,h)}(M,g)\simeq P_{\O(V_0,h_0)}(M,g)$ to $\Spin(V,h)$. Since
$\Ad_0(\Spin(V,h))=\SO(V,h)$, Remark \ref{rem:twostepreduction}
implies that a spin structure can exist on $(M,g)$ only when the
structure group of $P_{\O(V,h)}(M,g)$ reduces to $\SO(V,h)$, i.e.
only when $M$ is orientable. In that case, a spin structure is the
same as an $\Ad_0^\Cl$-reduction of $P_{\SO(V,h)}(M,g)$ to
$\Spin(V,h)$, where $P_{\SO(V,h)}(M,g)$ is the bundle of
positively-oriented pseudo-orthogonal coframes (which is naturally
isomorphic with the bundle of positively-oriented pseudo-orthogonal
frames) with respect to an orientation of $(M,g)$ determined by the
spin structure. Notice that changing the orientation of $(M,g)$
changes $P_{\SO(V,h)}(M,g)$ into an isomorphic bundle. The following
result was proved in \cite{Karoubi}:

\begin{prop}\cite{Karoubi}
\label{prop:SpinObs}
The following  statements are equivalent:
\begin{enumerate}[(a)]
\itemsep 0.0em
\item $(M,g)$ admits a $\Spin$ structure
\item The following conditions are satisfied:
\be
\w_1(M)=0~~\mathrm{and}~~ \w_2^+(M,g)+\w_2^-(M,g)=0~~.
\ee
\end{enumerate}
\end{prop}

\subsection{Twisted and untwisted $\Pin$ structures in even dimension}
\label{app:Pin}

Assume that $d$ is even. In this case, both the twisted and untwisted
vector representations of $\Pin(V,h)$ are surjective onto $\O(V,h)$,
hence one can define {\em two} kinds of pin structures on $(M,g)$.

\begin{definition}
A {\em twisted pin structure} on $(M,g)$ is a $\tAd_0^\Cl$-reduction
of $P_{\O(V,h)}(M,g)$ to $\Pin(V,h)$. An {\em untwisted pin structure}
on $(M,g)$ is an $\Ad_0^\Cl$-reduction of $P_{\O(V,h)}(M,g)$ to
$\Pin(V,h)$.
\end{definition}

\noindent Both twisted and untwisted pin structures on $(M,g)$ form
groupoids if one takes morphisms to be isomorphisms of $\tAd_0$-,
respectively $\Ad_0$-reductions. Notice the equality
$P_{\O(V,h)}(M,g)=P_{\O(V,h)}(M,-g)$. Proposition \ref{prop:AdtAdPin}
implies the following relation between the two types of pin structure:

\begin{prop} 
\label{prop:Pinpm}
When $d$ is even, there exists an equivalence of categories between
the groupoid of twisted pin structures of $(M,g)$ and the groupoid of
untwisted pin structures of $(M,-\sigma_{p,q}g)$. Namely:
\begin{enumerate}[1.]
\itemsep 0.0em
\item When $p-q\equiv_8 0,4$, the groupoid of twisted pin structures
  of $(M,g)$ is equivalent with the groupoid of untwisted pin
  structures of $(M,-g)$.
\item When $p-q\equiv_8 2,6$, the groupoids of twisted and untwisted
  pin structures of $(M,g)$ are equivalent to each other.
\end{enumerate}
\end{prop}
\noindent In particular, twisted and untwisted pin structures are
equivalent notions when $p-q\equiv_8 2,6$.

\begin{proof}
Let $\sigma\eqdef \sigma_{p,q}$.  If $(Q,\tau)$ is an untwisted pin
structure on $(M,-\sigma g)$, define $Q'$ to be the principal
$\Pin(V,h)$-bundle over $M$ having the same total space as $Q$ and
right group action given by:
\be
q\ast a\eqdef q\varphi(a)~~\forall q\in Q~~\forall a\in \Pin(V,h)~~,
\ee
where $\varphi:\Pin(V,h)\stackrel{\sim}{\rightarrow}\Pin(V,-\sigma h)$
is the ismorphism of groups given in Proposition
\ref{prop:AdtAdPin}. Let $\tau'\eqdef \tau:Q\rightarrow
P_{\O(V,-\sigma h)}(M,-\sigma g)=P_{\O(V,h)}(M,g)$.  Since
$\Ad_0^\Cl\circ \varphi=\tAd_0^\Cl$, we have:
\be
\tau'(q\ast a)=\tau(q\varphi(a))=\tau(q)\Ad_0^\Cl(\varphi(a))=\tau'(q)\tAd_0^\Cl~~\forall q\in Q~~\forall a\in \Pin(V,h)~~,
\ee
where we used the fact that $\tau$ is $\Ad_0$-equivariant. This shows
that $\tau'$ is $\tAd_0$-equivariant and hence $(Q',\tau')$ is a
twisted pin structure on $(M,g)$. Let $(Q_1,\tau_1)$, $(Q_2,\tau_2)$
be untwisted pin structures on $(M,-\sigma g)$ and
$f:(Q_1,\tau_1)\rightarrow (Q_,\tau_2)$ be an isomorphism of untwisted
pin structures. Let $(Q'_1,\tau'_1)$ and $(Q'_2,\tau'_2)$ be the
twisted pin structures on $(M,g)$ defined by $(Q_1,\tau_1)$ and
$(Q_2,\tau_2)$ as above. For all $q_1\in Q_1$ and $a\in \Pin(V,h)$, we
have:
\be
f(q_1\ast a)=f(q_1\varphi(a))=f(q_1)\varphi(a)=f(q_1)\ast a~~,
\ee
and: 
\be
\tau'_2\circ f=\tau_2\circ f=\tau_1=\tau'_1~~,
\ee
which shows that $f'\eqdef f$ is an isomorphism of twisted pin
structures from $(Q'_1,\tau'_1)$ to $(Q'_2,\tau'_2)$. It is easy to
check that the correspondence defined above is an equivalence of
categories.
\qed
\end{proof}

\begin{remark}
In most of the literature, the name ``$\Pin$ structure'' is reserved for
what we call a {\em twisted} pin structure.  One sometimes also
encounters the notion of ``$\Pin^-$ structure'' of $(M,g)$, which is
defined as a {\em twisted} pin structure of $(M,-g)$. The proposition
implies the following:
\begin{enumerate}[1.]
\itemsep 0.0em
\item When $p-q\equiv_8 0,4$, the groupoid of $\Pin^-$ structures
  of $(M,g)$ is equivalent with the groupoid of untwisted $\Pin$
  structures of $(M,g)$.
\item When $p-q\equiv_8 2,6$, the groupoid of $\Pin^-$ structures
  of $(M,g)$ is equivalent with the groupoid of untwisted $\Pin$
  structures of $(M,-g)$.
\end{enumerate}
\end{remark}

\begin{prop}
\label{prop:PinObs}
Let $d$ be even and $\sigma\eqdef \sigma_{p,q}$.
Then the following  statements are equivalent:
\begin{enumerate}[(a)]
\itemsep 0.0em
\item $(M,g)$ admits an untwisted $\Pin$ structure
\item $(M,-\sigma g)$ admits a twisted $\Pin$ structure
\item The following condition is satisfied:
\be
\w_2^+(M,g)+\w_2^-(M,g)+\w_1^{\sigma}(M,g)^2+\w_1^-(M,g)\w_1^+(M,g)=0~~.
\ee
\end{enumerate}
\end{prop}

\begin{proof} 
The equivalence of (a) and (b) follows from Proposition
\ref{prop:Pinpm}. The equivalence of (b) and (c) follows from
\cite[Proposition (1.1.26)]{Karoubi} upon noticing the relation:
\be
\w_1^\pm(M,-g)=\w_1^{\mp}(M,g)~~.
\ee
\qed
\end{proof}

\subsection{Twisted and untwisted $\Pin^q(V,h)$ structures in even dimension}
\label{app:Pinq} 

In even dimension, both the twisted and untwisted basic representation
of $\Pin^q(V,h)$ have image equal to $\O(V,h)\times \SO(3,\R)$ (see
Subsection \ref{sec:pinq}). This allows us to define two kinds of
$\Pin^q$ structures on $(M,g)$.

\begin{definition}
Let $d$ be even. Then: 
\begin{enumerate}[1.]
\itemsep 0.0em
\item A {\em twisted $\Pin^q$ structure} on $(M,g)$ is a triplet
  $(E,Q,\tau)$, where $E$ is a principal $\SO(3,\R)$-bundle over $M$ and
  $(Q,\tau)$ is a $\trho$-reduction of $P_{\O(V,h)}(M,g)\times_M E$ to $\Pin^q(V,h)$.
\item An {\em untwisted $\Pin^q$ structure} on $(M,g)$ is a triplet
  $(E,Q,\tau)$, where $E$ is a principal $\SO(3,\R)$-bundle over $M$ and
  $(Q,\tau)$ is a $\rho$-reduction of $P_{\O(V,h)}(M,g)\times_M E$ to $\Pin^q(V,h)$.
\end{enumerate}
\end{definition}

\noindent Proposition \ref{prop:rhotrho} implies the following result,
whose proof is similar to that of Proposition \ref{prop:Pinpm}:

\begin{prop} 
\label{prop:PinqtPinq}
When $d$ is even, there exists an equivalence of categories between
the groupoid of twisted $\Pin^q(V,h)$ structures of $(M,g)$ and that
of untwisted $\Pin^q(V,-\sigma_{p,q}h)$ structures of $(M,g)$. Namely: 
\begin{enumerate}[1.]
\itemsep 0.0em
\item When $p-q\equiv_8 0,4$, the groupoid of twisted $\Pin^q$ structures
  of $(M,g)$ is equivalent with the groupoid of untwisted $\Pin^q$
  structures of $(M,-g)$.
\item When $p-q\equiv_8 2,6$, the groupoids of twisted and untwisted
  $\Pin^q$ structures of $(M,g)$ are equivalent to each other.
\end{enumerate}
\end{prop}

\begin{prop}
\label{prop:PinqObs}
Let $d$ be even and $\sigma\eqdef \sigma_{p,q}$. Then the following
statements are equivalent:
\begin{enumerate}[(a)]
\itemsep 0.0em
\item $(M,g)$ admits an untwisted $\Pin^q$ structure
\item $(M,-\sigma g)$ admits a twisted $\Pin^q$ structure
\item There exists a principal $\SO(3,\R)$-bundle $E$ over $M$ such that
  the following condition is satisfied:
\be
\w_2^+(M,g)+\w_2^-(M,g)+\w_1^{\sigma}(M,g)^2+\w_1^-(M,g)\w_1^+(M,g)=\w_2(E)~~.
\ee
\end{enumerate}
\end{prop}

\begin{proof}
Equivalence of (a) and (b) follows from Proposition
\ref{prop:PinqtPinq}. To show equivalence of (b) and (c), notice that
the short exact sequence \eqref{Pinqtseq} induces the exact sequence of
pointed sets:
\ben
\label{Pinqle}
H^1(M,\Pin^q(V,h))\stackrel{\trho_\ast}\longrightarrow H^1(M,\O(V,h)\times \SO(3,\R))\stackrel{\partial}{\longrightarrow} H^2(M,\Z_2)~~,
\een 
where $H^1(M,\O(V,h)\times \SO(3,\R))=H^1(M,\O(V,h))\oplus
H^1(M,\SO(3,\R))$. When $P$ is an $\O(V,h)$-bundle on $M$, the
connecting map is given by $\partial([P\times_M
  E])=\partial'([P])+\w_2([E])=\w_2^+(P)+\w_2^-(P)+\w_1^-(P)^2+\w_1^-(P)\w_1^+(P)+\w_2(E)$,
where we used the fact that $\Sp(1)\simeq \Spin(3)=\Spin_{3,0}$ and
the connecting map in the sequence:
\be
H^1(M,\Sp(1))\stackrel{(\Ad_0^{(2)})_\ast}{\longrightarrow} H^1(M,\SO(3,\R))\longrightarrow H^2(M,\Z_2)
\ee
induced by the short exact sequence $1\rightarrow \Z_2\rightarrow
\Spin(3)\rightarrow \SO(3,\R)\rightarrow 1$ is given by
$[E]\rightarrow \w_2(E)$, while the connecting map $\partial'$ in the
sequence:
\be
H^1(M,\Pin(V,h))\stackrel{(\tAd_0^\Cl)_\ast}{\longrightarrow} H^1(M,\O(V,h))\stackrel{\partial'}{\longrightarrow} H^2(M,\Z_2)
\ee
induced by the exact sequence \eqref{pinseqtAd} is given by
$\partial'([P])=\w_2^+(P)+\w_2^-(P)+\w_1^-(P)^2+\w_1^-(P) \w_1^+(P)$
(see \cite{Karoubi}). Thus $(M,g)$ admits a twisted $\Pin^q(V,h)$
structure iff
$\w_2^+(M,g)+\w_2^-(M,g)+\w_1^-(M,g)^2+\w_1^-(M,g)\w_1^+(M,g)=\w_2(E)$. Applying
this to $(M,-\sigma g)$ shows that (b) is equivalent with (c).
\qed
\end{proof}

\begin{remark}
The connecting map $\partial$ in the sequence \eqref{Pinqle} can also
be determined directly as follows.  First, notice that we must have
$\partial([P]\oplus
[E])=a_1\w_2^+(P)+a_2\w_2^-(P)+a_3\w_1^-(P)^2+a_4\w_1^-(P)
\w_1^+(P)+a_5 \w_1^+(P)^2+a_6 \w_2(E)$ (since $\w_1(E)=0$). When $E$
is trivial, we have $\w_2(E)=0$ and it is easy to see that the bundle
$P\times_M E$ admits a twisted $\Pin^q$ structure iff $P$ admits a
twisted pin structure\footnote{The inverse implication is obvious. For
  the direct implication, notice that the natural inclusion
  $\Pin(V,h)\simeq \Pin(V,h)\cdot \Z_2 \subset \Pin^q(V,h)$ allows one
  to reduce the structure group of a pin structure $Q$ to $\Pin(V,h)$
  when $E$ is trivial; indeed, the characteristic representation $\mu$
  of $\Pin^q(V,h)$ induces an ismorphism $Q/\Pin(V,h)\simeq E$ and $E$
  has a section when it is trivial.}. This implies that
$\partial([P])$ must be given by the expression computed in
\cite{Karoubi}, i.e. we must have $a_1=a_2=a_3=a_4=1$ and $a_5=0$. On
the other hand, when $p=d$, $q=0$ and $\w_1(P)=0$, a $\Pin^q$
structure on $P$ reduces to a $\Spin^q$ structure on $P$, hence in
this case we have $\partial([P]\oplus [E])=\w_2(P)+\w_2(E)$ by the
results of \cite{Nagase}; this shows that $a_6=1$.
\end{remark}

\subsection{$\Spin^q$ structures in general signature}
\label{app:Spinq}

Let $\rho:\Spin^q(V,h)\rightarrow \SO(V,h)\times \SO(3,\R)$ denote the
basic representation of $\Spin^q(V,h)$ (see Subsection \ref{sec:spinq}).

\begin{definition}
A {\em $\Spin^q$ structure} on $(M,g)$ is a triplet $(E,Q,\tau)$, where
$E$ is a principal $\SO(3,\R)$-bundle over $M$ and $(Q,\tau)$ is a
$\rho$-reduction of $P_{\O(V,h)}(M,g)\times_M E$ to $\Spin^q(V,h)$.
\end{definition}

\noindent Since the image of $\rho$ equals $\SO(V,h)\times \SO(3,\R)$,
it follows that $(M,g)$ can admit a $\Spin^q$ structure only when $M$
is orientable; hence we can assume that $\w_1(M)=0$. When this
condition is satisfied, a $\Spin^q$ structure is the same as a
$\rho$-reduction of $P_{\SO(V,h)}(M,g)\times_M E$, where
$P_{\SO(V,h)}(M,g)$ is the principal bundle of positively-oriented
pseudo-orthogonal frames with respect to an orientation of $M$
determined by the $\Spin^q$ structure. In fact, we have:

\begin{prop}
\label{prop:SpinqObs}
The following statements are equivalent:
\begin{enumerate}[(a)]
\item $(M,g)$ admits a $\Spin^q$ structure
\item There exists a principal $\SO(3,\R)$-bundle $E$ over $M$ such that
  the following conditions are satisfied:
\be
\w_1(M)=0~~\mathrm{and}~~\w_2^+(M,g)+\w_2^-(M,g)=\w_2(E)~~.
\ee
\end{enumerate}
\end{prop}

\begin{proof}  
The short exact sequence \eqref{Spinqseq} induces the exact sequence of pointed sets:
\be
H^1(M,\Spin^q(V,h))\stackrel{\rho_\ast}\longrightarrow H^1(M,\SO(V,h)\times \SO(3,\R))\stackrel{\partial}{\longrightarrow} H^2(M,\Z_2)~~, 
\ee 
whose connecting map satisfies $\partial([P\times_M E])=\partial'([P])+\w_2(E)=\w_2^+(P)+\w_2^-(P)+\w_2(E)$, where we used the fact that the
connecting map $\partial'$ in the sequence: 
\be
H^1(M,\Spin(V,h))\stackrel{(\Ad_0^\Cl)_\ast}{\longrightarrow} H^1(M,\SO(V,h))\stackrel{\partial'}{\longrightarrow} H^2(M,\Z_2)
\ee
induced by \eqref{spinseqtAd} is given by
$\partial'([P])=\w_2^+(P)+\w_2^-(P)$ (see \cite{Karoubi}). Thus
$(M,g)$ admits a $\Spin^q(V,h)$ structure iff $\w_1(M)=0$ and
$\w_2^+(M,g)+\w_2^-(M,g)=\w_2(E)$.  \qed
\end{proof}

\subsection{$\Spin^o$ structures in signature $p-q\equiv_8 3,7$}

Assume that the signature $(p,q)$ of $(M,g)$ belongs to the complex
case, i.e. it satisfies $p-q\equiv_8 3,7$ (in particular, $d=p+q$ is
odd). Recall that $\alpha_{p,q}\eqdef
(-1)^{\frac{p-q+1}{4}}=\twopartdef{-1}{p-q\equiv_8 3}{+1}{p-q\equiv_8
  7}$. Let $\rho:\Spin^o(V,h)\rightarrow \mathrm{S}[\O(V,h)\times
  \O(2,\R)]\subset \O(V,h)\times \O(2,\R)$ denote the basic
representation of the adapted $\Spin^o$ group
$\Spin^o(V,h)=\Spin^o_{\alpha_{p,q}}(V,h)$ (see Subsection
\ref{sec:spino}).

\begin{definition}
An {\em adapted $\Spin^o$ structure} on $(M,g)$ is a triplet $(E,Q,\tau)$,
  where $E$ is a principal $\O(2,\R)$-bundle over $M$ and $(Q,\tau)$ is a
  $\rho$-reduction of $P_{\O(V,h)}(M,g)\times E$ to $\Spin^o(V,h)$.
\end{definition}

\noindent The following result is proved in \cite{spino}: 

\begin{prop}\cite{spino}
\label{prop:SpinoObs}
Let $d$ be odd and $\alpha\eqdef \alpha_{p,q}$. Then the following
statements are equivalent:
\begin{enumerate}[(a)]
\item $(M,g)$ admits an adapted $\Spin^o$ structure
\item There exists a principal $\O(2,\R)$-bundle $E$ over $M$ such that the
  following two conditions are satisfied:
\be
\w_1(M)=\w_1(E)
\ee
and: 
\begin{eqnarray}
\w_2^+(M,g)+\w_2^-(M,g) &=& \w_2(E)+\w_1(E)(p\w_1^+(M,g)+q\w_1^-(M,g))\nonumber\\ &+& \left[\delta_{\alpha,-1}+\frac{p(p+1)}{2}+\frac{q(q+1)}{2}\right]\w_1(E)^2~~.
\end{eqnarray}
\end{enumerate}
\end{prop}

\section{Elementary real pinor bundles and elementary real Lipschitz structures}
\label{sec:elementary2}

In this section, we consider the classification of bundles of
irreducible Clifford modules over $(M,g)$ and extract the topological
obstruction to existence of such bundles in every dimension and
signature. Let $(M,g)$ be a connected and second countable
pseudo-Riemannian manifold of signature $(p,q)$ and dimension $d=p+q$,
where we assume that $d>0$. Let $(V,h)$ be a model for the fibers of
the pseudo-Euclidean vector bundle $(T^\ast M,g^\ast)$.

\begin{definition}
An {\em elementary real pinor bundle} over $(M,g)$ is a real pinor
bundle $(S,\gamma)$ such that the Clifford representation
$\gamma_m:\Cl(T_m^\ast M, g_m^\ast)\rightarrow \End_\R(S_m)$ is a pin
representation for every point $m\in M$.
\end{definition}

\noindent Notice that an elementary real pinor bundle is weakly
faithful. If $\eta:\Cl(V,h)\rightarrow \End_\R(S_0)$ is a real Clifford
representation such that $(S,\gamma)$ has type $\eta$, then
$(S,\gamma)$ is elementary iff $\eta$ is a pin representation.

\begin{definition}
A {\em elementary real Lipschitz structure} of $(M,g)$ is a real Lipschitz
structure relative to a pin representation
$\eta:\Cl(V,g)\rightarrow \End_\R(S_0)$. A {\em reduced elementary
  real Lipschitz structure} of $(M,g)$ is a reduced real Lipschitz
structure relative to a pin representation.
\end{definition}

\begin{definition}
The {\em characteristic group} of $(V,h)$ is the compact Lie group
$\fG(V,h)$ defined as follows:
\begin{enumerate}[1.]
\itemsep 0.0em
\item $\fG\eqdef 1$ (the trivial one-element group), in the normal
  simple or non-simple case (i.e. if $p-q\equiv_8 0,1,2$)
\item $\fG(V,h)\eqdef \O(2,\R)$, in the complex case (i.e. if
  $p-q\equiv_8 3,7$)
\item $\fG(V,h)\eqdef \SO(3,\R)$, in the quaternionic simple or
  non-simple case (i.e. if $p-q\equiv_8 4,5,6$).
\end{enumerate}
\end{definition}

\begin{definition}
A {\em characteristic bundle} for $(M,g)$ is a principal bundle $E$
over $M$ with structure group $\fG(V,h)$.
\end{definition}

\noindent In the normal (simple or non-simple) case, $E$ is the
trivial $1:1$ cover bundle. Recall the canonical spinor group
$\Lambda(V,h)$ and its vector, characteristic and basic
representations $\lambda:\Lambda(V,h)\rightarrow \O(V,h)$,
$\mu:\Lambda(V,h)\rightarrow \fG(V,h)$ and
$\rho:\Lambda(V,h)\rightarrow \O(V,h)\times \fG(V,h)$ discussed in
Subsection \ref{sec:cangroup}.

\begin{definition}
A {\em canonical spinor structure} on $(M,g)$ is a $\lambda$-reduction
$(Q,\tau_1)$ of $P_{\O(V,h)}$ to the canonical spinor group
$\Lambda(V,h)$. A {\em modified canonical spinor structure} on
$(M,g)$ is a triplet $(E,Q,\tau)$, where $E$ is a characteristic
bundle for $(M,g)$ and $(Q,\tau)$ is a $\rho$-reduction of
$P_{\O(V,h)}\times_M E$ to $\Lambda(V,h)$.
\end{definition}

\noindent Canonical and modified canonical spinor structures on $(M,g)$ form
groupoids whose morphisms are the isomorphisms of $\lambda$-reductions
and $\rho$-reductions respectively.

\begin{prop}
\label{prop:canmodcan}
The groupoids of canonical and modified canonical spinor structures on
$(M,g)$ are equivalent.
\end{prop}

\begin{proof} 
Given a canonical spinor structure $(Q,\tau_1)$ on $(M,g)$, the
characteristic morphism $\mu:\Lambda(V,h)\rightarrow \fG(V,h)$ induces
a principal $\fG(V,h)$-bundle $E\eqdef Q\times_{\mu} \fG(V,h)$ on $M$
and a $\mu$-equivariant bundle map $\tau_2:Q\rightarrow E$ given by
$\tau_2(q)\eqdef [q,1]$. Since $\rho=\lambda\times \mu$, the bundle
map $\tau\eqdef \tau_1\times \tau_2:Q\rightarrow
P_{\O(V,h)}(M,g)\times_M E$ is $\rho$-equivariant, hence $(E, Q,\tau)$
is a modified canonical spinor structure on $(M,g)$. Conversely, let
$(E,Q,\tau)$ be a modified canonical spinor structure on
$P_{\O(V,h)}(M,g)$ and set $\tau_1=\pi_1\circ \tau$, where
$\pi_1:P_{\O(V,h)}(M,g)\times_M E\rightarrow P_{\O(V,h)}(M,g)$ is the
bundle map given by fiberwise projection on the first factor. Then
$(Q,\tau_1)$ is a canonical spinor structure on $(M,g)$.  It is easy
to see that the correspondences defined above extend to mutually
quasi-inverse functors between the groupoids of canonical and modified
canonical spinor structures on $(M,g)$ \qed
\end{proof}

\begin{thm}
\label{thm:canspinor}
The following groupoids are equivalent for any pseudo-Riemannian
manifold $(M,g)$:
\begin{enumerate}[(a)]
\itemsep 0.0em
\item The groupoid of elementary real pinor bundles of $(M,g)$.
\item The groupoid of elementary real Lipschitz structures.
\item The groupoid of elementary reduced real Lipschitz structures.
\item The groupoid of canonical spinor structures.
\item The groupoid of modified canonical spinor structures.
\end{enumerate}
Depending on the dimension and signature, this groupoid equals:
\begin{enumerate}[1.]
\itemsep 0.0em
\item When $p-q\equiv_8 0,2$ (normal simple case): the groupoid of
  {\em untwisted} $\Pin$ structures.
\item When $p-q\equiv_8 3,7$ (complex case): the groupoid of $\Spin^o$
  structures
\item When $p-q\equiv_8 4, 6$ (quaternionic simple case): the groupoid
  of {\em untwisted} $\Pin^q$ structures.
\item When $p-q\equiv_8 1$ (normal non-simple case): the groupoid of
  $\Spin$ structures.
\item When $p-q\equiv_8 5$ (quaternionic non-simple case): the
  groupoid of $\Spin^q$ structures.
\end{enumerate}
\end{thm}

\begin{proof} 
The equivalence between the groupoids at (a) and (b) follows from
Theorem \ref{thm:BundleLipschitz}. For any pin representation $\eta$,
the vector representation of $\L(\eta)$ factors through the
normalization morphism $\pi_0:\L(\eta)\simeq \R_{>0}\times
\cL(\eta)\rightarrow \cL(\eta)$ of Subsection
\ref{subsec:ReducedLipschitz}. The correspondence which takes a
Lipschitz structure $(Q,\tau)$ into the reduced Lipschitz structure
$(Q\times_{\pi_0}\cL(\eta),\tau_0)$ (where $\tau_0([q,a_0])\eqdef
\tau(q)\Ad_0(a_0)$ for all $q\in Q$ and $a_0\in \cL(\eta)$) induces an
equivalence of categories between the groupoid of real Lipschitz
structures relative to $\eta$ and the groupoid of reduced real
Lipschitz structures relative to $\eta$. This establishes the
equivalence between the groupoids at points (b) and (c). The
equivalence between the groupoids at (c) and (d) follows from Theorem
\ref{thm:cLpres} and Theorem \ref{thm:basicreps}.  The equivalence
between the groupoids at (d) and (e) follows from Proposition
\ref{prop:canmodcan}.  The remaining statements follow from the
definition of the canonical spinor group $\Lambda(V,h)$ and of its
basic representation.  \qed
\end{proof}

\noindent Theorem \ref{thm:canspinor} and the results of Section
\ref{sec:structures} immediately imply:

\begin{thm}
Let $\sigma\eqdef \sigma_{p,q}$.
\begin{enumerate}[1.]
\itemsep 0.0 em
\item In the normal simple case ($p-q\equiv_8 0,2)$, the following
  statements are equivalent:
\begin{enumerate}[(a)]
\itemsep 0.0em
\item There exists an elementary real pinor bundle on $(M,g)$
\item $(M,g)$ admits an elementary real Lipschitz structure
\item $(M,g)$ admits an untwisted $\Pin(V,h)$ structure
\item $(M,g)$ admits a twisted $\Pin(V,-\sigma h)$ structure
\item The following condition is satisfied:
\be
\w_2^+(M,g)+\w_2^-(M,g)+\w_1^{\sigma}(M,g)^2+\w_1^-(M,g)\w_1^+(M,g)=0~~.
\ee
\end{enumerate}
\item In the complex case, the following statements are equivalent:
\begin{enumerate}[(a)]
\itemsep 0.0em
\item There exists an elementary real pinor bundle on $(M,g)$
\item $(M,g)$ admits an elementary real Lipschitz structure
\item $(M,g)$ admits a $\Spin^o$ structure
\item There exists a principal $\O(2,\R)$-bundle $E$ over $M$ such that the following two conditions are satisfied:
\be
\w_1(M)=\w_1(E)
\ee
and: 
\begin{eqnarray}
\label{eq:spinoobs}
\w_2^+(M,g)+\w_2^-(M,g) &= & \w_2(E)+\w_1(E)(p\w_1^+(M,g)+q\w_1^-(M,g))\nonumber\\ && + \left[\delta_{\alpha,-1}+\frac{p(p+1)}{2}+\frac{q(q+1)}{2}\right]\w_1(E)^2~~~~~
\end{eqnarray}
\end{enumerate}
\item In the quaternionic simple case ($p-q\equiv_8 4,6)$, the
  following statements are equivalent:
\begin{enumerate}[(a)]
\itemsep 0.0em
\item There exists an elementary real pinor bundle on $(M,g)$
\item $(M,g)$ admits an elementary real Lipschitz structure
\item $(M,g)$ admits an untwisted $\Pin^q$ structure
\item $(M,-\sigma g)$ admits a twisted $\Pin^q$ structure
\item There exists a principal $\SO(3,\R)$-bundle $E$ over $M$ such that
  the following condition is satisfied:
\be
\w_2^+(M,g)+\w_2^-(M,g)+\w_1^{\sigma}(M,g)^2+\w_1^-(M,g)\w_1^+(M,g)=\w_2(E)~~.
\ee
\end{enumerate}
\item In the normal non-simple case ($p-q\equiv_8 1)$, the following
  statements are equivalent:
\begin{enumerate}[(a)]
\itemsep 0.0em
\item There exists an elementary real pinor bundle on $(M,g)$
\item $(M,g)$ admits an elementary real Lipschitz structure
\item $(M,g)$ admits a $\Spin$ structure
\item The following two conditions are satisfied:
\be
\w_1(M)=0~~\mathrm{and}~~\w_2^+(M,g)+\w_2^-(M,g)=0~~.
\ee
\end{enumerate}
\item In the quaternionic non-simple case ($p-q\equiv_8 5)$, the
  following statements are equivalent:
\begin{enumerate}[(a)]
\itemsep 0.0em
\item There exists an elementary real pinor bundle on $(M,g)$
\item $(M,g)$ admits an elementary real Lipschitz structure
\item $(M,g)$ admits a $\Spin^q$ structure
\item There exists a principal $\SO(3,\R)$-bundle $E$ over $M$ such that
  the following two conditions are satisfied:
\be
\w_1(M)=0~~\mathrm{and}~~\w_2^+(M,g)+\w_2^-(M,g)=\w_2(E)~~.
\ee
\end{enumerate}
\end{enumerate}
\end{thm}

\begin{remark}
Since $TM=T_+M\oplus T_-M$, we have $\w_2(M)=\w_2^+(M,g)+\w_2^-(M,g)+\w_1^+(M,g)\w_1^-(M,g)$. 
This allows one to express the conditions in the Theorem in various equivalent forms. 
For example, the conditions for existence of a spin structure can also be written as $\w_1(M)=0$ 
and $\w_2(M)=\w_1^-(M,g)^2$. 
\end{remark}

\section{Some remarks on the spin geometry of M-theory}
\label{sec:Mtheory}

In this section, we apply our results to a theory of physical
interest, deriving a no-go result regarding the interpretation of its
spinorial fields.  Eleven-dimensional supergravity is a physical
theory formulated on a connected and paracompact smooth 11-manifold
$M$, which involves a metric $g$ of Lorentzian signature, a four-form
field strength $F$ and a spin $3/2$ fermion called the gravitino. In
the standard local formulation, the gravitino is a real local field
$\psi_\mu^\alpha$ carrying a covector index $\mu$ and a spinorial
index $\alpha$, the latter running from $1$ to $32$. The theory admits
supersymmetry transformations parameterized by a real fermionic
supersymmetry generator $\chi^\alpha$. It is natural to ask how the
local formulas appearing in the construction of this theory found in
the Physics literature should be interpreted globally and what are the
minimal conditions on $(M,g)$ under which a consistent global
interpretation is possible. When approaching this question, one has to
consider the two possible choices of Lorentzian signature:
\begin{enumerate}
\item ``Mostly plus'' signature, i.e. $(p,q)=(10,1)$, which belongs to
  the normal non-simple case $p-q\equiv_8 1$. In this case, the
  smallest real representations of $\Cl_{10,1}$ are the two
  irreducible representations, which have dimension $32$ and are
  distinguished by the choice of signature $\epsilon\in \{-1,1\}$.
\item ``Mostly minus'' signature, i.e. $(p,q)=(1,10)$, which belongs
  to the complex case with $p-q\equiv_8 7$. In this case, the smallest
  real representation of $\Cl_{1,10}$ is the irreducible real
  representation, which has dimension 64. However, the smallest real
  representations of the even subalgebra $\Cl^\ev_{1,10}$ are the two
  real chiral (Majorana-Weyl) representations, both of which have
  dimension $32$ and are distinguished by the condition that the
  Clifford volume element maps to $\epsilon \id$ in the representation
  space, where $\epsilon\in \{-1,1\}$.
\end{enumerate}
When $M$ is non-compact, Lorentzian metrics on $M$ always exist. When
$M$ is compact, it is well-known that they exist iff the Euler
characteristic of $M$ vanishes. The results of this paper imply the
following.

\begin{prop}
\label{prop:Mspin}
Let $(M,g)$ be a Lorentzian manifold of dimension $d=11$ and ``mostly
plus'' signature $(p,q)=(10,1)$. Then $(M,g)$ admits an elementary
real pinor bundle if and {\em only if} it admits a spin structure,
i.e. if and only if $\w_1(M)=0$ and $\w_2^+(M,g)=0$, which is
equivalent with the conditions $\w_1(M)=0$ and
$\w_2(M)+\w_1^-(M,g)^2=0$. In that case, $S$ has rank $32$.
\end{prop}

\begin{proof}
We have $p-q\equiv_8 1$, which corresponds to the normal non-simple
case. Hence a reduced elementary Lipschitz structure on $(M,g)$ is a
spin structure and the topological condition for existence of such is
$\w_1(M)=\w_2^+(M,g)+\w_2^-(M,g)=0$. We have
$\w_2^-(M,g)=\w_2(T_-M)=0$ since $q=1$ and $T_-M$ is a real line
bundle, hence the second condition reduces to $\w_2^+(M,g)=0$. We also
have
$\w_2(M)=\w_2(T_+M)+\w_2(T_-M)+\w_1(T_+M)\w_1(T_-M)=\w_2(T_+M)+\w_1(T_-M)^2$,
where we used the condition $\w_1(T_+M)+\w_1(T_-M)=\w_1(TM)=0$. Hence
the second topological condition is equivalent modulo the first with
the condition $\w_2(M)+\w_1^-(M,g)^2=0$.  \qed
\end{proof}

\begin{remark}
Suppose that $(M,g)$ is time-orientable, i.e. it admits a
globally-defined timelike vector field $X$. Then we can take $T_-M$ to
be the real line bundle generated by $X$. Hence $T_-M$ is
topologically trivial and we have $\w_1^-(M,g)=0$. In this case, the
structure group of $P_{\O(V,h)}(M,g)$ reduces to $\O(10,\R)$ and the
topological conditions for existence of a spin structure reduce to
$\w_1(M)=0$ and $\w_2(M)=0$. Notice that $M$ admits a time-orientable
Lorentzian metric iff it admits an arbitrary Lorentzian metric.
\end{remark}
\noindent The result above implies: 

\

\noindent {\em Assume that eleven-dimensional supergravity is
  formulated on a smooth Lorentzian eleven-manifold
  $(M,g)$ of mostly plus signature. Then the supersymmetry generator $\chi$
  of the theory can be interpreted as a smooth global section of a bundle $S$
  of irreducible real Clifford modules if and {\em only if} $M$ is oriented
  and spin. In that case, the gravitino $\psi$ can be interpreted as a
  global section of the bundle $T^\ast M\otimes S$. Up to isomorphism,
  there are in fact two real pinor bundles $S$ which can be considered
  in that case (assuming that the spin structure is fixed), which are
  distinguished by the signature $\epsilon$. }

\

Since physics should be invariant under changing $g$ into $-g$, one
expects a similarly simple interpretation in mostly minus signature
when $(M,g)$ admits a spin structure. In order to provide an argument in favor of the equivalence between the spinor bundles in both signatures, when $M$ admits a spin structure, we need to further elaborate on the theory of $\Spin^o$ structures and hence we defer this analysis to Reference \cite{spino}.


The global interpretation of the local formulas of supergravity is
affected by cover ambiguities. This implies that one is not forced
apriori to interpret $\chi$ as a global section of a bundle $S$ of
irreducible Clifford modules.  In fact, eleven-dimensional
supergravity can be defined on unoriented eleven-manifolds, as
explained in references \cite{WittenFluxQuantization} and
\cite{WittenParityAnomaly}. In the approach of op. cit., one assumes a
$\Pin$ structure on $(M,g)$ and constructs the theory using the
modified Dirac operator \cite{TrautmanPin3}, even though a bundle of
irreducible real Clifford modules does not exist on $(M,g)$. In view
of this, the results above tell us {\em precisely} when it is possible to
globally construct the theory using a vector bundle $S$ endowed with {\em
  internal} Clifford multiplication. We show that this is 
possible exactly when $(M,g)$ admits a spin structure. 

We mention that the situation is considerably more involved 
when considering supergravity theories in lower dimensions (coupled to matter). 
As we show in forthcoming work, the results of this paper can be used to 
construct certain such theories without assuming that the corresponding 
space-time admits a $\Spin$ or $\Pin$ structure. 

\section{Relation to other work} 
\label{sec:relation}
Lipschitz groups for {\em complex} Clifford representations were
considered in \cite{FriedrichTrautman, Trautman,
  BobienskiTrautman}. As apparent from the present work, the
corresponding theory for real Clifford representations is considerably
more involved. $\Spin^q(V,h)$ structures in positive signature
$p=d,q=0$ were introduced in \cite{Nagase}. However, reference
\cite{Nagase} considers so-called ``quaternionic spinor bundles'',
i.e.  vector bundles associated to a $\Spin^q(V,h)$ structure through
a {\em quaternionic} representation
$\gamma_\H:\Cl(V,h)\rightarrow \End_\H(S)$ which is irreducible {\em
  over $\H$}. Here, $S$ is a right $\H$-module and the representation
is through $\H$-module endomorphisms. Any such representation is also
a real representation upon viewing $S$ as an $\R$-vector space by
restriction of scalars, but that real representation need not be
irreducible as a representation over $\R$ (since $S$ may admit
invariant $\R$-subspaces which are not $\H$-submodules). In fact, a
brief look at Table 1 on page 98 of \cite{Nagase} shows that the
$\H$-irreducible quaternionic Clifford representations listed there
are reducible over $\R$ except for $d\equiv_8 4$, which in our
terminology corresponds to a sub-case of the quaternionic simple
case. We stress that ``quaternionic spinor bundles'' based on
$\H$-irreducible quaternionic Clifford representations (as in
\cite{Nagase}) are not directly relevant for most physical theories,
where one is interested instead in elementary pinor bundles in the
sense of this paper (namely, vector bundles whose fibers are
$\R$-irreducible real Clifford representations). Similar remarks apply
to the work of \cite{Herrera1, Herrera2}, which extend the
constructions of \cite{Nagase} by replacing $\Sp(1)=\Spin(3)$ with a
higher spin group. We study $\Spin^o$ structures in detail in
reference \cite{spino}.

\begin{acknowledgements}
We thank E. Witten and A. Moroianu for correspondence. The work of
C. I. L. was supported by grant IBS-R003-S1. The work of C.S.S. is
supported by the ERC Starting Grant 259133 Observable String.
\end{acknowledgements}

\appendix

\section{Hyperbolic numbers}
\label{app:hyp}

Let $\D$ be the commutative algebra of hyperbolic (a.k.a. split
complex/double) numbers. Any element $z\in \D$ can be written uniquely
in the form $z=x+jy$, where $x,y\in \R$ and $j^2=+1$. The map
$\varphi:\D\rightarrow \R\times \R$ given by:
\be
\varphi(x+jy)=(x+y,x-y)~~
\ee
is a unital isomorphism of $\R$-algebras which satisfies
$\varphi(-1)=(-1,-1)$, $\varphi(j)=(1,-1)$ and
$\varphi(-j)=(-1,1)$. In particular, $\varphi$ induces an isomorphism
of groups $\D^\times \simeq \R^\times \times \R^\times$ and the
component maps $\varphi_\pm:\D\rightarrow \R$ given by
$\varphi_\pm(x+jy)=x\pm y$ are unital morphisms of $\R$-algebras which
satisfy $\varphi_\pm(j)=\pm 1$ and
$\varphi_\pm(\D^\times)=\R^\times$. The group $\D^\times$ has four
connected components:
\be
\D^{\epsilon_1,\epsilon_2}\eqdef \{z\in \D^\times|\sign(\varphi_+(z))=\epsilon_1~~,~~\sign(\varphi_-(z))=\epsilon_2\}~~,
\ee
where $\epsilon_1,\epsilon_2\in \{-1,1\}$. This gives a
$D_4$-grading of $\D^\times$ which corresponds to the
grading morphism $z\rightarrow
(\sign(\varphi_+(z)),\sign(\varphi_-(z)))\in \mG_2\times\mG_2\simeq
D_4$. We have
\be
(-1)\D^{\epsilon_1,\epsilon_2}=\D^{-\epsilon_1,-\epsilon_2}~~,~~j\D^{\epsilon_1,\epsilon_2}=\D^{\epsilon_1,-\epsilon_2}~~,~~(-j)\D^{\epsilon_1,\epsilon_2}=\D^{-\epsilon_1,\epsilon_2}
\ee
and hence $\D^\times \simeq \D^{++}\times D_4$. Moreover, we have
$\D^{++}\simeq \R_{>0}\times \R_{>0}$ and $1\in \D^{++}$, $-1\in \D^{--}$.  

Recall that {\em hyperbolic conjugation} is the unital $\R$-algebra
automorphism of $\D$ defined through:
\be
(x+jy)^\ast=x-jy~~\forall x,y\in \R~~
\ee
and that the {\em hyperbolic modulus} is the surjective map $\rM:\D\rightarrow \R$ given by: 
\be
\rM(x)=z^\ast z=x^2-y^2=\varphi_+(z)\varphi_-(z)~~\forall z=x+jy\in \D~~(x,y\in \R)~~.
\ee
We have $\rM(z_1z_2)=\rM(z_1) \rM(z_2)$, $\rM(1)=1$, $\rM(j)=-1$ and
$\rM(jz)=-\rM(z)$ A hyperbolic number $z=x+jy\in \D$ is invertible iff
$\rM(z)\neq 0$. The zero divisors of $\D$ are characterized by
$\rM(z)=0$ and correspond to the union of the lines $y=\pm x$. The
hyperbolic modulus induces a group morphism $\rM:\D^\times
\rightarrow \R^\times$.

The {\em group of unit hyperbolic numbers} is the subgroup of
$\D^\times$ given by:
\beqa
\U(\D)\eqdef \{z\in \D^\times| |\rM(z)|=1\}&=&\{z\in \D|\varphi_+(z)\varphi_-(z)\in \{-1,1\}\}=\nn\\
&=& \{x+jy\in \D|x^2-y^2\in \{-1,1\}\}~~,
\eeqa
and fits into the exact sequence: 
\be
1\longrightarrow \U(\D) \hookrightarrow \D^\times\stackrel{|\rM|}{\longrightarrow} \R_{>0}\longrightarrow 1~~.
\ee
This group has four connected components given by
$\U^{\epsilon_1,\epsilon_2}(\D)\eqdef \U(\D)\cap
\D^{\epsilon_1,\epsilon_2}$ and we have $\U(\D)\simeq \U^{++}(\D)\times
D_4$. The map:
\be
\R\ni \theta \rightarrow \cosh\theta+ j\sinh\theta\in \U^{++}(\D)~~
\ee
gives a group isomorphism $\U^{++}(\D)\simeq (\R,+)\simeq \R_{>0}$. Moreover, the map:
\be
\D^\times\ni z\rightarrow (|\rM(z)|,\frac{z}{\sqrt{|\rM(z)|}})\in \R_{>0}\times \U(\D)
\ee
is an isomorphism of groups and $\U(\D)$ is homotopy-equivalent with $\D^\times$.

\section{On internal and external Clifford multiplication}
\label{app:cliffmul}

In this appendix, we give a general construction of Clifford
multiplication for weakly faithful real Clifford representations and
certain types of associated vector bundles, explaining the difference
between the ``internal'' and ``external'' versions of the former.

Let $\gamma:\Cl(V,h)\rightarrow \End(S_0)$ be a weakly faithful real
Clifford representation, $H$ be a Lie group and $\theta,
\theta':H\rightarrow \Aut_\R(S_0)$ be linear representations of $H$ in
$S_0$. Let $\lambda: H\rightarrow \O(V,h)$ be a representation of $H$
through isometries of $(V,h)$ and $\rho:H\rightarrow \GL(V\otimes
S_0)\simeq \Aut_\R(V)\otimes \Aut_\R(S_0)$ denote the inner tensor product of
$\lambda$ and $\theta$:
\be
\rho(a)\eqdef \lambda(a)\otimes \theta(a)~~\forall a\in H~~.
\ee
Consider the linear map $\mu:V\otimes S_0\rightarrow S_0$ given by $\mu(v,\xi)\eqdef \gamma(v)\xi$. 

\begin{prop}
The map $\mu$ is a based morphism of representations from $\rho$ to $\theta'$ iff the following relation holds for all $a\in H$ and all $v\in V$:
\ben
\label{vind}
\gamma(\lambda(a)v)=\theta'(a)\circ \gamma(v)\circ \theta(a)^{-1}~~.
\een
In particular, $\theta$ and $\theta'$ determine $\lambda$ (since $\gamma$ is weakly faithful). 
\end{prop}

\begin{proof}
We have: 
\beqa
&& (L_\mu\circ \rho)(a)(v\otimes \xi)=(\mu\circ \rho(a))(v\otimes \xi)=(\gamma(\lambda(a)(v))\circ \theta(a))\xi~~\nn\\
&& (\theta'\circ R_\mu)(a)(v\otimes \xi)=(\theta'(a)\circ \mu)(v\otimes \xi)=(\theta'(a)\circ \gamma(v))\xi~~.
\eeqa
Recall that $\mu$ is a based morphism of representations iff $\mu\circ \rho(a)=\theta'(a)\circ \mu$ for all $a\in H$, i.e. 
iff $L_\mu \circ \rho=\theta'\circ R_\mu$. Using the relations above, this amounts to the condition:
\ben
\label{CMcond}
\gamma(\lambda(a)v)\circ \theta(a)=\theta'(a)\circ \gamma(v)~~\forall v\in V~~,
\een
which amounts to \eqref{vind}. \qed
\end{proof}

Relation \eqref{vind} says that the pair of representations $(\theta,\theta')$ of $H$ in $S_0$ 
``implements'' the pseudo-orthogonal representation $\lambda$ of $H$ in $V$. In particular, this relation 
requires $\theta'(a)\circ \gamma(V)\circ \theta(a)^{-1}\subset \gamma(V)$ for all $a\in H$, which is 
a non-trivial condition on $\theta$ and $\theta'$. 

Let $P_{\O(V,h)}$ be a principal $\O(V,h)$-bundle over a manifold $M$
and $P_H$ be a $\lambda$-reduction of $P$. Let $T$ be the vector
bundle associated to $P_{\O(V,h)}$ through the tautological
representation of $\O(V,g)$ on $V$ and $S=P_H\times_{\theta} S_0$,
$S'=P_H\times_{\theta'} S_0$ be the vector bundles associated to $P_H$
through the representations $\theta$ and $\theta'$ on $S_0$. Notice
that $T=P_H\times_\lambda V$, i.e. $T$ is associated to $P_H$ through
the representation $\lambda$. It follows that $T\otimes S$ is
associated to $P_H$ through the representation $\rho=\lambda\otimes
\theta$. Hence when
condition \eqref{CMcond} holds, the morphism of representations $\mu$
induces a morphism of vector bundles $\mu_\ast:T\otimes S\rightarrow S'$
which is called {\em (generalized) Clifford multiplication}. The
Clifford multiplication is called {\em internal} if
$\theta$ and $\theta'$ are equivalent representations of $H$. In this
case, the vector bundles $S$ and $S'$ are isomorphic and we can
identify them, thus obtaining a map $T\otimes
S\rightarrow S$ which induces a fiberwise $\Cl(T)$-module
structure on $S$. The Clifford multiplication is called
{\em external} if $\theta$ and $\theta'$ are inequivalent representations
of $H$. When $P_{\O(V,h)}$ is the orthogonal coframe bundle of a
pseudo-Riemannian manifold $(M,g)$, we can take $T=T M$ and have
$\mu_\ast: T M\otimes S\rightarrow S'$, which in the internal
case makes $S$ into a bundle of Clifford modules over
$(M,g)$. Consider the following applications of this construction, in those 
dimensions and signature where \eqref{vind} can be satisfied with the choices 
listed below:

\begin{enumerate}
\itemsep 0.0em
\item $\gamma$ is irreducible over $\R$, $H=\Spin(V,h)$,
  $\theta=\theta'=\gamma|_{\Spin(V,h)}$ and $\lambda=\Ad_0$, the
  vector representation of $\Spin(V,h)$. In this case, $S$ is a
  spinor bundle associated to the spin structure $P_H$ and $\mu_\ast:T
  M\otimes S\rightarrow S$ is the ordinary Clifford
  multiplication of $S$.
\item $\gamma$ is irreducible over $\R$, $H=\Pin(V,h)$,
  $\theta=\theta'=\gamma|_{\Pin(V,h)}$ and $\lambda=\Ad_0$, the {\em
  untwisted} vector representation of $\Pin(V,h)$. Then $P_H$ is an
  untwisted pin structure. We have $\mu_\ast:T M\otimes S
  \rightarrow S$, i.e. Clifford multiplication with a vector maps
  $S$ into itself.
\item $\gamma$ is irreducible over $\R$, $H=\Pin(V,h)$,
  $\theta=\gamma|_{\Pin(V,h)}$, $\theta'=\theta\circ \pi|_{\Pin(V,h)}$
  and $\lambda=\tAd_0$, where $\pi$ is the parity involution of $\Cl(V,h)$
  (recall that $\pi(\Pin(V,h))=\Pin(V,h)$). In this case, $P_H$ is a
  twisted pin structure and the bundles $S$ and $S'$ are
  generally non-isomorphic. In this case, we obtain a generally
  external Clifford multiplication $\mu_\ast:T^\ast M\otimes
  S\rightarrow S'$.
\item $H=\cL$ (the reduced Lipschitz group) and $\theta=\theta'$
  coincide with the tautological representation of $\cL$ on $S$.  Then
  \eqref{vind} can always be satisfied with $\lambda=\Ad_0$,
  the vector representation of $\cL$. In this case, $P_H$ is a
  Lipschitz structure while $S$ is the pinor bundle associated to
  $P_H$. We have $\mu_\ast:T M\otimes S\rightarrow S$, i.e. the
  Clifford multiplication is internal.
\end{enumerate}
In the case of twisted pin structures, the usual definition of the
Dirac operator gives an operator which maps $S$ into $S'$, an
inconvenient feature which (in the case of complex pinor bundles) was
noticed and discussed in \cite{TrautmanPin1, TrautmanPin2,
  TrautmanPin3}. Notice that Lipschitz structures always lead to
well-defined internal Clifford multiplication on $S$. In fact,
Lipschitz structures are {\em designed} to make this happen.

\section{Twisted automorphisms of $\S$-modules and $\S$-valued pairings}
\label{app:twisted}

\subsection{Twisted morphisms of modules}

\noindent Let $\S,\S'$ be unital associative $\R$-algebras. 

\begin{definition}
A {\em twisted morphism} from a left $\S$-module $A$ to a left
$\S'$-module $A'$ is a pair $(\varphi_0,\varphi)$, where $\varphi_0\in
\Hom_\Alg(\S,\S')$ is a unital morphism of $\R$-algebras and
$\varphi\in \Hom_\R(A,A')$ is an $\R$-linear map, such that the
following condition is satisfied:
\be
\varphi(s a)=\varphi_0(s)\varphi(a)~~\forall s\in \S~~\mathrm{and}~~a\in A~~.
\ee
\end{definition} 

\

\noindent Left modules over unital associative $\R$-algebras and twisted morphisms
form a category denoted $\TwMod$. This fibers over the category
$\Alg$ of unital associative $\R$-algebras through the forgetful functor which
takes a left $\S$-module $A$ to $\S$ and a twisted morphism
$(\varphi_0,\varphi)$ to $\varphi_0$. The fiber over $\S$ is the
usual category $\Mod_\S$ of left $\S$-modules and ordinary $\S$-module
morphisms (those twisted module morphisms $(\varphi_0,\varphi)$ for which
$\varphi_0=\id_\S$).

\subsection{Twisted automorphisms}

Let $\Aut_\S^\tw(A)$ denote the group of twisted automorphisms of the
left $\S$-module $A$ and $\Aut_\S(A)$ denote the group of usual
$\S$-module automorphisms. We have the following obvious result:

\begin{prop}
\label{prop:twseq}
There exists an exact sequence of groups: 
\ben
\label{twseq}
1\longrightarrow \Aut_\S(A)\hookrightarrow \Aut_\S^\tw(A)\stackrel{F}{\longrightarrow} \Aut_\Alg(\S)~~,
\een
where $F(\varphi_0,\varphi)=\varphi_0$. 
\end{prop}

\subsection{$\S$-valued symmetric pairings}

\begin{definition}
An {\em $\S$-valued symmetric pairing} on the left $\S$-module $A$ is an
$\R$-bilinear symmetric map $\fp:A\times A\rightarrow \S$.  The {\em
  image algebra} $I(\fp)$ determined by $\fp$ is the subalgebra of $\S$
generated by the set $\fp(A\times A)$ over $\R$.
\end{definition}

\noindent Notice that $\fp$ is uniquely determined by its {\em
  diagonal quadratic form} $\fp_d:A\rightarrow \R$, which is defined
through $\fp_d(a)\eqdef \fp(a,a)$. Indeed, we have the polarization
identity:
\be
\fp(a_1,a_2)=\frac{1}{2}(\fp_d(a_1+a_2)-\fp_d(a_1)-\fp_d(a_2))~~\forall a_1,a_2\in A~~.
\ee
We have $\fp_d(\lambda a)=\lambda^2 \fp_d(a)$ for all $\lambda\in
\R$ and $a\in A$ as well as the parallelogram identity:
\be
\fp_d(a_1+a_2)+\fp_d(a_1-a_2)=2\left[\fp_d(a_1)+\fp_d(a_2)\right]~~(a_1,a_2\in A)~~.
\ee

\begin{definition}
\label{def:porthogonal}
Let $\fp$ be an $\S$-valued symmetric pairing on the left $\S$-module
$A$. A twisted automorphism $(\varphi_0,\varphi)\in \Aut_\S^\tw(A)$ is
called {\em $\fp$-orthogonal} if the following condition is
satisfied:
\ben
\label{OrtRho}
\fp(\varphi(a_1),\varphi(a_2))=\varphi_0(\fp(a_1,a_2))~~\forall a_1, a_2\in A~~.
\een  
\end{definition}

\noindent Notice that the restriction $\varphi_0|_{I(\fp)}$ is
uniquely determined by $\varphi$. Using the polarization identity,
condition \eqref{OrtRho} is equivalent with:
\ben
\label{OrtQuad}
\fp_d(\varphi(a))=\varphi_0(\fp_d(a))~~\forall a\in A~~.
\een 
Let $\Aut_\S^\tw(A,\fp)$ denote the group of $\fp$-orthogonal
twisted automorphisms of $A$ and $\Aut_\S(A,\fp)$ denote the subgroup
of $\fp$-orthogonal module automorphisms (those $\fp$-orthogonal
twisted automorphism with $\varphi_0=\id_\S$).

\begin{definition}
The {\em twist group} of $\fp$ is the following subgroup of $\Aut_\Alg(\S)$: 
\be
G_\fp\eqdef F(\Aut_\S^\tw(A,\fp))\subset \Aut_\Alg(\S)~~.
\ee
\end{definition}

\noindent The sequence \eqref{twseq} induces a short exact sequence: 
\ben
\label{twseqrho}
1\longrightarrow \Aut_\S(A,\fp)\hookrightarrow \Aut_\S^\tw(A,\fp)\stackrel{F}{\longrightarrow} G_\fp\longrightarrow 1~~.
\een

Any $\R$-algebra automorphism of $\S$ restricts to a group
automorphism of $\S^\times$. This gives a morphism of groups
$Res:\Aut_\Alg(\S)\rightarrow \Aut_{\Gp}(\S^\times)=\Aut_{\Gp}((\S^\times)^\op)$. Let $F_0:\eqdef
Res\circ F:\Aut_\S^\tw(A)\rightarrow \Aut_\Gp(\S^\times)$ denote the
morphism of groups induced by the map $F$ of \eqref{twseq}.

\subsection{The case of rank one free $\S$-modules}

\begin{prop}
\label{prop:RankOneTw}
Let  $A$ be a free left $\S$-module of rank one and $u$ be a basis of $A$ over $\S$. 
Then $\Aut_\S(A)\simeq (\S^\times)^\op$, the sequence \eqref{twseq} gives a split short exact sequence: 
\ben
\label{twseq1}
1\longrightarrow (\S^\times)^\op \rightarrow \Aut_\S^\tw(A)\stackrel{F}{\longrightarrow} \Aut_\Alg(\S)\longrightarrow 1~~
\een
and there exists an isomorphism of groups:
\ben
\label{RankOne}
\Aut^\tw_\S(A)\simeq (\S^\times)^\op \rtimes_{Res} \Aut_\Alg(\S)~~.
\een
\end{prop}

\proof 
Since $A$ is a free rank one left $\S$-module with basis $u$, we have
$A=\S u$ and any $x\in A$ can be written as $x=s u$ for some
uniquely-determined $s$. Thus any $\R$-linear map
$\varphi\in \End_\R(A)$ defines an element $\sigma_u(\varphi)\in \S$
through the relation $\varphi(u)=\sigma_u(\varphi)u$. Conversely, any
$s\in \S$ defines an $\R$-linear operator $x\rightarrow sx$ acting in
$A$, which takes $u$ into $su$. This gives a surjective $\R$-linear map:
\ben
\label{surj}
\sigma_u:\End_\R(A)\longrightarrow \S
\een
which satisfies $\sigma_u(\id_A)=1$. For
$(\varphi_0, \varphi), (\varphi'_0,\varphi')\in \Aut_\S^\tw(A)$, we
have:
\ben
\label{sigmarel}
\sigma_u(\varphi\circ \varphi')=\varphi_0(\sigma_u(\varphi'))\sigma_u(\varphi)=\sigma_u(\varphi)\cdot^\op \varphi_0(\sigma_u(\varphi'))~~,
\een
where $\cdot^\op$ denotes multiplication in $\S^\op$. It is easy to see that \eqref{sigmarel} implies the
inclusion $\sigma_u(\Aut_\R^\tw(A))\subset \S^\times$ as well as the
relation:
\be
\sigma_u(\varphi)^{-1}=\varphi_0(\sigma_u(\varphi^{-1}))~~\mathrm{for}~~(\varphi_0,\varphi)\in \Aut_\S^\tw(A)~~.
\ee
When $\varphi\in \Aut_\S(A)$ is an untwisted automorphism of $A$ (thus
$\varphi_0=\id_\S$), we have
$\varphi(su)=s\varphi(u)=s\sigma_u(\varphi)u$ for all $s\in \S$, which
shows that $\varphi$ is uniquely-determined by $\sigma_u(\varphi)\in
\S^\times$. This implies that the map $\sigma_u$ of \eqref{surj} restricts to a bijection between
$\Aut_\S(A)$ and $\S^\times$ while \eqref{sigmarel} with
$\varphi_0=\id_\S$ shows that this bijection is an isomorphism
of groups between $\Aut_\S(A)$ and $(\S^\op)^\times=(\S^\times)^\op$.

Relation \eqref{sigmarel} also implies that the map $T_u:\Aut^\tw_\S(A)\rightarrow  (\S^\times)^\op \rtimes_{Res} \Aut_\Alg(\S)$ given by: 
\be
T_u(\varphi_0,\varphi)\eqdef (\sigma_u(\varphi),\varphi_0)
\ee
is a morphism of groups. This map is bijective since, for any pair
$(\sigma,\varphi_0)\in \S^\times\times \Aut_\Alg(\S)$, there exists a
unique $\varphi\in \Aut_\R(A)$ such that $(\varphi_0,\varphi)\in
\Aut_\S^\tw(A)$ and $\sigma_u(\varphi)=\sigma$, namely:
\be
\varphi(su)=\varphi_0(s)\sigma u~~(s\in \S)~~;
\ee
we have $T_u^{-1}(\sigma,\varphi_0)=(\varphi_0,\varphi)$.
Taking $\sigma=1$ shown that the map $F$ of \eqref{twseq} is
surjective. In view of \eqref{twseq}, its kernel coincides with $\Aut_\S(A)$, 
which (as show above) is isomorphic with $(\S^\times)^\op$.

In fact, the morphism $G_u:\Aut_\Alg(\S)\rightarrow
\Aut_\S^\tw(A)$ given by $G_u(\varphi_0)\eqdef T_u^{-1}(1,\varphi_0)$
is a section of $F$:
\be
F\circ G_u=\id_{\Aut_\Alg(\S)}~~
\ee
and hence splits the sequence \eqref{twseq1}. The
semidirect product presentation \eqref{RankOne} is the one induced by
the splitting morphism $G_u$. \qed

\

\noindent For the following, we fix a basis $u\in
A$. Identifying $A$ with $\S$ through the isomorphism $\S\ni
s\rightarrow su\in A$, an $\S$-valued $\R$-bilinear symmetric form
$\fp$ on $A$ corresponds to an $\S$-valued $\R$-bilinear symmetric
form:
\ben
\label{rho_u}
\fp_u:\S\times \S\rightarrow \S
\een
on $\S$, namely:
\be
\fp_u(s_1,s_2)\eqdef \fp(s_1u, s_2 u)~~.
\ee
Using relation \eqref{RankOne}, a twisted $\fp$-orthogonal automorphism
$(\varphi_0,\varphi)\in \Aut_\S^\tw(A,\fp)$ corresponds to a pair
$(\sigma_u, \varphi_0)\eqdef (\sigma_u(\varphi),\varphi_0)\in
\S^\times \times \Aut_\Alg(\S)$ which satisfies:
\ben
\label{rhoconstraint}
\fp_u(\varphi_0(s_1)\sigma_u,\varphi_0(s_2) \sigma_u)=\varphi_0(\fp_u(s_1,s_2))~~\forall s_1,s_2\in \S~~.
\een
Accordingly, the group $\Aut_\S^\tw(A,\fp)$ identifies with the
subgroup of $(\S^\times)^\op \rtimes \Aut_\Alg(\S)$ consisting of all
such pairs. Obviously, this is a subgroup of $(\S^\times)^\op \times G_\fp$.

\phantomsection
\bibliographystyle{JHEP}

\end{document}